\documentclass[envcountsect,smallextended]{svjour3}       
\smartqed
\usepackage{graphicx}
\usepackage{arevmath}
\usepackage[utf8]{inputenc}
\usepackage[T1]{fontenc}
\usepackage{lmodern}
\usepackage{amsmath}
\usepackage{amssymb}
\usepackage{marginnote}
\usepackage{xcolor}
\usepackage{fancyhdr}
\usepackage{esint}
\usepackage{tikz} 
\usetikzlibrary{fadings}
\usetikzlibrary{patterns}
\usetikzlibrary{shadows.blur}
\usetikzlibrary{shapes}
\usepackage[arrow, matrix, curve]{xy}
\usepackage{hyperref}
\usetikzlibrary{fit,positioning} 
\usepackage{geometry}
\usepackage{pifont}
\usepackage{nccmath}
\usepackage{tocloft}
\usepackage{setspace}
\usepackage{graphicx}
\usepackage{calligra}
\usepackage[english]{babel}
\usepackage{amssymb}
\usepackage{scalerel}
\usepackage{blindtext}
\usepackage{upgreek}
\usepackage[justification=centering,labelfont=bf,belowskip=-10pt,aboveskip=3pt,font=small]{caption}

\usepackage[rose,novargreek,operator]{paper_diening}

\newcommand{\vertiii}[1]{{\vert\kern-0.25ex\vert\kern-0.25ex\vert #1 
		\vert\kern-0.25ex\vert\kern-0.25ex\vert}}
	
	\usepackage{color}
	\definecolor{Darkgreen}{rgb}{0,0.4,0}
	\usepackage{hyperref}
	\hypersetup{%
		pdfborder = {0 0 0},
		colorlinks,
		citecolor=blue,
		filecolor=black,
		linkcolor=blue,
		urlcolor=cyan!50!black!90
	}
\DeclareMathAlphabet\bscal{OMS}{cmsy}{b}{n}
 \DeclareMathAlphabet\mathbfscr{OMS}{mdugm}{b}{n}
 
 \newenvironment{psmallmatrix}
 {\left(\begin{smallmatrix}}
 	{\end{smallmatrix}\right)}

\spnewtheorem{defn}[equation]{Definition}{\bfseries}{\upshape}
\spnewtheorem{prop}[equation]{Proposition}{\bfseries}{\upshape}
\spnewtheorem{thm}[equation]{Theorem}{\bfseries}{\upshape}
\spnewtheorem{cor}[equation]{Corollary}{\bfseries}{\upshape}
\spnewtheorem{rmk}[equation]{Remark}{\bfseries}{\upshape}
\spnewtheorem{lem}[equation]{Lemma}{\bfseries}{\upshape}
\spnewtheorem{expl}[equation]{Example}{\bfseries}{\upshape}

\numberwithin{equation}{section} 

\DeclareMathSymbol{\subsetneq}{\mathrel}{AMSb}{"28}

\begin{document}
	
		\title{Variable exponent Bochner--Lebesgue spaces with symmetric gradient structure}
	
	\author{A. Kaltenbach\and M. \Ruzicka}

	\institute{A. Kaltenbach \at
		Institute of Applied Mathematics, Albert-Ludwigs-University Freiburg, Ernst-Zermelo-Straße 1, 79104 Freiburg,\\
		\email{alex.kaltenbach@mathematik.uni-freiburg.de}  
		  \\
		M. R\r{u}\v{z}i\v{c}ka \at Institute of Applied Mathematics,
		Albert-Ludwigs-University Freiburg,
		Ernst-Zermelo-Straße 1, 79104 Freiburg,\\
		\email{rose@mathematik.uni-freiburg.de} 
	}
	
	\date{Received: date / Accepted: date}

	\maketitle
	
	\begin{abstract}
		We introduce function spaces for the treatment of non-linear parabolic 
		equations with variable $\log$-Hölder continuous exponents, which only incorporate 
		information of the symmetric part of a gradient. As an analogue 
		of Korn's inequality for these functions spaces is not available,
		 the construction of an appropriate smoothing method proves 
		 itself to be difficult. To this end, we prove a point-wise 
		 Poincar\'e inequality near the boundary of a  bounded Lipschitz 
		 domain involving only the symmetric gradient. Using 
		 this inequality, we construct a smoothing operator with 
		 convenient properties. In particular, this smoothing operator leads 
		 to several density results, and therefore to a
                 generalized formula of integration by parts with
                 respect to time. Using this formula and the theory of maximal monotone operators, we prove an abstract existence result.
	\end{abstract}

	\section{Introduction}
	
	Let $\Omega\subseteq \setR^d$, $d\ge 2$, be a bounded 
	Lipschitz domain, $I:=\left(0,T\right)$, $T<\infty$, a 
	bounded time interval, $Q_T:=I\times \Omega$ a cylinder and $\Gamma_T:=I\times\pa\Omega$. 
	We consider as a model problem the initial-boundary value problem
	\begin{align}
	\begin{split}
	\begin{alignedat}{2}
		\pa_t \bsu
		-\divo(\bfS(\cdot,\cdot,\bfvarepsilon(\bsu)))
		+\bfb(\cdot,\cdot,\bsu)&
		=\bsf-\divo(\bsF)&&\quad\text{ in }Q_T,\\
		\bsu&=0&&\quad\text{ on }\Gamma_T,\\
		\bsu(0)&=\bfu_0&&\quad\text{ in }\Omega.
		\end{alignedat}
			\end{split}\label{eq:model}
	\end{align}
	Here, $\bsf:Q_T\to \setR^d$ is a given vector field, 
	$\bsF:Q_T\to \mathbb{M}^{d\times d}_{\sym}$ 
	a given tensor field\footnote{$\mathbb{M}^{d\times d}_{\sym}$ is the vector 
		space of all symmetric $d\times d$-matrices $\bfA=(A_{ij})_{i,j=1,...,d}$. 
		We equip $\mathbb{M}^{d\times d}_{\sym}$ with the scalar product $\bfA:\bfB:=\sum_{i,j=1}^{d}{A_{ij}B_{ij}}$ 
		and the norm $\vert\bfA\vert:=(\bfA:\bfB)^{\frac{1}{2}}$. 
		By $\bfa\cdot\bfb$ we denote the usual scalar product in $\setR^d$ and 
		by $\vert \bfa\vert$ we denote the Euclidean norm.}, 
	$\bfu_0:\Omega\to \setR^d$ an initial condition, and 
	$\bfvarepsilon(\bfu):=\frac{1}{2}(\nb \bfu+\nb \bfu^\top)$ 
	denotes the symmetric part of the gradient. Moreover, let $p:Q_T\to\left[1,+\infty\right)$
	 be a globally $\log$-Hölder continuous exponent with
	\begin{align*}
		1< p^-:=\inf_{(t,x)^\top\in Q_T}{p(t,x)}\leq p^+:=\sup_{(t,x)^\top\in Q_T}{p(t,x)}<\infty.
	\end{align*}
	The mapping $\bfS:Q_T\times \mathbb{M}^{d\times d}_{\sym}\to 
	\mathbb{M}^{d\times d}_{\sym}$ is supposed to possess
        \textit{$(p(\cdot,\cdot),\delta)$-structure}, i.e., for
        some exponent $p(\cdot,
        \cdot)$ and some $\delta\ge 0$ the following properties are satisfied:
	\begin{description}[{(S.4)}]
		\item[\textbf{(S.1)}]\hypertarget{S.1}{} $\bfS:Q_T\times \mathbb{M}^{d\times d}_{\sym}\to 
		\mathbb{M}^{d\times d}_{\sym}$ is a Carath\'eodory mapping\footnote{$\bfS(\cdot,\cdot,\bfA):Q_T\to \mathbb{M}^{d\times d}_{\sym}$ is Lebesgue measurable for every $\bfA\in \mathbb{M}^{d\times d}_{\sym}$ and $\bfS(t,x,\cdot):\mathbb{M}^{d\times d}_{\sym}\to \mathbb{M}^{d\times d}_{\sym}$ is continuous for almost every $(t,x)^\top\in Q_T$.}.
              \item[\textbf{(S.2)}]\hypertarget{S.2}{} $\vert\bfS(t,x,\bfA)\vert \leq \alpha(\delta+\vert \bfA\vert)^{p(t,x)-2}\vert\bfA\vert+\beta(t,x)$ 
		for all $\bfA\in \mathbb{M}^{d\times d}_{\sym}$, a.e. $(t,x)^\top\in Q_T$\newline
		$(\alpha\ge 1,\;\beta\in L^{p'(\cdot,\cdot)}(Q_T,\setR_{\ge 0}))$.
		\item[\textbf{(S.3)}]\hypertarget{S.3}{} $\bfS(t,x, \bfA):\bfA\ge 
		c_0(\delta+\vert \bfA\vert)^{p(t,x)-2}\vert\bfA\vert^2-c_1(t,x)$ 
		for all $\bfA\in \mathbb{M}^{d\times d}_{\sym}$,  a.e. $(t,x)^\top\in Q_T$\newline
		${(c_0>0,\;c_1\in L^1(Q_T,\setR_{\ge 0}))}$.
		\item[\textbf{(S.4)}]\hypertarget{S.4}{} $\big(\bfS(t,x,\bfA)-\bfS(t,x,\bfB)\big):
		\big(\bfA-\bfB\big)\ge 0$ for all 
		$\bfA,\bfB\in \mathbb{M}^{d\times d}_{\sym}$, a.e. $(t,x)^\top\in Q_T$.
	\end{description}
	The mapping $\bfb:Q_T\times \setR^d\to \setR^d$ is supposed
	to satisfy:
	\begin{description}[{(B.3)}]
		\item[\textbf{(B.1)}]\hypertarget{B.1}{} $\bfb:Q_T\times \setR^d\to \setR^d$ is a 
		Carath\'eodory mapping.
		\item[\textbf{(B.2)}]\hypertarget{B.2}{} $\vert\bfb(t,x,\bfa)\vert \leq 
		\gamma(1+\vert \bfa\vert)^r+\eta(t,x)$ for all 
		$\bfa\in \setR^d$, a.e. $(t,x)^\top\in Q_T$ \newline$\big(\gamma\ge 0,\;r<\frac{(p^-)_*}{(p^-)'},\;\eta\in L^{(p^-)'}(Q_T,\setR_{\ge 0})\big)$.\footnote{For $p\in\left(1,\infty\right)$ we define $p_*\in \left(1,\infty\right)$ through $p_*:=p\frac{d+2}{d}$ if $p<d$ and $p_*:=p+2$ if $p\ge d$.}
		\item[\textbf{(B.3)}]\hypertarget{B.3}{} $\bfb(t,x,\bfa)\cdot\bfa\ge -c_2(t,x)$ 
		for all $\bfa\in \setR^d$, a.e. $(t,x)^\top\in Q_T$ $(c_2\in L^1(Q_T,\setR_{\ge 0}))$.
	\end{description}
	A system like \eqref{eq:model} occurs in 
        nonlinear elasticity \cite{Z86} or in image restoration
        \cite{LLP10,HHLT13,CLR06}, or can be seen as a simplification
        of the unsteady $p(\cdot,\cdot)$-Stokes and
        $p(\cdot,\cdot)$-Navier--Stokes equations, as it does not
        contain an incompressibility constraint, but the nonlinear
        elliptic operator is governed
        by the symmetric gradient. In fact, a long term goal of this research is to
        prove the existence of weak solutions of the unsteady $p(\cdot,\cdot)$-Stokes and
        $p(\cdot,\cdot)$-Navier--Stokes equations.
        Similar equations also appear in
        Zhikov's model for the termistor problem in \cite{Z08b}, or in
        the investigation of variational integrals with non-standard
        growth in \cite{Mar91,AM00,AM01,AM02}.
        Major progress was also made by Diening and R\r u\v zi\v cka in \cite{DR03}, at least in the analysis of the steady counterpart of \eqref{eq:model}, i.e., for a fixed time slice $t_0\in I$ the boundary value problem
	\begin{align*}
	\begin{alignedat}{2}-\divo(\bfS(t_0,\cdot,\bfvarepsilon(\bfu)))
	+\bfb(t_0,\cdot,\bfu)&
	=\bsf(t_0,\cdot)-\divo(\bfF(t_0,\cdot))&&\quad\text{ in }\Omega,\\
	\bfu&=0&&\quad\text{ in }\pa\Omega,
	\end{alignedat}
	\end{align*}
	is considered and the validity of Korn's inequality in
        variable exponent Sobolev spaces is proved (cf.~Proposition \ref{2.4}). In fact, due to 
	the validity of Korn's inequality in variable exponent Sobolev spaces, the solvability of \eqref{eq:model}
	is a straightforward application of the main theorem on pseudo-monotone operators, tracing back to Brezis 
	\cite{Bre68}, which states the following.\footnote{All
		notion are defined in Section~\ref{sec:2}.} 
	
	\begin{thm}\label{1.1}
		Let $X$ be a reflexive Banach space 
		and $A:X\to X^*$ a bounded, pseudo-monotone 
		and coercive operator. Then, $R(A)=X^*$.
	\end{thm} 

	In fact, for every $t\in I$ the operators
	$A(t):=S(t)+B(t):X^p_{\bfvarepsilon}(t)\to X^p_{\bfvarepsilon}(t)^*$, where $X^p_{\bfvarepsilon}(t)$ is the closure of $C_0^\infty(\Omega)^d$ with respect to $\|\cdot\|_{L^{p(t,\cdot)}(\Omega)^d}+\|\bfvarepsilon(\cdot)\|_{L^{p(t,\cdot)}(\Omega)^{d\times d}}$, given via 
	$\langle S(t)\bfu,\bfv\rangle_{X^p_{\bfvarepsilon}(t)}:=
	\int_\Omega{\bfS(t,\cdot,\bfvarepsilon(\bfu))
		:\bfvarepsilon(\bfv)\,dx}$ and 
	$\langle B(t)\bfu,\bfv\rangle_{X^p_{\bfvarepsilon}(t)}
	:=\int_\Omega{\bfb(t,\cdot,\bfu)\cdot\bfv\,dx}$ 
	for every $\bfu,\bfv\in X^p_{\bfvarepsilon}(t)$, are bounded, 
	coercive and pseudo-monotone.
	
	 A time-dependent analogue of 
	 Brezis' contribution is the following (cf.~\cite{Lio69,Zei90B,Ru04,Rou05}).
	\begin{thm}\label{1.2}
		Let $(V,H,j)$ be an evolution triple\footnote{$V$ 
			is a reflexive Banach space, 
			$H$ a Hilbert space 
			and $j:V\to H$ a dense embedding.}, 
		$I=\left(0,T\right)$, ${T<\infty}$, and ${p\in \left(1,\infty\right)}$. 
		Moreover, let $\bscal{A}:W^{1,p,p'}(I,V,V^*)\subseteq
                L^p(I,V)\to L^p(I,V)^*$\footnote{Here,
                  $W^{1,p,p'}(I,V,V^*)$ denotes the space of all
                  functions $\bsu\in L^p(I,V)$ which possess a
                  distributional derivative $\frac{d\bsu}{dt}$ in $L^p(I,V)^*$ (cf.~Subsection~\ref{subsec:2.2}).} be
		bounded, coercive and $\frac{d}{dt}$-pseudo-monotone, i.e., from
		\begin{align*}
		\bsu_n\overset{n\to\infty}{\weakto}
		\bsu\quad\text{ in }W^{1,p,p'}(I,V,V^*),
		\\
		\limsup_{n\to\infty}\langle\bscal{A}\bsu_n,
		\bsu_n-\bsu\rangle_{L^p(I,V)}\leq
		0,
		\end{align*}
		it follows that $ \langle\bscal{A}\bsu,\bsu
		-\bsv\rangle_{L^p(I,V)}
		\leq\liminf_{n\to\infty}{\langle\bscal{A}\bsu_n,\bsu_n
			-\bsv\rangle_{L^p(I,V)}}$ for every $\bsv\in L^p(I,V)$.  
		Then, for arbitrary 
		$\bsu_0\in H$ and 
	$\bsf\in L^p(I,V)^*$ there exists a solution 
		$\bsu\in W^{1,p,p'}(I,V,V^*)$ of the initial value problem
		\begin{align*}
		\begin{alignedat}{2}
		\frac{d\bsu}{dt}+\bscal{A}\bsv&
		=\bsf&&\quad\text{ in }L^p(I,V)^*,\\\bsu(0)&
		=\bsu_0&&\quad\text{ in }H.
		\end{alignedat}
		\end{align*}
	\end{thm}
	This result is essentially a consequence of the main theorem on
	pseudo-monotone perturbations of maximal monotone mappings, stemming
	from Browder \cite{Bro68} and Brezis \cite{Bre68} (see also Proposition~\ref{2.1a}). In doing so, one
	interprets the restricted time derivative
	${\frac{d}{dt}\!:\!W_0\!\subseteq \!L^p(I,V)\!\to \!(L^p(I,V))^*}$, where $W_0:=\{\bsu\in W^{1,p,p'}(I,V,V^*) \fdg
	\bsu(0)=\boldsymbol{0}\text{ in }H\}$, as a
	linear, densely defined and maximal monotone mapping.
	Note that the maximal monotonicity of $\frac{d}{dt}$ can basically
	be traced back to a generalized formula of integration by parts. For details we refer to \cite{Lio69,GGZ74,Zei90B,Sho97,Rou05,Ru04}.
	
	If the variable exponent $p\in \mathcal{P}^{\log}(Q_T)$ is a constant, 
	i.e., $p(\cdot,\cdot)\equiv p\in \left(1,\infty\right)$, 
	it is well-known that the natural energy spaces for the 
	treatment of the initial value problem \eqref{eq:model} 
	are the Bochner--Lebesgue $L^p(I,X^p)$, where $X^p:=W^{1,p}_0(\Omega)^d$, 
	and the Bochner--Sobolev space $W^{1,p,p'}(I,X^p,(X^p)^*)$, and therefore
	 \eqref{eq:model} falls into the framework of Theorem \ref{1.2}. However,
	 if the elliptic part $\bfS$ has a non-constant variable
	 exponent structure, the situation changes substantially. 
	 In this case the coercivity condition (\hyperlink{S.3}{S.3})
	 suggests the energy space
	 \begin{align*}
	 	\bscal{X}^{2,p}_{\bfvarepsilon}(Q_T)
	 	:=\{\bsu\in L^2(Q_T)^d\fdg 
	 	\bfvarepsilon(\bsu)\in L^{p(\cdot,\cdot)}(Q_T,\mathbb{M}^{d\times d}_{\sym}),
	 	\;\bsu(t)\in X^p_{\bfvarepsilon}(t)\text{ for a.e. }t\in I\}.
	 \end{align*}
	Unfortunately, $\bscal{X}^{2,p}_{\bfvarepsilon}(Q_T)$ does not fit into the Bochner--Lebesgue 
	framework required in Theorem~\ref{1.2}. This predicament 
	is not new and has already been considered in \cite{AS09,naegele-dipl,DNR12}. 
	But these results only treated the case where the monotone part $\bfS$ 
	is depending on the full gradient, which in turn suggested the energy space 
	\begin{align*}
	\bscal{X}^{2,p}_{\nb}(Q_T)
	:=\{\bsu\in L^2(Q_T)^d\fdg 
	\nb\bsu\in L^{p(\cdot,\cdot)}(Q_T,\mathbb{M}_{\sym}^{d\times d}),
	\;\bsu(t)\in X^p_{\nb}(t)\text{ for a.e. }t\in I\}, 
	\end{align*}
	where $X^p_{\nb}(t)$, $t\in I$, are the closures of $C_0^\infty(\Omega)^d$ with respect to $\|\cdot\|_{L^{p(t,\cdot)}(\Omega)^d}+\|\nb\cdot\|_{L^{p(t,\cdot)}(\Omega)^{d\times d}}$.
	In the steady case, due to Korn's inequality for variable exponent Sobolev spaces (cf.~Proposition~\ref{2.4}), the spaces 
	$X^p_{\bfvarepsilon}(t)$ and $X^p_{\nb}(t)$
	 coincide for every $t\in I$. 
	Apart from that, if $p\in\left(1,\infty\right)$ is a constant, 
	the steady Korn's inequality transfers to the Bochner--Lebesgue level, 
	and thus implies the algebraic identity $\bscal{X}^{2,p}_{\bfvarepsilon}(Q_T)
	=\bscal{X}^{2,p}_{\nb}(Q_T)$ with norm equivalence. 
	However, the situation is more delicate if we consider a 
	non-constant variable exponent. In fact, 
	even if $\Omega$ is smooth and $p\in C^\infty(\setR^d)$, i.e., not depending on time, with 
	$p^->1$, we cannot hope for a coincidence of  
	$\bscal{X}^{2,p}_{\bfvarepsilon}(Q_T)$ and 
	$\bscal{X}^{2,p}_{\nb}(Q_T)$. This is due to the invalidity of 
	a Korn type inequality on $\bscal{X}^{2,p}_{\bfvarepsilon}(Q_T)$
	 (cf.~Remark \ref{4.5}). In consequence, one has to distinguish between the spaces $\bscal{X}^{2,p}_{\bfvarepsilon}(Q_T)$
	  and $\bscal{X}^{2,p}_{\nb}(Q_T)$. 
	At first sight, it might seem that the treatment of the space 
	$\bscal{X}^{2,p}_{\bfvarepsilon}(Q_T)$ does 
	not cause any further difficulties in comparison to 
	$\bscal{X}^{2,p}_{\nb}(Q_T)$. However, a closer inspection of 
	\cite{DNR12} shows that the proof of a formula of integration by parts, which is indispensable for the adaptation of the 
	methods on which Theorem~\ref{1.2} is based on, and the 
	related density of smooth functions in
	 $\bscal{X}^{2,p}_{\bfvarepsilon}(Q_T)$, 
	essentially make use of the following point-wise Poincar\'e
	 inequality on bounded Lipschitz domains (cf.~\cite[Proposition~4.5]{DNR12})
	\begin{align}
		\vert\bfu(x)\vert \leq c_0
		\int_{B_{2r(x)}^d(x)}{\frac{\vert \nb\bfu(y)\vert}
			{\vert x-y\vert^{d-1}\,dy}}\quad\text{ for a.e. }
		x\in \Omega,\label{eq:pw}
	\end{align}
	 where $c_0>0$ is a global constant, $r(x):=\text{dist}(x,\pa\Omega)$ 
	 for every $x\in \Omega$, and $\bfu\in W^{1,1}_0(\Omega)^d$. 
	 Therefore, the major objective of this paper is to prove an analogue of 
	 inequality \eqref{eq:pw} involving only the symmetric part of a gradient. 
	 As soon as this analoguous inequality is proved, we will be able to adapt the
	 smoothing methods from \cite{DNR12}, in order to gain density results
	 for $\bscal{X}^{2,p}_{\bfvarepsilon}(Q_T)$, which in turn 
	 leads to an appropriate formula of integration by parts. Using this
	  formula of integration by parts, we will adapt the methods on which Theorem \ref{1.2} 
	  is based on, i.e., we apply a corollary of the main theorem on pseudo-monotone
	   perturbations of maximal monotone mappings (cf.~Proposition~\ref{2.1a}), 
	   in order to gain an analogue of Theorem \ref{1.2} for our functional setting.
	 
	 \textbf{Plan of the paper:} In Section~\ref{sec:2} we recall some basic definitions
	 and results concerning pseudo-monotone operators, Bochner--Sobolev spaces 
	 and Lebesgue spaces with variable exponents. In Section~\ref{sec:3} we prove a 
	 point-wise Poincar\'e inequality near the boundary of a bounded Lipschitz domain
	 involving only the symmetric part of a gradient. In Section~\ref{sec:4} we 
	 introduce the function space $\bscal{X}^{q,p}_{\bfvarepsilon}(Q_T)$ 
	 and prove some of its basic properties, such as completeness, separability, reflexivity 
	 and a characterization of its dual space. In Section~\ref{sec:5} we apply the point-wise 
	 Poincar\'e inequality of Section~\ref{sec:3} in the construction of a smoothing 
	 operator for $\bscal{X}^{q,p}_{\bfvarepsilon}(Q_T)$ and its dual space $\bscal{X}^{q,p}_{\bfvarepsilon}(Q_T)^*$. In particular, we prove 
	 the density of $C^\infty_0(Q_T)^d$ in $\bscal{X}^{q,p}_{\bfvarepsilon}(Q_T)$. 
	 In Section~\ref{sec:6} we introduce the notion of the generalized time derivative which 
	 extends to usual notion for Bochner--Sobolev spaces to the framework of the energy spaces $\bscal{X}^{q,p}_{\bfvarepsilon}(Q_T)$ and 
	 $\bscal{X}^{q,p}_{\bfvarepsilon}(Q_T)^*$. To be more precise, we 
	 introduce the energy space $\bscal{W}^{q,p}_{\bfvarepsilon}(Q_T)$, 
	 a subspace of $\bscal{X}^{q,p}_{\bfvarepsilon}(Q_T)$ incorporating 
	 this new notion of generalized time derivative, and prove the density of 
	 $C^\infty(\overline{I},C_0^\infty(\Omega)^d)$ in $\bscal{W}^{q,p}_{\bfvarepsilon}(Q_T)$, 
	 which in turn leads to a generalized integration by parts formula.
	  This integration by parts formula is used afterwards in Section~\ref{sec:7} 
	  to adapt the methods on which Theorem \ref{1.2} is based on, 
	  in order to gain an analogue of Theorem \ref{1.2} in the functional 
	  setting of the preceding sections.

	\section{Preliminaries}
	\label{sec:2}
	\subsection{Operators}
	For a Banach space $X$ with norm $\|\cdot\|_X$ we denote by
        $X^*$ its dual space equipped with the norm
        $\|\cdot\|_{X^*}$. The duality pairing is denoted by
        $\langle\cdot,\cdot\rangle_X$. All occurring Banach spaces are
        assumed to be real. For Banach spaces $X, Y$ we denote the
        Cartesian product by $X\times Y$ and equip it with the norm
        $\|(x,y)^\top\|_{X\times Y}:=\|x\|_X +\|y\|_Y$ for
        $(x,y)^\top\in X\times Y$. By $D(A)$ we denote the domain of
        definition of an operator $A:D(A)\subseteq X\to Y$, by
        $R(A):=\{Ax\fdg x\in D(A)\}$ its range, and by
        $G(A):=\{(x,Ax)^\top\in X\times Y\fdg x\in D(A)\}$ its
        graph. For a linear bounded operator $A:X\to Y $ the adjoint
        operator $A^*:Y^*\to X^*$, which is again linear and bounded,
        is defined via
        $\langle y^*,Ax\rangle _Y=:\langle A^*y^*,x\rangle _X $ for
        every $x\in X$ and $y^*\in Y^*$. The sum of two subsets
        $S_1,S_2$ of $X$ is defined by
        $S_1+S_2:=\{s_1+s_2\fdg s_1\in S_1,s_2\in S_2\}$. For
        $X=\setR^n$, $n\in \setN$, $\Omega\subseteq \setR^n$ and
        $B_\vep^n(0):=\{x\in \setR^n\fdg |x|<\vep\}$ with $\vep>0$ we
        have
        $\Omega+B_\vep^n(0)=\{{x\in \setR^n\fdg}
        \textup{dist}(x,\Omega)<\vep\}$, which is an approximation of
        $\Omega$ from outside. On the other hand, we define
        $\Omega_\vep:=\{x\in \Omega\fdg
        \textup{dist}(x,\pa\Omega)>\vep\}$ for
        $\Omega\subseteq \setR^n$, $n\in \setN$, and $\vep>0$, which
        is an approximation of $\Omega$ from inside.
	\begin{defn}\label{2.1}
		Let $X$ and $Y$ be Banach spaces. 
		The operator $A:D(A)\subseteq X\to Y$ is said to be
		\begin{description}[{(iii)}]
			\item[{(i)}] \textbf{bounded}, if for all bounded 
			$M\subseteq D(A)\subseteq X$ the image $A(M)\subseteq Y$ is bounded.
			\item[{(ii)}] \textbf{monotone}, if $Y=X^*$ 
			and $\langle Ax-Ay,x-y\rangle_X\ge 0$ for all $x,y\in D(A)$.
			\item[{(iii)}] \textbf{pseudo-monotone}, if $Y=X^*$, $D(A)=X$ and 
			for a sequence $(x_n)_{n\in\setN}\subseteq X$ satisfying 
			$x_n\weakto x$~in~$X$~${(n\to\infty)}$ and 
			$\limsup_{n\to\infty}{\langle Ax_n,x_n-x\rangle_X}\leq 0$, it follows 
			$\langle Ax,x-y\rangle_X\leq \liminf_{n\to\infty}{\langle Ax_n,x_n-y\rangle_X}$ 
			for every $y\in X$.
			\item[{(iv)}] \textbf{maximal monotone}, if 
			$Y=X^*$, $A:D(A)\subseteq X\to X^*$ is monotone and 
			for every ${(x,x^*)^\top\in X\times X^*}$ from $\langle x^*-Ay,x-y\rangle_X\ge 0$ 
			for every $y\in D(A)$, it follows $x\in D(A)$ and $Ax=x^*$ in $X^*$.
			\item[{(v)}] \textbf{coercive}, if $Y=X^*$, $D(A)\subseteq X$ is 
			unbounded and
			$\lim_{\substack{\|x\|_X\to \infty\\x\in D(A)}}{\frac{\langle Ax,x\rangle_X}{\|x\|_X}}=\infty$.
		\end{description}
	\end{defn}

The following result is crucial for the existence result for problem
\eqref{eq:model} (cf.~Theorem \ref{7.4}). In particular, it
          generalizes the usual treatment of parabolic problems with
          the help of maximal monotone mappings for trivial initial
          conditions to the case of non-trivial initial values. 
	\begin{prop} \label{2.1a}Let $X,Y$ be reflexive Banach spaces, 
		$L:D(L)\subseteq X\to X^*$  linear and densely defined
                and $\gamma_0:D(L)\to Y$ linear, such that $R(\gamma_0)$ is dense in $Y$, $G((L,\gamma_0)^\top)$~is~closed~in $X\times X^*\times Y$,
		there exists a constant $c_0>0$, such that $\langle Lx,x\rangle_X+c_0\|\gamma_0(x)\|_Y^2\ge 0$~for~all~${x\in D(L)}$, and $L:N(\gamma_0)\subseteq X\to X^*$, where 
		$N(\gamma_0):=\{x\in D(L)\fdg \gamma_0(x)=0\}$, is maximal monotone.
		Moreover, let $A:D(L)\subseteq X\to X^*$ be 
		coercive and pseudo-monotone with respect to $L$, i.e., for a sequence 
		$(x_n)_{n\in \setN}\subseteq D(L)$ and an element $x\in D(L)$ from 
		\begin{align*}
			x_n\overset{n\to\infty}{\weakto}x\quad\text{ in  }X,\qquad\sup_{n\in \mathbb{N}}{\|Lx_n\|_{X^*}}<\infty,\\
			\limsup_{n\to \infty}{\langle Ax_n,x_n-x\rangle_X}\leq 0,
		\end{align*}
		it follows $\langle Ax,x-y\rangle_X\leq \limsup_{n\to \infty}{\langle Ax_n,x_n-x\rangle_X}$ 
		for every $y\in X$, and assume that there exist a bounded function $\psi:\setR_{\ge 0}\times \setR_{\ge 0}\to \setR_{\ge 0}$ and a constant $\theta\in \left[0,1\right)$, such that for every $x\in D(L)$
		\begin{align*}
                  \|Ax\|_{X^*}\leq \psi(\|x\|_X,\|\gamma_0(x)\|_Y)+\theta
                  \|Lx\|_{X^*}. 
		\end{align*}
		Then, $R((L+A,\gamma_0)^\top)=X^*\times Y$, i.e., for every $x^*\in X^*$ and $y_0\in Y$ there exists an element $x\in D(L)$, such that $Lx+Ax=x^*$ and $\gamma_0(x)=y_0$ in $Y$. 
	\end{prop}

	\begin{proof}
		This is a simplified version of \cite[Corollary 22]{Bre72}.\hfill$\qed$
	\end{proof}
	
	\subsection{Bochner--Lebesgue spaces}
	\label{subsec:2.2a}
        We use the standard notation for Bochner--Lebesgue spaces $L^r(I,X)$,
        where $1\le r\le \infty$ and $X$ is a Banach space. Elements
        of Bochner--Lebesgue spaces will be denoted by small boldface letters,
        e.g.~we write $\bsx \in L^r(I,X)$. In particular, we use the
        following characterization of duality in Bochner--Lebesgue spaces. 

        \begin{prop}[Characterization of $L^r(I,X)^*$]\label{riesz-boch}
 	Let $X$ be a reflexive Banach space and $r \in (1,\infty)$. Then,
               for every $\bsf^* \in L^r(I,X)^*$ there exists an
                element $\bsf\in L^{r'}(I,X^*)$, such that for every
                $\bsx\in L^r(I,X)$
		\begin{align*}
			\langle
                  \bsf^*,\bsx\rangle_{L^r(I,X)}=\int_{I}\langle
                  \bsf(t),\bsx(t)\rangle _X\,dt=: \langle  \bscal{J}_X\bsf,\bsx\rangle_{L^r(I,X)}
		\end{align*}
		and $\|\bsf^*\|_{L^r(I,X)^*} =\|\bsf\|_{L^{r'}(I,X^*)}$, 
		i.e., the operator ${\bscal{J}_X: L^{r'}(I,X^*)\to
                  L^r(I,X)^*:\bsf \mapsto  \bscal{J}_X\bsf}$ is an isometric isomorphism. 
	\end{prop}

        \subsection{Bochner-Sobolev spaces}
	\label{subsec:2.2}
        Let $(X_0,Y,j_0)$ and $(X_1,Y,j_1)$ be evolution triples,
        i.e., for $i=0,1$,  $X_i$ is a reflexive Banach space, $Y$ a
        Hilbert space and $j_i:X_i\to Y$ a dense embedding, i.e., $j_i
        $ is linear, bounded and injective and $R(j_i)$ is dense in $Y$. Let $R:Y\to Y^*$ be the Riesz isomorphism with respect to $(\cdot,\cdot)_Y$. As $j_1$ is a dense embedding, the adjoint $j_1^*:Y^*\to X_1^*$ and therefore $e:=j_1^*Rj_0:X_0\to X_1^*$ are embeddings as well. Note that
	\begin{align}
		\langle ex_0,x_1\rangle_{X_1}=(j_0x_0,j_1x_1)_Y\quad\text{ for all }x_0\in X_0,\;x_1\in X_1.\label{eq:iden}
	\end{align}
	Let $I=\left(0,T\right)$, $T<\infty$, and ${p_0,p_1\in\left[1,\infty\right]}$.
	A function 
	$\bsx_0\in L^{p_0}(I,X_0)$ possesses a \textbf{generalized time derivative with respect 
		to $e$ in $L^{p_1}(I,X_1^*)$} if there exists a function $\bsx_1^*\in L^{p_1}(I,X_1^*)$, 
	such that for every $\varphi\in C_0^\infty(I)$ and $x_1\in X_1$ it holds
	\begin{align}
	-\int_I{(j_0(\bsx_0(t)),j_1x_1)_Y\varphi^\prime(t)\,dt}
	=\int_I{\langle \bsx_1^*(t),x_1\rangle_{X_1}\varphi(t)\,dt}.\label{eq:gtd}
	\end{align}
	As such a function $\bsx_1^*\in L^{p_1}(I,X_1^*)$ is unique 
	(cf.~\cite[Prop. 23.18]{Zei90A}), 
	$\frac{d_e\bsx_0}{dt}:=\bsx_1^*$ is well-defined.~By
	\begin{align*}
	W_e^{1,p_0,p_1}(I,X_0,X_1,Y):=\bigg\{\bsx_0\in L^{p_0}(I,X_0)\bigg|\;
	\exists\;\frac{d_e\bsx_0}{dt}\in L^{p_1}(I,X_1^*)\bigg\}
	\end{align*}
	we denote the \textbf{Bochner--Sobolev space with respect to }$e$.
	
	The following proposition shows that the generalized time derivative is just a special case of the usual vector-valued distributional time derivative, such as e.g. in \cite[Definition 2.1.4]{Dro01}.
	
	\begin{prop}\label{equal}
		The following statements are equivalent:
		\begin{description}[{(ii)}]
			\item[(i)] $\bsx_0\in W_e^{1,p_0,p_1}(I,X_0,X_1,Y)$.
			\item[(ii)] $\bsx_0\in L^{p_0}(I,X_0)$ and there exists a function $\bsx^*_1\in L^{p_1}(I,X_1^*)$, such that for every $\varphi\in C_0^\infty(I)$
			\begin{align}
				e\bigg(-\int_I{\bsx_0(t)\varphi^\prime(t)\,dt}\bigg)=\int_I{\bsx^*_1(t)\varphi(t)\,dt}\qquad
                          \textrm{ in } X^*_1.\label{eq:equal.a}
			\end{align} 
		\end{description}
	\end{prop}

	\begin{proof}
		\textbf{(i) $\boldsymbol{\Rightarrow}$ (ii)} Using the
                properties of the Bochner integral and \eqref{eq:iden}, we deduce for every $\varphi\in C_0^\infty(I)$ and $x_1\in X_1$ that
		\begin{align}
		\begin{split}
			\bigg\langle e\bigg(-\int_I{\bsx_0(t)\varphi^\prime(t)\,dt}\bigg),x_1\bigg\rangle_{X_1}&=
			-\int_I{\langle e(\bsx_0(t)),x_1\rangle_{X_1}\varphi^\prime(t)\,dt}\\&=
			-\int_I{(j_0(\bsx_0(t)),j_1x_1)_Y\varphi^\prime(t)\,dt}.
		\end{split}\label{eq:equal1}
		\end{align}
		Similarly, we have for every $\varphi\in C_0^\infty(I)$ and $x_1\in X_1$
		\begin{align}
			\int_I{\langle \bsx_1^*(t),x_1\rangle_{X_1}\varphi(t)\,dt}=
			\bigg\langle\int_I{\bsx_1^*(t)\varphi(t)\,dt},x_1\bigg\rangle_{X_1}.\label{eq:equal2}
		\end{align}
		Taking into account \eqref{eq:gtd}, we conclude from \eqref{eq:equal1} and \eqref{eq:equal2} that \eqref{eq:equal.a} holds, i.e., \textbf{(ii)}.\\[-3mm]
		
		\textbf{(ii) $\boldsymbol{\Rightarrow}$ (i)} We observe that \eqref{eq:equal1} and \eqref{eq:equal2} are still valid. Hence, connecting \eqref{eq:equal1} and \eqref{eq:equal2} by means of \eqref{eq:equal.a} yields \eqref{eq:gtd}.\hfill$\qed$
	\end{proof}
	
	By means of Proposition \ref{equal} and \cite[Theorem 2.1.1]{Dro01} we immediately conclude that the space $W_e^{1,p_0,p_1}(I,X_0,X_1,Y)$ forms a Banach space, if we equip it with the norm
	\begin{align*}
	\|\cdot\|_{W_e^{1,p_0,p_1}(I,X_0,X_1,Y)}:=\|\cdot\|_{L^{p_0}(I,X_0)}
	+\left\|\frac{d_e}{dt}\cdot\right\|_{L^{p_1}(I,X_1^*)}.
	\end{align*}
	\subsection{Variable exponent spaces and their basic properties}
	
	In this section we summarize all essential information
        concerning variable exponent Lebesgue and Sobolev spaces which
        will find use in the hereinafter analysis. These spaces has
        already been studied by Hudzik \cite{Hud80}, Musielak
        \cite{Mus83}, Kov\'{a}\v{c}ik, R\'{a}kosn\'{\i}k \cite{KR91},
        R\r{u}\v{z}i\v{c}ka \cite{Ru00} and many others. A detailed
        analysis, including all results of this section, can be found in \cite{KR91,DHHR11,CUF13}.
	\medskip

    Let $G\subseteq \setR^n$, $n\in \setN$, be a measurable set and $p:G\to\left[1,\infty\right)$ a measurable function.
    Then, $p$ will be called a \textbf{variable exponent} and we define  $\mathcal{P}(G)$ to be the set of all variable exponents.
    Furthermore, we define 
	$p^-:=\essinf_{x\in G}{p(x)}$, $p^+:=\esssup_{x\in G}{p(x)}$ and the set of \textbf{bounded variable exponents} $\mathcal{P}^\infty(G):=\{p\in \mathcal{P}(G)\fdg p^+<\infty\}$.
	
	For $p\in \mathcal{P}^\infty(G)$, 
	the \textbf{variable exponent Lebesgue space} $L^{p(\cdot)}(G)$ consists of all 
	measurable functions $f:G\to \setR$ for which the modular
	\begin{align*}
	\rho_{p(\cdot)}(f):=\int_{G}{\vert f(x)\vert^{p(x)}\,dx}
	\end{align*}
	is finite. The Luxembourg norm 
	\begin{align*}
	\|f\|_{L^{p(\cdot)}(G)}:=
	\inf\{\lambda>0\fdg \rho_{p(\cdot)}(\lambda^{-1}f)\leq 1\}
	\end{align*}
	turns $L^{p(\cdot)}(G)$ into a Banach space. 
	
	For $p\in \mathcal{P}^\infty(G)$ with $p^->1$, $p'\in
        \mathcal{P}^\infty(G)$ is defined via
        $p'(x):=\frac{p(x)}{p(x)-1}$ for almost~every~${x\in G}$. We
        make frequent use of H\"older's inequality
        \begin{align*}
          \int_{G}|f(x)|\,|g(x)|\,dx \le 2\, \|f\|_{p(\cdot)} \|g\|_{p'(\cdot)}
        \end{align*}
        valid for every $f \in L^{p(\cdot)}(G)$ and $g\in
        L^{p'(\cdot)}(G)$. Thus, the following product
	\begin{align}
		(f,g)_{L^{p(\cdot)}(G)}:=\int_{G}{f(x)g(x)\,dx},\label{prod}
	\end{align}
	 which we will make frequent use of, is well-defined for every $f\in L^{p'(\cdot)}(\Omega)$ and $g\in L^{p(\cdot)}(\Omega)$. The products $(\cdot,\cdot)_{L^{p(\cdot)}(G)^n}$ and $(\cdot,\cdot)_{L^{p(\cdot)}(G)^{n\times n}}$ are defined accordingly.
	 
 We have the following characterization of duality in $L^{p(\cdot)}(G)$.
	
	\begin{prop}[Characterization of $L^{p(\cdot)}(G)^*$]\label{riesz}
 	Let $p\in \mathcal{P}^\infty(G)$ with $p^->1$. Then,
               for every $f^* \in L^{p(\cdot)}(G)^*$ there exists an
                element $f\in L^{p'(\cdot)}(G)$, such that for every
                $g\in L^{p(\cdot)}(G)$
		\begin{align*}
			\langle f^*,g\rangle_{L^{p(\cdot)}(G)}=\int_{G}{f(x)g(x)\,dx}=: \langle Jf,g\rangle_{L^{p(\cdot)}(G)}
		\end{align*}
		and $\frac{1}{2}\|f\|_{L^{p'(\cdot)}(G)}\leq \|f^*\|_{L^{p(\cdot)}(G)^*}\leq 2\|f\|_{L^{p'(\cdot)}(G)}$, 
		i.e., the operator $J:L^{p'(\cdot)}(G)\to
                  L^{p(\cdot)}(G)^*:f \mapsto Jf$ is an isomorphism. Moreover, 
		$L^{p(\cdot)}(G)$ is reflexive.  
	\end{prop}

	Furthermore, we will use the following special notion of variable exponent Sobolev spaces, 
	which is a slight generalization of the notion in \cite{Hu11}.
	
	\begin{defn}
		Let $G\subseteq \setR^n$, $n\in \setN$, be open
		and $q,p\in \mathcal{P}^\infty(G)$. We define the \textbf{variable exponent Sobolev spaces}
		\begin{align*}
			\widehat{X}^{q(\cdot),p(\cdot)}_{\nb}(G)&:=\{f\in L^{q(\cdot)}(G)\fdg \nabla f\in L^{p(\cdot)}(G)^n\},\\
			\widehat{X}^{q(\cdot),p(\cdot)}_{\bfvarepsilon}(G)&:=\{\bff\in L^{q(\cdot)}(G)^n\fdg \bfvarepsilon(\bff)\in L^{p(\cdot)}(G,\mathbb{M}^{n\times n}_{\sym})\},
		\end{align*} 
		where the $\nabla f$ and $\bfvarepsilon(\bff)$ have to be understood in a distributional sense.
	\end{defn}
	
	\begin{prop}
		Let $G\subseteq \setR^n$, $n\in \setN$, be open
		and $q,p\in \mathcal{P}^\infty(G)$. Then, the norms
		\begin{align*}
			\|\cdot\|_{\widehat{X}^{q(\cdot),p(\cdot)}_{\nb}(G)}&:=\|\cdot\|_{L^{q(\cdot)}(G)}
			+\|\nb\cdot\|_{L^{p(\cdot)}(G)^n},\\
			\|\cdot\|_{\widehat{X}^{q(\cdot),p(\cdot)}_{\bfvarepsilon}(G)}&:=\|\cdot\|_{L^{q(\cdot)}(G)^n}
			+\|\bfvarepsilon(\cdot)\|_{L^{p(\cdot)}(G)^{n\times n}}
		\end{align*}
		turn $\widehat{X}^{q(\cdot),p(\cdot)}_{\nb}(G)$  and $\widehat{X}^{q(\cdot),p(\cdot)}_{\bfvarepsilon}(G)$, respectively, into Banach spaces. 
	\end{prop}

	\begin{proof}
		We only give a proof for $\widehat{X}^{q(\cdot),p(\cdot)}_{\bfvarepsilon}(G)$, since all arguments transfer to $\widehat{X}^{q(\cdot),p(\cdot)}_{\nb}(G)$.
		So, let $(\bff_k)_{k\in \setN}\subseteq \widehat{X}^{q(\cdot),p(\cdot)}_{\bfvarepsilon}(G)$ be a Cauchy sequence, i.e., both $(\bff_k)_{k\in \mathbb{N}}\subseteq L^{q(\cdot)}(G)^n$ and $(\bfvarepsilon(\bff_k))_{k\in \setN}\subseteq L^{p(\cdot)}(G,\mathbb{M}^{n\times n}_{\sym})$ are  Cauchy sequences. Thus, as $L^{q(\cdot)}(G)^n$ and $L^{p(\cdot)}(G,\mathbb{M}^{n\times n}_{\sym})$ are Banach spaces, we obtain elements $\bff \in L^{q(\cdot)}(G)^n$ and $\bscal{E}\in L^{p(\cdot)}(G,\mathbb{M}^{n\times n}_{\sym})$, such that $\bff_k\to \bff $ in $L^{q(\cdot)}(G)^n$ $(k\to \infty)$ and $\bfvarepsilon(\bff_k)\to \bscal{E}$ in $L^{p(\cdot)}(G,\mathbb{M}^{n\times n}_{\sym})$ $(k\to \infty)$. Hence, for every $\mathbf{X}\in C_0^\infty(G,\mathbb{M}^{n\times n}_{\sym})$ there holds
		\begin{align*}
			\int_{G}{\bff(x)\cdot\divo \bfX(x)\,dx}&=\lim_{k\to\infty}{\int_{G}{\bff_k(x)\cdot\divo \bfX(x)\,dx}}\\&=\lim_{k\to\infty}{-\int_{G}{\bfvarepsilon(\bff_k)(x):\bfX(x)\,dx}}=-\int_{G}{\bscal{E}(x):\bfX(x)\,dx},
		\end{align*}
		i.e., $\bff\in \widehat{X}^{q(\cdot),p(\cdot)}_{\bfvarepsilon}(G)$ with $\bfvarepsilon(\bff)=\bscal{E}$ in $L^{p(\cdot)}(G,\mathbb{M}^{n\times n}_{\sym})$ and $\bff_k\to \bff $ in $\widehat{X}^{q(\cdot),p(\cdot)}_{\bfvarepsilon}(G)$ $(k\to\infty)$. \hfill$\qed$
	\end{proof}
	
	\begin{defn}
		We define $X^{q(\cdot),p(\cdot)}_{\nb}(G)$ and $X^{q(\cdot),p(\cdot)}_{\bfvarepsilon}(G)$  to be  the closures of $C_0^\infty(G)$ and  $C_0^\infty(G)^n$  with respect to
		$\|\cdot\|_{\widehat{X}^{q(\cdot),p(\cdot)}_{\nb}(G)}$ and  $\|\cdot\|_{\widehat{X}^{q(\cdot),p(\cdot)}_{\bfvarepsilon}(G)}$, respectively. 
	\end{defn}
	
	Clearly, we have the embeddings ${\widehat{X}^{q(\cdot),p(\cdot)}_{\nb}(G)^n\!\embedding\! \widehat{X}^{q(\cdot),p(\cdot)}_{\bfvarepsilon}(G)}$  and ${X^{q(\cdot),p(\cdot)}_{\nb}(G)^n\!\embedding\! X^{q(\cdot),p(\cdot)}_{\bfvarepsilon}(G)}$. Conversely, we have ${\widehat{X}^{q(\cdot),p(\cdot)}_{\bfvarepsilon}(G)\not\embedding \widehat{X}^{q(\cdot),p(\cdot)}_{\nb}(G)^n}$  and ${X^{q(\cdot),p(\cdot)}_{\bfvarepsilon}(G)\not\embedding X^{q(\cdot),p(\cdot)}_{\nb}(G)^n}$ in general.
		However, there is a special modulus of continuity for variable exponents 
	that is sufficient for the boundedness of the Hardy-Littlewood maximal operator and a coincidence of the spaces $X^{q(\cdot),p(\cdot)}_{\nb}(G)^n$ and $ X^{q(\cdot),p(\cdot)}_{\bfvarepsilon}(G)$.

	We say that a bounded exponent $p\in \mathcal P^\infty (G)$ is \textbf{locally
	$\log$-Hölder continuous}, if there is a constant $c_1>0$, such that
	for all $x,y\in G$
	\begin{align*}
	\vert p(x)-p(y)\vert \leq \frac{c_1}{\log(e+1/\vert x-y\vert)}.
	\end{align*}
	We say that $p \in \mathcal P^\infty (G)$ satisfies the \textbf{$\log$-Hölder decay condition}, if there exist 
	constant $c_2>0$ and $p_\infty\in \setR$, such that for all $x\in G$
	\begin{align*}
	\vert p(x)-p_\infty\vert \leq\frac{c_2}{\log(e+1/\vert x\vert)}.
	\end{align*} 
	We say that $p$ is \textbf{globally $\log$-Hölder continuous} on $G$, if it is locally 
	log-Hölder continuous and satisfies the log-Hölder decay condition. 
	The maximum $c_{\log}(p):=\max\{c_1,c_2\}$ is  called the \textup{\textsf{\textbf{$\log$-Hölder constant}}} of $p$.
	Moreover, we denote by $\mathcal{P}^{\log}(G)$ the set of 
	globally $\log$-Hölder continuous
	functions on $G$. 
	
	One pleasant property of globally $\log$-Hölder continuous exponents is their 
	ability to admit extensions to the whole space $\setR^n$, with similar characteristics.
	 
	 \begin{prop}\label{2.1.1}
	 	Let $G\subsetneq \setR^n$, $n\in\setN$, be a domain and $p\in \mathcal{P}^{\log}(G)$. Then, there 
	 	exists an extension $\overline{p}\in \mathcal{P}^{\log}(\setR^n)$ with 
	 	$\overline{p}=p$ in $G$, $\overline{p}^-=p^-$ and $\overline{p}^+=p^+$.
	 \end{prop}
	 
	 The main reason for considering globally $\log$-Hölder continuous exponents is the boundedness
	 of the Hardy-Littlewood 
	 maximal operator.
	
	\begin{prop}\label{2.2}
		Let 
		$p\in \mathcal{P}^{\log}(\setR^n)$, $n\in \setN$, with $p^->1$. 
		Then, the Hardy-Littlewood 
		maximal operator $M_n:L^{p(\cdot)}(\setR^n)\to L^{p(\cdot)}(\setR^n)$, for every $f\in L^{p(\cdot)}(\setR^n)$
		given via
		\begin{align*}
		M_n(f)(x):=\sup_{r>0}{\fint_{B_r^n(x)}{\vert f(y)\vert\,dy}}:=\sup_{r>0}{\frac{1}{\vert B_r^n(x)\vert}\int_{B_r^n(x)}{\vert f(y)\vert\,dy}},
		\end{align*}
		is well-defined and bounded. In particular, there exists 
		a constant $c=c(c_{\log}(p) ,p^-,n)>0$, 
		such that for every $f\in L^{p(\cdot)}(\setR^n)$ it holds
		\begin{align*}
		\|M_n(f)\|_{L^{p(\cdot)}(\setR^n)}\leq c\|f\|_{L^{p(\cdot)}(\setR^n)}.
		\end{align*}
	\end{prop}
	
	\begin{prop}\label{2.3}
		Let $p\in \mathcal{P}^{\log}(\setR^n)$, $n\in \setN$, with $p^->1$. 
		For the \textbf{standard 
		mollifier} ${\omega^n\in C_0^\infty(B_1^n(0))}$, given via $\omega^n(x):=c_{\omega^n}\exp\big(\frac{-1}{1-\vert x\vert^2}\big)$ if $x\in  B_1^n(0)$ and $\omega^n(x):=0$ if $x\in \setR^n$, where $c_{\omega^n}>0$ is chosen so that $\int_{\setR^n}{\omega^n(x)\,dx}=1$, we set $\omega_\vep^n(x)
		 :=\frac{1}{\vep^n}\omega^n(\frac{x}{\vep})$ for every 
		 $x\in \setR^n$ and $\vep>0$. 
		 Then, for every $f\in L^{p(\cdot)}(\setR^n)$ it holds:
		\begin{description}[{(iii)}]
			\item[(i)] There exists $K\!=\!K(\|\omega^n\|_{L^1(\setR^n)}, c_{\log}(p))\!>\!0$, such that $\sup_{\vep>0}{\|\omega_\vep^n\ast f\|_{L^{p(\cdot)}(\setR^n)}}
			\!\leq\! K\|f\|_{L^{p(\cdot)}(\setR^n)}$.
			\item[(ii)] 
			$\sup_{\vep> 0}{\vert \omega_\vep^n\ast f\vert}\leq 2M_n(f)$ 
			almost everywhere in $\setR^n$.
			\item[(iii)] 
			$ \omega_\vep^n\ast f\to f$ in $L^{p(\cdot)}(\setR^n)$ $(\vep\to 0)$.
		\end{description}
	\end{prop}

	The following proposition expresses that under the assumption
        of log-Hölder continuity of the exponent, i.e., $p\in
        \mathcal{P}^{\log}(G)$ with $p^->1$, the spaces $X^{q(\cdot),p(\cdot)}_{\bfvarepsilon}(G)$ and $X^{q(\cdot),p(\cdot)}_{\nb}(G)^n$ coincide.
	
	\begin{prop}[Korn's inequality]\label{2.4}
		Let $G\subseteq \setR^n$, $n\in \setN$, be a bounded domain, 
		and $p\in\mathcal{P}^{\log}(G)$ with $p^->1$. Then, there exists a 
		constant $c>0$, such that for every ${\bff\in X^{q(\cdot),p(\cdot)}_{\bfvarepsilon}(G)}$ it holds $\bff\in X^{q(\cdot),p(\cdot)}_{\nb}(G)^n$ with
		\begin{align*}
		\|\nb \bff\|_{L^{p(\cdot)}(G)^{n\times n}}\leq 
		c\|\bfvarepsilon(\bff)\|_{L^{p(\cdot)}(G)^{n\times n}}.
		\end{align*}
	\end{prop}

	\section{Point-wise Poincar\'e inequality for the symmetric gradient 
		near the boundary of a bounded Lipschitz domain}
	\label{sec:3}
	
	The following inequality plays a central role in the later construction of 
	an appropriate smoothing operator. The method of incorporating only
	the symmetric gradient is based on ideas from \cite{CM19}. More precisely, 
	\cite[Inequality (2.12)]{CM19} is to weak for our approaches as the integral
	on right-hand side is with respect to the whole domain rather than a
	ball which is shrinking, when approaching the boundary of the domain, such as
	e.g. in \eqref{eq:pw}. In fact, this smallness near the boundary of the domain 
	is indispensable for the 
	viability of the smoothing method in Section \ref{sec:5}, as can be seen in \eqref{eq:5.4}.

	\begin{thm}\label{2.31}
		Let $\Omega\subseteq \setR^d$, $d\ge 2$, be a bounded 
		Lipschitz domain and $r(x):=\textup{dist}(x,\pa\Omega)$ for $x\in \Omega$.
		Then, there exist constants $c_0=c_0(d,\Omega),h_0=h_0(\Omega)>0$, such that 
		for every $\bfu\in X^{1,1}_{\bfvarepsilon}(G)$ and almost every 
		$x\in \Omega$ with $r(x)\leq h_0$ 
		there holds
		\begin{align}
		\vert \bfu(x)\vert \leq c_0\int_{B^d_{2r(x)}(x)\cap \Omega}
		{\frac{\vert \bfvarepsilon(\bfu)(y)\vert}{\vert x-y\vert^{d-1}}dy}.\label{eq:poin}
		\end{align}
		\end{thm}
	
	\begin{proof}
		We split the proof into three main steps:
		
		\textbf{Step 1:} 
		To begin with, we will show that for a bounded Lipschitz domain in $\setR^d$
		there exist constants  $h_0=h_0(\Omega),c_1=c_1(d,\Omega)>0$, 
		such that 
		for every $x\in\Omega$ with $r(x)\leq h_0$ the 
		intersection $\pa B^d_{2r(x)}(x)\cap \Omega^c$ 
		contains a hyperspherical cap of $\pa B^d_{2r(x)}(x)$ 
		with surface area greater or equal than $c_1(2r(x))^{d-1}$.
		
		To this end, we use that a bounded 
		Lipschitz domain $\Omega$  satisfies a uniform exterior cone property 
		(cf.~\cite[Theorem 1.2.2.2]{Gri85}), i.e., there exist a height $h>0$ and
		an opening angle $\theta\in\left(0,\frac{\pi}{2}\right)$, such that for every 
		$x\in\pa \Omega$ there exists some axis $\xi_x\in\mathbb{S}^{d-1}$ 
		with $x+\mathcal{C}(\xi_x,\theta,h)\subseteq \Omega^c$, where 
		$\mathcal{C}(\xi_x,\theta,h):=\{y\in\setR^d\fdg 
		y\cdot\xi_x>\vert y\vert \cos(\theta),y_d<h\}$
		is an open cone. 
		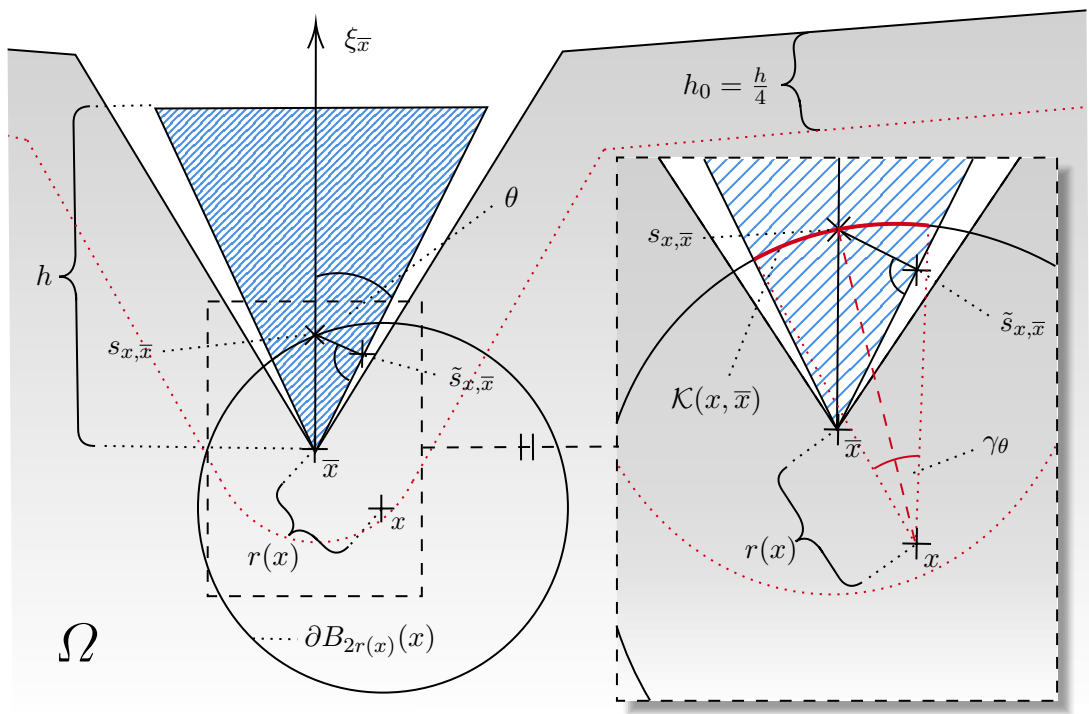
\begin{figure}[hbt!]
			\centering
			
			\tikzset {_10bcscmrx/.code = {\pgfsetadditionalshadetransform{ \pgftransformshift{\pgfpoint{0 bp } { 0 bp }  }  \pgftransformrotate{-90 }  \pgftransformscale{2 }  }}}
			\pgfdeclarehorizontalshading{_cmbmqkbv4}{150bp}{rgb(0bp)=(0.82,0.82,0.82);
				rgb(37.5bp)=(0.82,0.82,0.82);
				rgb(44.01786101715905bp)=(0.86,0.86,0.86);
				rgb(47.76786314589637bp)=(0.89,0.89,0.89);
				rgb(62.5bp)=(1,1,1);
				rgb(100bp)=(1,1,1)}
			
			
			\tikzset{
				pattern size/.store in=\mcSize, 
				pattern size = 5pt,
				pattern thickness/.store in=\mcThickness, 
				pattern thickness = 0.3pt,
				pattern radius/.store in=\mcRadius, 
				pattern radius = 1pt}
			\makeatletter
			\pgfutil@ifundefined{pgf@pattern@name@_9vu7l91ds}{
				\pgfdeclarepatternformonly[\mcThickness,\mcSize]{_9vu7l91ds}
				{\pgfqpoint{0pt}{0pt}}
				{\pgfpoint{\mcSize+\mcThickness}{\mcSize+\mcThickness}}
				{\pgfpoint{\mcSize}{\mcSize}}
				{
					\pgfsetcolor{\tikz@pattern@color}
					\pgfsetlinewidth{\mcThickness}
					\pgfpathmoveto{\pgfqpoint{0pt}{0pt}}
					\pgfpathlineto{\pgfpoint{\mcSize+\mcThickness}{\mcSize+\mcThickness}}
					\pgfusepath{stroke}
			}}
			\makeatother
			
			
			\tikzset {_7wbes2ceh/.code = {\pgfsetadditionalshadetransform{ \pgftransformshift{\pgfpoint{0 bp } { 0 bp }  }  \pgftransformrotate{-90 }  \pgftransformscale{2 }  }}}
			\pgfdeclarehorizontalshading{_6wtnomvnx}{150bp}{rgb(0bp)=(0.82,0.82,0.82);
				rgb(37.5bp)=(0.82,0.82,0.82);
				rgb(52.14285752602986bp)=(0.86,0.86,0.86);
				rgb(56.24860256910324bp)=(0.89,0.89,0.89);
				rgb(62.5bp)=(1,1,1);
				rgb(100bp)=(1,1,1)}
			
			
			\tikzset{
				pattern size/.store in=\mcSize, 
				pattern size = 5pt,
				pattern thickness/.store in=\mcThickness, 
				pattern thickness = 0.3pt,
				pattern radius/.store in=\mcRadius, 
				pattern radius = 1pt}
			\makeatletter
			\pgfutil@ifundefined{pgf@pattern@name@_sfjk5nbj7}{
				\pgfdeclarepatternformonly[\mcThickness,\mcSize]{_sfjk5nbj7}
				{\pgfqpoint{0pt}{0pt}}
				{\pgfpoint{\mcSize+\mcThickness}{\mcSize+\mcThickness}}
				{\pgfpoint{\mcSize}{\mcSize}}
				{
					\pgfsetcolor{\tikz@pattern@color}
					\pgfsetlinewidth{\mcThickness}
					\pgfpathmoveto{\pgfqpoint{0pt}{0pt}}
					\pgfpathlineto{\pgfpoint{\mcSize+\mcThickness}{\mcSize+\mcThickness}}
					\pgfusepath{stroke}
			}}
			\makeatother
			\tikzset{every picture/.style={line width=0.75pt}} 
			
			\begin{tikzpicture}[x=0.9pt,y=0.9pt,yscale=-1,xscale=1]
			
			\path  [shading=_cmbmqkbv4,_10bcscmrx] (561.89,1.79) -- (561.89,298.79) -- (102.49,300.54) -- (102.49,293.54) -- (100.78,18.26) -- (132.89,20.79) -- (233.69,187.93) -- (337.72,19.41) -- cycle ; 
			\draw   (561.89,1.79) -- (561.89,298.79) -- (102.49,300.54) -- (102.49,293.54) -- (100.78,18.26) -- (132.89,20.79) -- (233.69,187.93) -- (337.72,19.41) -- cycle ; 
			
			\draw [color={rgb, 255:red, 255; green, 255; blue, 255 }  ,draw opacity=1 ][line width=3.75]    (561.5,296) -- (102.49,300.54) ;
			\draw  [pattern=_9vu7l91ds,pattern size=2.8499999999999996pt,pattern thickness=0.75pt,pattern radius=0pt, pattern color={rgb, 255:red, 74; green, 144; blue, 226}] (233.69,187.93) -- (166.53,43.4) -- (305.98,42.87) -- cycle ;
			\draw   (228,186.64) -- (238.96,186.64)(233.48,180.84) -- (233.48,192.45) ;
			\draw   (256,211.64) -- (266.96,211.64)(261.48,205.84) -- (261.48,217.45) ;
			\draw  [dash pattern={on 0.84pt off 2.51pt}]  (233.5,186.54) -- (219.98,201.15) ;
			\draw  [dash pattern={on 0.84pt off 2.51pt}]  (248.63,226.63) -- (262.17,211.33) ;
			\draw   (184.42,211.33) .. controls (184.42,168.39) and (219.23,133.59) .. (262.17,133.59) .. controls (305.11,133.59) and (339.91,168.39) .. (339.91,211.33) .. controls (339.91,254.27) and (305.11,289.08) .. (262.17,289.08) .. controls (219.23,289.08) and (184.42,254.27) .. (184.42,211.33) -- cycle ;
			\draw    (233.69,187.93) -- (233.96,9.17) ;
			\draw [shift={(233.96,7.17)}, rotate = 450.09] [color={rgb, 255:red, 0; green, 0; blue, 0 }  ][line width=0.75]    (10.93,-3.29) .. controls (6.95,-1.4) and (3.31,-0.3) .. (0,0) .. controls (3.31,0.3) and (6.95,1.4) .. (10.93,3.29)   ;
			\draw  [dash pattern={on 0.84pt off 2.51pt}]  (141.47,42.5) -- (166.33,42) ;
			\draw  [dash pattern={on 0.84pt off 2.51pt}]  (145.17,185.33) -- (233.5,186.54) ;
			\draw [color={rgb, 255:red, 255; green, 255; blue, 255 }  ,draw opacity=1 ][line width=3.75]    (102.29,299.15) -- (561.7,297.4) ;
			\draw   (139,42.45) .. controls (134.33,42.52) and (132.04,44.89) .. (132.11,49.56) -- (132.93,104) .. controls (133.04,110.67) and (130.76,114.04) .. (126.09,114.11) .. controls (130.76,114.04) and (133.14,117.33) .. (133.24,124)(133.19,121) -- (134.06,178.44) .. controls (134.13,183.11) and (136.5,185.4) .. (141.17,185.33) ;
			\draw  [draw opacity=0] (234.19,114.08) .. controls (245.71,110.38) and (258.37,114.05) .. (266.15,123.21) -- (243.31,142.66) -- cycle ; \draw   (234.19,114.08) .. controls (245.71,110.38) and (258.37,114.05) .. (266.15,123.21) ;
			\draw  [dash pattern={on 0.84pt off 2.51pt}]  (248,129.82) -- (309.39,83.46) ;
			\draw   (229.59,134.79) -- (237.38,142.5)(237.57,134.52) -- (229.4,142.77) ;
			\draw  [draw opacity=0] (248.05,157.69) .. controls (247.85,157.61) and (247.65,157.52) .. (247.46,157.42) .. controls (242.34,154.83) and (240.32,148.52) .. (242.95,143.32) .. controls (243.06,143.09) and (243.18,142.87) .. (243.31,142.66) -- (252.21,148) -- cycle ; \draw   (248.05,157.69) .. controls (247.85,157.61) and (247.65,157.52) .. (247.46,157.42) .. controls (242.34,154.83) and (240.32,148.52) .. (242.95,143.32) .. controls (243.06,143.09) and (243.18,142.87) .. (243.31,142.66) ;
			\draw  [color={rgb, 255:red, 0; green, 0; blue, 0 }  ][line width=1.5] [line join = round][line cap = round] (247.28,149.57) .. controls (247.28,149.57) and (247.28,149.57) .. (247.28,149.57) ;
			\draw   (248,146.55) -- (258.96,146.74)(253.58,140.84) -- (253.38,152.45) ;
			\draw  [color={rgb, 255:red, 0; green, 0; blue, 0 }  ,draw opacity=1 ] (441,11.45) .. controls (436.37,12) and (434.32,14.59) .. (434.86,19.22) -- (435.28,22.79) .. controls (436.06,29.41) and (434.13,32.99) .. (429.5,33.54) .. controls (434.13,32.99) and (436.84,36.03) .. (437.62,42.65)(437.27,39.67) -- (438.04,46.22) .. controls (438.58,50.85) and (441.17,52.9) .. (445.81,52.36) ;
			\draw  [draw opacity=0][fill={rgb, 255:red, 255; green, 255; blue, 255 }  ,fill opacity=1 ][dash pattern={on 4.5pt off 4.5pt}][blur shadow={shadow xshift=4.5pt,shadow yshift=-4.5pt, shadow blur radius=2.25pt, shadow blur steps=4 ,shadow opacity=38}] (360.5,64) -- (545.49,64) -- (545.49,292.52) -- (360.5,292.52) -- cycle ;
			\path  [shading=_6wtnomvnx,_7wbes2ceh] (378,64) -- (452.99,178.26) -- (529.78,64.15) -- (545.49,64) -- (545.49,292.52) -- (360.5,292.52) -- (360.5,64) -- cycle ; 
			\draw  [dash pattern={on 4.5pt off 4.5pt}] (378,64) -- (452.99,178.26) -- (529.78,64.15) -- (545.49,64) -- (545.49,292.52) -- (360.5,292.52) -- (360.5,64) -- cycle ; 
			
			\draw    (378,64) -- (452.99,178.26) -- (529.78,64.15) ;
			\draw [color={rgb, 255:red, 0; green, 0; blue, 0 }  ,draw opacity=1 ][pattern=_sfjk5nbj7,pattern size=6pt,pattern thickness=0.75pt,pattern radius=0pt, pattern color={rgb, 255:red, 74; green, 144; blue, 226}]   (397.12,64.68) -- (452.99,178.26) -- (510.12,64.68) ;
			\draw  [dash pattern={on 4.5pt off 4.5pt}]  (397.39,64.08) -- (510.39,64.08) ;
			\draw   (447.46,178.34) -- (459.51,178.55)(453.6,171.84) -- (453.37,185.05) ;
			\draw   (480.46,111.34) -- (492.51,111.55)(486.6,104.84) -- (486.37,118.05) ;
			\draw  [draw opacity=0] (360.3,175.89) .. controls (367.93,153.35) and (381.8,133.02) .. (401.63,117.67) .. controls (443.19,85.5) and (499.39,84.19) .. (544.78,109.67) -- (486.08,226.78) -- cycle ; \draw   (360.3,175.89) .. controls (367.93,153.35) and (381.8,133.02) .. (401.63,117.67) .. controls (443.19,85.5) and (499.39,84.19) .. (544.78,109.67) ;
			\draw    (453.89,64.08) -- (452.99,178.26) ;
			\draw  [draw opacity=0] (374.44,292.38) .. controls (368.39,282.27) and (363.69,271.71) .. (360.32,260.95) -- (482.96,231.08) -- cycle ; \draw   (374.44,292.38) .. controls (368.39,282.27) and (363.69,271.71) .. (360.32,260.95) ;
			\draw   (448.71,89.51) -- (458.84,98.77)(458.51,88.96) -- (449.04,99.32) ;
			\draw    (453.44,94.32) -- (487.17,111.33) ;
			\draw  [draw opacity=0] (481.56,121.47) .. controls (481.36,121.38) and (481.16,121.29) .. (480.97,121.2) .. controls (475.85,118.61) and (473.83,112.3) .. (476.45,107.09) .. controls (476.53,106.94) and (476.62,106.79) .. (476.7,106.64) -- (485.72,111.78) -- cycle ; \draw   (481.56,121.47) .. controls (481.36,121.38) and (481.16,121.29) .. (480.97,121.2) .. controls (475.85,118.61) and (473.83,112.3) .. (476.45,107.09) .. controls (476.53,106.94) and (476.62,106.79) .. (476.7,106.64) ;
			\draw  [color={rgb, 255:red, 0; green, 0; blue, 0 }  ][line width=1.5] [line join = round][line cap = round] (481.28,113.57) .. controls (481.28,113.57) and (481.28,113.57) .. (481.28,113.57) ;
			\draw  [draw opacity=0][line width=1.5]  (418.37,106.76) .. controls (440.95,94.56) and (466.39,90.01) .. (491.47,92.62) -- (486.08,226.78) -- cycle ; \draw  [color={rgb, 255:red, 208; green, 2; blue, 27 }  ,draw opacity=1 ][line width=1.5]  (418.37,106.76) .. controls (440.95,94.56) and (466.39,90.01) .. (491.47,92.62) ;
			\draw [color={rgb, 255:red, 208; green, 2; blue, 27 }  ,draw opacity=1 ] [dash pattern={on 0.84pt off 2.51pt}]  (491.47,92.62) -- (486.08,226.78) ;
			\draw [color={rgb, 255:red, 208; green, 2; blue, 27 }  ,draw opacity=1 ] [dash pattern={on 0.84pt off 2.51pt}]  (418.37,106.76) -- (486.08,226.78) ;
			\draw  [draw opacity=0] (468.63,194.88) .. controls (474.01,191.5) and (480.52,189.52) .. (487.54,189.52) .. controls (487.55,189.52) and (487.56,189.52) .. (487.57,189.52) -- (487.54,220.26) -- cycle ; \draw  [color={rgb, 255:red, 208; green, 2; blue, 27 }  ,draw opacity=1 ] (468.63,194.88) .. controls (474.01,191.5) and (480.52,189.52) .. (487.54,189.52) .. controls (487.55,189.52) and (487.56,189.52) .. (487.57,189.52) ;
			\draw [color={rgb, 255:red, 0; green, 0; blue, 0 }  ,draw opacity=1 ] [dash pattern={on 0.84pt off 2.51pt}]  (510.27,187.04) -- (484.48,196.93) ;
			\draw  [dash pattern={on 0.84pt off 2.51pt}]  (434.47,191.69) -- (452.99,178.26) ;
			\draw  [dash pattern={on 0.84pt off 2.51pt}]  (466.47,240.69) -- (486.08,226.78) ;
			\draw   (431.98,194.15) .. controls (428.06,196.69) and (427.37,199.92) .. (429.91,203.83) -- (436.07,213.32) .. controls (439.7,218.91) and (439.56,222.97) .. (435.64,225.52) .. controls (439.56,222.97) and (443.33,224.5) .. (446.96,230.09)(445.33,227.58) -- (453.79,240.62) .. controls (456.33,244.54) and (459.56,245.23) .. (463.48,242.69) ;
			\draw [color={rgb, 255:red, 208; green, 2; blue, 27 }  ,draw opacity=1 ] [dash pattern={on 4.5pt off 4.5pt}]  (453.44,94.32) -- (486.08,226.78) ;
			\draw  [dash pattern={on 4.5pt off 4.5pt}]  (378,64) -- (397.39,64.08) ;
			\draw  [dash pattern={on 4.5pt off 4.5pt}]  (510.39,64.08) -- (529.78,64.15) ;
			\draw [color={rgb, 255:red, 255; green, 255; blue, 255 }  ,draw opacity=1 ][line width=5.25]    (560.72,-4.59) -- (562.86,304.28) ;
			\draw [color={rgb, 255:red, 208; green, 2; blue, 27 }  ,draw opacity=1 ] [dash pattern={on 0.84pt off 2.51pt}]  (100.55,53.95) -- (112.66,56.67) ;
			\draw [color={rgb, 255:red, 208; green, 2; blue, 27 }  ,draw opacity=1 ] [dash pattern={on 0.84pt off 2.51pt}]  (112.66,56.67) -- (193.99,199.79) ;
			\draw [color={rgb, 255:red, 208; green, 2; blue, 27 }  ,draw opacity=1 ] [dash pattern={on 0.84pt off 2.51pt}]  (355.07,60.67) -- (559.07,42.67) ;
			\draw [color={rgb, 255:red, 208; green, 2; blue, 27 }  ,draw opacity=1 ] [dash pattern={on 0.84pt off 2.51pt}]  (355.07,60.67) -- (274.66,200.21) ;
			\draw [color={rgb, 255:red, 255; green, 255; blue, 255 }  ,draw opacity=1 ][line width=5.25]    (101.28,-8.33) -- (103.66,305.69) ;
			\draw  [draw opacity=0][dash pattern={on 0.84pt off 2.51pt}] (274.66,200.21) .. controls (267.2,215.31) and (252.23,225.69) .. (234.88,225.86) .. controls (217.06,226.03) and (201.54,215.4) .. (193.99,199.79) -- (234.41,178.38) -- cycle ; \draw  [color={rgb, 255:red, 208; green, 2; blue, 27 }  ,draw opacity=1 ][dash pattern={on 0.84pt off 2.51pt}] (274.66,200.21) .. controls (267.2,215.31) and (252.23,225.69) .. (234.88,225.86) .. controls (217.06,226.03) and (201.54,215.4) .. (193.99,199.79) ;
			\draw   (219.98,201.15) .. controls (216.75,204.52) and (216.82,207.81) .. (220.19,211.04) -- (220.19,211.04) .. controls (225,215.65) and (225.8,219.64) .. (222.57,223.01) .. controls (225.8,219.64) and (229.82,220.26) .. (234.64,224.87)(232.47,222.8) -- (237.74,227.85) .. controls (241.11,231.08) and (244.41,231) .. (247.64,227.63) ;
			\draw  [dash pattern={on 4.5pt off 4.5pt}] (188.5,124.55) -- (278.5,124.55) -- (278.5,248.52) -- (188.5,248.52) -- cycle ;
			\draw  [dash pattern={on 4.5pt off 4.5pt}]  (279,186) -- (361.36,185.58) ;
			\draw [shift={(320.18,185.79)}, rotate = 539.71] [color={rgb, 255:red, 0; green, 0; blue, 0 }  ][line width=0.75]    (0,5.59) -- (0,-5.59)(-5.03,5.59) -- (-5.03,-5.59)   ;
			\draw [color={rgb, 255:red, 0; green, 0; blue, 0 }  ,draw opacity=1 ] [dash pattern={on 0.84pt off 2.51pt}]  (516.81,130.52) -- (487.17,111.33) ;
			\draw [color={rgb, 255:red, 0; green, 0; blue, 0 }  ,draw opacity=1 ] [dash pattern={on 0.84pt off 2.51pt}]  (453.44,94.32) -- (394.59,98) ;
			\draw [color={rgb, 255:red, 0; green, 0; blue, 0 }  ,draw opacity=1 ] [dash pattern={on 0.84pt off 2.51pt}]  (428.59,103) -- (407.59,156) ;
			\draw [color={rgb, 255:red, 0; green, 0; blue, 0 }  ,draw opacity=1 ] [dash pattern={on 0.84pt off 2.51pt}]  (284.17,155.33) -- (253.17,147) ;
			\draw [color={rgb, 255:red, 0; green, 0; blue, 0 }  ,draw opacity=1 ] [dash pattern={on 0.84pt off 2.51pt}]  (233.44,138.32) -- (167.27,145.81) ;
			\draw   (480.46,226.34) -- (492.51,226.55)(486.6,219.84) -- (486.37,233.05) ;
			\draw    (233.44,138.32) -- (253.17,147) ;
			\draw  [draw opacity=0][dash pattern={on 0.84pt off 2.51pt}] (544.84,185.52) .. controls (525.59,222.58) and (491.19,247.4) .. (451.73,247.79) .. controls (414.44,248.15) and (381.25,226.61) .. (360.85,193.31) -- (450.44,116.41) -- cycle ; \draw  [color={rgb, 255:red, 208; green, 2; blue, 27 }  ,draw opacity=1 ][dash pattern={on 0.84pt off 2.51pt}] (544.84,185.52) .. controls (525.59,222.58) and (491.19,247.4) .. (451.73,247.79) .. controls (414.44,248.15) and (381.25,226.61) .. (360.85,193.31) ;
			\draw [color={rgb, 255:red, 0; green, 0; blue, 0 }  ,draw opacity=1 ] [dash pattern={on 0.84pt off 2.51pt}]  (207.44,266.32) -- (226.17,266.67) ;
			
			\draw (264.17,214.33) node [anchor=north west][inner sep=0.75pt]    {$x$};
			\draw (235.5,189.54) node [anchor=north west][inner sep=0.75pt]  [font=\large]   {$\overline{x}$};
			\draw (204,225) node [anchor=north west][inner sep=0.75pt]  [font=\large]  {$r( x)$};
			\draw (244.77,7.77) node [anchor=north west][inner sep=0.75pt]  [font=\large]   {$\xi _{\overline{x}}$};
			\draw (115,108) node [anchor=north west][inner sep=0.75pt] [font=\large]    {$h$};
			\draw (312,75) node [anchor=north west][inner sep=0.75pt]   [font=\large]  {$\theta $};
			\draw (288.5,151) node [anchor=north west][inner sep=0.75pt]   [font=\large]  {$\tilde{s}_{x,\overline{x}}$};
			\draw (145,140) node [anchor=north west][inner sep=0.75pt]   [font=\large]  {$s_{x,\overline{x}}$};
			\draw  [draw opacity=0]  (134, 271) circle [x radius= 19.85, y radius= 19.85]   ;
			\draw (123,259) node [anchor=north west][inner sep=0.75pt]  [font=\huge]  {$\Omega $};
			\draw (387,25) node [anchor=north west][inner sep=0.75pt]  [font=\large,color={rgb, 255:red, 0; green, 0; blue, 0 }  ,opacity=1 ]  {$h_{0} =\frac{h}{4}$};
			\draw (488.08,229.78) node [anchor=north west][inner sep=0.75pt]   [font=\large]  {$x$};
			\draw (454.99,181.26) node [anchor=north west][inner sep=0.75pt]   [font=\large]  {$\overline{x}$};
			\draw (372,93) node [anchor=north west][inner sep=0.75pt]   [font=\large]  {$s_{x,\overline{x}}$};
			\draw (520.17,128) node [anchor=north west][inner sep=0.75pt]   [font=\large]  {$\tilde{s}_{x,\overline{x}}$};
			\draw (513,180) node [anchor=north west][inner sep=0.75pt]  [font=\large,color={rgb, 255:red, 0; green, 0; blue, 0 }  ,opacity=1 ]  {$\gamma _{\theta }$};
			\draw (413,222) node [anchor=north west][inner sep=0.75pt]   [font=\large]  {$r( x)$};
			\draw (382,159) node [anchor=north west][inner sep=0.75pt]   [font=\large]  {$\mathcal{K}( x,\overline{x})$};
			\draw (228,259) node [anchor=north west][inner sep=0.75pt]   [font=\large]  {$\partial B_{2r( x)}( x)$};
			\end{tikzpicture}
			\caption{Sketch of the construction in two dimensions.\\[3mm]}
		\end{figure}
	We set $h_0:=h_0(\Omega):=\frac{h}{4}$ and fix an arbitrary $x\in\Omega$ 
		with $r(x)\leq h_0$. As $\pa\Omega$ 
		is compact there exists $\overline{x}\in\pa\Omega$, such that
		\begin{align}
		\vert x-\overline{x}\vert =r(x).\label{eq:pp1}
		\end{align}
		The uniform exterior cone property yields some 
		$\xi_{\overline{x}}\in \mathbb{S}^{d-1}$ 
		such that $\overline{x}+\mathcal{C}(\xi_{\overline{x}},\theta,h)\subseteq \Omega^c$. 
		As $\Omega^c$ is closed, we even have
		\begin{align*}
		\overline{x}+\overline{\mathcal{C}(\xi_{\overline{x}},\theta,h)}\subseteq \Omega^c.
		\end{align*} 
		Due to $\vert x-\overline{x}\vert =r(x)\leq h_0=\frac{h}{4}$, there holds
		\begin{align*}
		\vert (\overline{x}+h\xi_{\overline{x}})-x\vert\ge h-\vert x-\overline{x}\vert\ge h-h_0
		= \frac{3h}{4}\ge 3h_0\ge 3r(x),
		\end{align*}
		i.e., $\overline{x}+h\xi_{\overline{x}}\notin B^d_{2r(x)}(x)$. 
		But since $\overline{x}\in B^d_{2r(x)}(x)$ (cf.~\eqref{eq:pp1}) 
		the sphere $\pa B^d_{2r(x)}(x)$ intersects the axis 
		$\overline{x}+\left(0,h\right)\xi_{\overline{x}}$ non-trivially, i.e., there exists some
		\begin{align}
		s_{x,\overline{x}}\in\pa B^d_{2r(x)}(x)\cap(\overline{x}+\left(0,h\right)\xi_{\overline{x}}).\label{eq:pp4}
		\end{align}
		Next, we want to show that $\pa B^d_{2r(x)}(x)\cap(\overline{x}+\mathcal{C}(\xi_{\overline{x}},\theta,h))$ 
		contains a hyperspherical cap of the desired size. 
		Due to $s_{x,\overline{x}}\in \pa B^d_{2r(x)}(x)$ 
		(cf.~\eqref{eq:pp4}) and $\overline{x}\in \pa B^d_{r(x)}(x)$ 
		(cf.~\eqref{eq:pp1}), there holds
		\begin{align}
		r(x)\leq \vert s_{x,\overline{x}}-\overline{x}\vert
		\leq \vert s_{x,\overline{x}}-x\vert+\vert x-\overline{x}\vert= 
		3	r(x).\label{eq:pp5}
		\end{align}
		Denote by $\mathcal{L}(\xi_{\overline{x}},\theta,h):=
		\{y\in \setR^d\fdg y\cdot\xi_{\overline{x}}=\vert y\vert\cos(\theta), y_d\leq h\}$ 
		the lateral surface of $\mathcal{C}(\xi_{\overline{x}},\theta,h)$.
		From the compactness of $\overline{x}+\mathcal{L}(\xi_{\overline{x}},\theta,h)$ 
		we obtain some point 
		$\tilde{s}_{x,\overline{x}}\in \overline{x}+\mathcal{L}(\xi_{\overline{x}},\theta,h)$, 
		such that 
		\begin{align*}
		\vert s_{x,\overline{x}}-\tilde{s}_{x,\overline{x}}\vert=
		\text{dist}(s_{x,\overline{x}},\overline{x}+\mathcal{L}(\xi_{\overline{x}},\theta,h)).
		\end{align*}
		Then, considering the right-angled triangle 
		$\Updelta (s_{x,\overline{x}},\overline{x},\tilde{s}_{x,\overline{x}})$, 
		with $\angle (s_{x,\overline{x}}-\tilde{s}_{x,\overline{x}}, \tilde{s}_{x,\overline{x}}-\overline{x})=\frac{\pi}{2}$
		and $\angle (s_{x,\overline{x}}-\overline{x}, \tilde{s}_{x,\overline{x}}-\overline{x})=\theta $, 
		also using \eqref{eq:pp5}, one readily sees that
		\begin{align}
		\begin{split}
		\text{dist}(s_{x,\overline{x}},\overline{x}+\mathcal{L}(\xi_{\overline{x}},\theta,h))&=\vert s_{x,\overline{x}}-\tilde{s}_{x,\overline{x}}\vert
		=\sin(\angle (s_{x,\overline{x}}-\overline{x}, \tilde{s}_{x,\overline{x}}-\overline{x}))
		\vert s_{x,\overline{x}}-\overline{x}\vert
		\\&=\sin(\theta)\vert s_{x,\overline{x}}-\overline{x}\vert
		\ge \sin(\theta)r(x).
		\end{split}\label{eq:pp7}
		\end{align}
		Denote by $\mathcal{B}(\xi_{\overline{x}},\theta,h):=
		\{y\in \setR^d\fdg y\cdot\xi_{\overline{x}}\ge\vert y\vert\cos(\theta), y_d= h\}$ 
		the basis of the cone $\mathcal{C}(\xi_{\overline{x}},\theta,h)$.  
		Due to \eqref{eq:pp5}, there holds
		\begin{align}
		\begin{split}
		\text{dist}(s_{x,\overline{x}},\overline{x}+\mathcal{B}(\xi_{\overline{x}},\theta,h))
		\ge h- \vert s_{x,\overline{x}}-\overline{x}\vert\ge h-3r(x)\ge h-\frac{3h}{4}=\frac{h}{4}\ge r(x).
		\end{split}\label{eq:pp7b}
		\end{align}
		Since $\pa\mathcal{C}(\xi_{\overline{x}},\theta,h)
		=\mathcal{L}(\xi_{\overline{x}},\theta,h)\cup \mathcal{B}(\xi_{\overline{x}},\theta,h)$, 
		\eqref{eq:pp7} and \eqref{eq:pp7b} imply
		\begin{align}
		\text{dist}(s_{x,\overline{x}},\overline{x}+\pa\mathcal{C}(\xi_{\overline{x}},\theta,h))
		\ge \sin(\theta)r(x).\label{eq:pp7c}
		\end{align}
		Denote by $\mathcal{K}(x,\overline{x}):=\{y\in \pa B^d_{2r(x)}(x)
		\fdg\angle(y-x,s_{x,\overline{x}}-x)\leq \gamma_\theta\}$ the hyperspherical~cap~of~$\pa B^d_{2r(x)}(x)$ 
		with centre $s_{x,\overline{x}}$ and radius $\nu_\theta2r(x)>0$, 
		where $\nu_\theta:=\sin(\gamma_\theta)>0$ and 
		$\gamma_\theta=2\arcsin(\frac{\sin(\theta)}{4})>0$ 
		denotes the opening angle. Then, according to \cite{Li11} there exists a 
		constant $c_1:=c_1(d,\gamma_\theta)>0$, depending only 
		on the dimension $d\in \setN$ and the opening angle $\gamma_\theta>0$, such that
		\begin{align}
		\mathscr{H}^{d-1}(\mathcal{K}(x,\overline{x}))=c_1(2r(x))^{d-1}.\label{eq:pp9}
		\end{align}
		It remains to show that 
		$\mathcal{K}(x,\overline{x})\subseteq \pa B^d_{2r(x)}(x)\cap \Omega^c$. 
		Since for every $y\in K_{x,\overline{x}}\subseteq \pa B_{2r(x)}(x)$ 
		the triangle $\Updelta(y,s_{x,\overline{x}},x)$ is isosceles, as 
		$\vert s_{x,\overline{x}}-x\vert =\vert y-x\vert =2r(x)$, 
		also using \eqref{eq:pp7c}, we infer that
		\begin{align*}
		\begin{split}
		\vert s_{x,\overline{x}}-y\vert &= 4r(x)
		\sin\left(\frac{\angle(s_{x,\overline{x}}-x,y-x)}{2}\right)
		\leq 4r(x)\sin\left(\frac{\gamma_{\theta}}{2}\right)
		\\&=r(x)\sin(\theta)
		\leq\text{dist}(s_{x,\overline{x}},\overline{x}
		+\pa\mathcal{C}(\xi_{\overline{x}},\theta,h)),
		\end{split}
		\end{align*}
		i.e., $\mathcal{K}(x,\overline{x})
		\subseteq \overline{B^d_{\text{dist}(s_{x,\overline{x}},\overline{x}
				+\pa\mathcal{C}(\xi_{\overline{x}},\theta,h))}(s_{x,\overline{x}})}
		\subseteq \overline{x}+\overline{\mathcal{C}(\xi_{\overline{x}},\theta,h)}
		\subseteq \Omega^c$ (cf.~\eqref{eq:pp1}) and 
		\begin{align}
		\mathcal{K}(x,\overline{x})\subseteq B^d_{2r(x)}(x)\cap\Omega^c.\label{eq:pp8}
		\end{align}
		
		\quad\textbf{Step 2:} This step shows that it suffices to treat the case ${\bfu\in C_0^{\infty}(\Omega)^d}$. In fact, assume
		that there exist constants $c_0=c_0(d,\Omega),h_0=h_0(\Omega)>0$, such that 
		for every $\bfu\in C_0^{\infty}(\Omega)^d$ and 
		$x\in \Omega$ with $r(x)\leq h_0$ 
		there holds
		\begin{align}
		\vert \bfu(x)\vert \leq c_0\int_{B^d_{2r(x)}(x)\cap \Omega}
		{\frac{\vert \bfvarepsilon(\bfu)(y)\vert}{\vert x-y\vert^{d-1}}\,dy}.\label{eq:approx.0}
		\end{align}
		Next, let ${\bfu\in X^{1,1}_{\bfvarepsilon}(\Omega)}$. By definition there exists a sequence $(\bfu_n)_{n\in \mathbb{N}}\subseteq  C_0^\infty(\Omega)^d$, 
		such that ${\bfu_n\to \bfu}$ in $X^{1,1}_{\bfvarepsilon}(\Omega)$ $(n\to\infty)$. 
		Without loss of generality we may assume that  
		$\bfu_n\to \bfu$ $(n\to\infty)$  almost everywhere in $\Omega$. On the other hand, defining for almost every $x\in \Omega$ the measurable functions $f_n(x):=\int_{B^d_{2r(x)}(x)\cap \Omega}
		{\frac{\vert \bfvarepsilon(\bfu_n)(y)\vert}{\vert x-y\vert^{d-1}}\,dy}$, $n\in \mathbb{N}$,  and $f(x):=\int_{B^d_{2r(x)}(x)\cap \Omega}
		{\frac{\vert \bfvarepsilon(\bfu)(y)\vert}{\vert x-y\vert^{d-1}}\,dy}$, we deduce by the weak-type $L^1$-estimate for the Riesz potential operator (cf.~\cite[Thm. 5.1.1]{St70}), that there exists a constant $c>0$, not depending on $n\in \mathbb{N}$, such that
		for every $\lambda >0$ there holds
		\begin{align*}
			\vert \{\vert f_n-f\vert>\lambda\}\vert \leq \frac{c}{\lambda}\|\bfvarepsilon(\bfu_n)-\bfvarepsilon(\bfu_n)\|_{L^1(\Omega)^{d\times d}}\overset{n\to\infty}{\to} 0,
		\end{align*}
		i.e., $ f_n\to f$ $(n\to \infty)$ in measure in $\Omega$. Thus, since $\vert \Omega\vert<\infty$, there exists a not relabelled subsequence, such that
		$f_n\to f$ $(n\to \infty)$ almost everywhere in $\Omega$. Putting everything together, since the sequence $(\bfu_n)_{n\in \mathbb{N}}\subseteq  C_0^\infty(\Omega)^d$ satisfies \eqref{eq:approx.0} for every $x\in \Omega$ with $r(x)\leq h_0$  and $n\in \mathbb{N}$, we conclude by passing for $n\to\infty$ in \eqref{eq:approx.0}, that $\bfu\in X^{1,1}_{\bfvarepsilon}(\Omega)$ satisfies \eqref{eq:approx.0} for almost every  $x\in \Omega$ with $r(x)\leq h_0$.\\[-3mm]

		\quad\textbf{Step 3:} Now we prove \eqref{eq:poin}. Due to step \textbf{2}, it suffices~to~treat~${\bfu\in C_0^\infty(\Omega)^d}$. 
		We fix $x\in \Omega$ with 
		$r(x)\leq h_0$ and consider functions 
		$\bfPhi_i=(\Phi_i^1,...,\Phi_i^d)^\top:\setR^{d-1}\to
                \mathbb{S}^{d-1}\subseteq \setR^d$, $i=1,...,d$, for every $\eta=(\eta_1,...,\eta_{d-1})^\top\in \setR^{d-1}$ and $j=1,...,d$
		defined as 
		\begin{align*}
		\Phi_d^j(\eta):=\begin{cases}
		\frac{1}{\sqrt{\vert\eta\vert^2+1}}&\text{ for }j=d\\
		\frac{\eta_j}{\sqrt{\vert\eta\vert^2+1}}&\text{ else }
		\end{cases}\qquad\text{ and }\qquad\Phi_i^j(\eta):=\begin{cases}\frac{-\eta_j}{\sqrt{\vert\eta\vert^2+1}}&\text{ for }j=i\\
		\frac{1}{\sqrt{\vert\eta\vert^2+1}}&\text{ for }j=d\\
		\frac{\eta_j}{\sqrt{\vert\eta\vert^2+1}}&\text{ else }
		\end{cases}.
		\end{align*}
		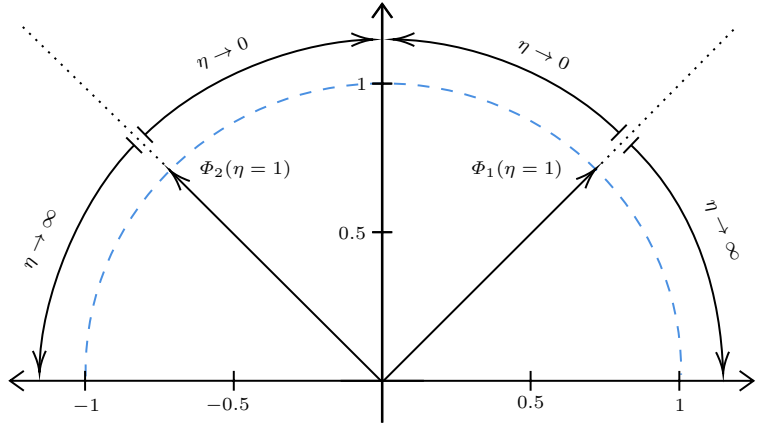
\begin{figure}[hbt!]
			\centering
			\tikzset{every picture/.style={line width=0.75pt}} 
			
			\begin{tikzpicture}[x=0.75pt,y=0.75pt,yscale=-1,xscale=1]
			
			\draw  [draw opacity=0][dash pattern={on 4.5pt off 4.5pt}] (55.92,201.98) .. controls (55.91,201.3) and (55.9,200.62) .. (55.9,199.94) .. controls (55.49,119.83) and (122.39,54.55) .. (205.33,54.13) .. controls (288.27,53.71) and (355.83,118.31) .. (356.23,198.42) .. controls (356.24,199.34) and (356.23,200.25) .. (356.22,201.17) -- (206.06,199.18) -- cycle ; \draw  [color={rgb, 255:red, 74; green, 144; blue, 226 }  ,draw opacity=1 ][dash pattern={on 4.5pt off 4.5pt}] (55.92,201.98) .. controls (55.91,201.3) and (55.9,200.62) .. (55.9,199.94) .. controls (55.49,119.83) and (122.39,54.55) .. (205.33,54.13) .. controls (288.27,53.71) and (355.83,118.31) .. (356.23,198.42) .. controls (356.24,199.34) and (356.23,200.25) .. (356.22,201.17) ;
			\draw    (205,204.37) -- (99.79,99.45) ;
			\draw [shift={(98.37,98.03)}, rotate = 404.91999999999996] [color={rgb, 255:red, 0; green, 0; blue, 0 }  ][line width=0.75]    (10.93,-3.29) .. controls (6.95,-1.4) and (3.31,-0.3) .. (0,0) .. controls (3.31,0.3) and (6.95,1.4) .. (10.93,3.29)   ;
			\draw    (324.87,78.99) .. controls (292.82,47.5) and (247.43,32.64) .. (212.37,31.94) ;
			\draw [shift={(210.77,31.92)}, rotate = 360.44] [color={rgb, 255:red, 0; green, 0; blue, 0 }  ][line width=0.75]    (10.93,-3.29) .. controls (6.95,-1.4) and (3.31,-0.3) .. (0,0) .. controls (3.31,0.3) and (6.95,1.4) .. (10.93,3.29)   ;
			\draw [shift={(324.87,78.99)}, rotate = 404.49] [color={rgb, 255:red, 0; green, 0; blue, 0 }  ][line width=0.75]    (0,5.59) -- (0,-5.59)   ;
			\draw [line width=0.75]    (206,204.37) -- (311.96,98.45) ;
			\draw [shift={(313.37,97.03)}, rotate = 495.01] [color={rgb, 255:red, 0; green, 0; blue, 0 }  ][line width=0.75]    (10.93,-3.29) .. controls (6.95,-1.4) and (3.31,-0.3) .. (0,0) .. controls (3.31,0.3) and (6.95,1.4) .. (10.93,3.29)   ;
			\draw [line width=0.75]  [dash pattern={on 0.84pt off 2.51pt}]  (98.63,98.07) -- (22.63,24.07) ;
			\draw    (80.36,85.34) .. controls (52.85,113.26) and (32.97,162.82) .. (32.63,198.45) ;
			\draw [shift={(32.63,200.07)}, rotate = 269.57] [color={rgb, 255:red, 0; green, 0; blue, 0 }  ][line width=0.75]    (10.93,-3.29) .. controls (6.95,-1.4) and (3.31,-0.3) .. (0,0) .. controls (3.31,0.3) and (6.95,1.4) .. (10.93,3.29)   ;
			\draw [shift={(80.36,85.34)}, rotate = 314.58000000000004] [color={rgb, 255:red, 0; green, 0; blue, 0 }  ][line width=0.75]    (0,5.59) -- (0,-5.59)   ;
			\draw    (85.87,79.99) .. controls (120.17,46.67) and (168.6,32.84) .. (199.01,32.02) ;
			\draw [shift={(200.85,31.98)}, rotate = 539.4200000000001] [color={rgb, 255:red, 0; green, 0; blue, 0 }  ][line width=0.75]    (10.93,-3.29) .. controls (6.95,-1.4) and (3.31,-0.3) .. (0,0) .. controls (3.31,0.3) and (6.95,1.4) .. (10.93,3.29)   ;
			\draw [shift={(85.87,79.99)}, rotate = 495.83] [color={rgb, 255:red, 0; green, 0; blue, 0 }  ][line width=0.75]    (0,5.59) -- (0,-5.59)   ;
			\draw  (184.5,204.24) -- (392.73,204.24)(205.32,14.04) -- (205.32,225.37) (385.73,199.24) -- (392.73,204.24) -- (385.73,209.24) (200.32,21.04) -- (205.32,14.04) -- (210.32,21.04) (280.32,199.24) -- (280.32,209.24)(355.32,199.24) -- (355.32,209.24)(200.32,129.24) -- (210.32,129.24)(200.32,54.24) -- (210.32,54.24) ;
			\draw   ;
			\draw  (226.5,204.22) -- (17.73,204.22)(205.62,14.04) -- (205.62,225.35) (24.73,199.22) -- (17.73,204.22) -- (24.73,209.22) (210.62,21.04) -- (205.62,14.04) -- (200.62,21.04) (130.62,199.22) -- (130.62,209.22)(55.62,199.22) -- (55.62,209.22)(210.62,129.22) -- (200.62,129.22)(210.62,54.22) -- (200.62,54.22) ;
			\draw   ;
			\draw    (331.01,85.31) .. controls (366.25,119.14) and (376.13,164.24) .. (377.38,199.03) ;
			\draw [shift={(377.43,200.61)}, rotate = 268.36] [color={rgb, 255:red, 0; green, 0; blue, 0 }  ][line width=0.75]    (10.93,-3.29) .. controls (6.95,-1.4) and (3.31,-0.3) .. (0,0) .. controls (3.31,0.3) and (6.95,1.4) .. (10.93,3.29)   ;
			\draw [shift={(331.01,85.31)}, rotate = 223.82999999999998] [color={rgb, 255:red, 0; green, 0; blue, 0 }  ][line width=0.75]    (0,5.59) -- (0,-5.59)   ;
			\draw [line width=0.75]  [dash pattern={on 0.84pt off 2.51pt}]  (312.87,97.76) -- (384.72,24.94) ;
			
			\draw (249,91.37) node [anchor=north west][inner sep=0.75pt]  [font=\scriptsize]  {$\Phi_1( \eta =1)$};
			\draw (112,91.37) node [anchor=north west][inner sep=0.75pt]  [font=\scriptsize]  {$\Phi_{2}( \eta =1)$};
			\draw (108.77,40.22) node [anchor=north west][inner sep=0.75pt]  [font=\scriptsize,rotate=-337.34]  {$\eta \rightarrow 0$};
			\draw (275.2,28.24) node [anchor=north west][inner sep=0.75pt]  [font=\scriptsize,rotate=-25.64]  {$\eta \rightarrow 0$};
			\draw (22.71,147.69) node [anchor=north west][inner sep=0.75pt]  [font=\scriptsize,rotate=-293.68]  {$\eta \rightarrow \infty $};
			\draw (48,211.98) node [anchor=north west][inner sep=0.75pt]  [font=\scriptsize]  {$-1$};
			\draw (191,50) node [anchor=north west][inner sep=0.75pt]  [font=\scriptsize]  {$1$};
			\draw (115,211.35) node [anchor=north west][inner sep=0.75pt]  [font=\scriptsize]  {$-0.5$};
			\draw (182,125) node [anchor=north west][inner sep=0.75pt]  [font=\scriptsize]  {$0.5$};
			\draw (272,210.35) node [anchor=north west][inner sep=0.75pt]  [font=\scriptsize]  {$0.5$};
			\draw (373.11,110.49) node [anchor=north west][inner sep=0.75pt]  [font=\scriptsize,rotate=-65.3]  {$\eta \rightarrow \infty $};
			\draw (351.58,211.98) node [anchor=north west][inner sep=0.75pt]  [font=\scriptsize]  {$1$};

			\end{tikzpicture}
				\caption{Behavior of $\boldsymbol{\Phi}_i:\mathbb{R}^{d-1}\to \mathbb{S}^{d-1}\subseteq \mathbb{R}^d$, $i=1,...,d$, for $d=2$.}
				\label{fig2}
		\end{figure}%
		It is not difficult to see, that the functions 
		$\bfPhi_i:\setR^{d-1}\to \setR^d$, $i=1,...,d$, 
		are smooth and injective. Moreover, their derivatives are for every
		$\eta\in \setR^{d-1}$ given via
		\begin{align*}
		(\textsf{D}\bfPhi_d)(\eta)=\frac{1}{(\vert\eta\vert^2+1)^{\frac{3}{2}}}\begin{pmatrix}
		\sum_{\substack{j=2}}^{d-1}{\eta_j^2}+1 & -\eta_1\eta_2 & \dots & -\eta_1\eta_{d-2} & -\eta_1\eta_{d-1}\\
		-\eta_2\eta_1 &		\sum_{\substack{j=1\\j\neq 2}}^{d-1}{\eta_j^2}+1 & 		-\eta_2\eta_3 & \dots & -\eta_2\eta_{d-1}
		\\
		\vdots & \ddots&  & 
		\\
		-\eta_{d-1}\eta_1 & -\eta_{d-1}\eta_2 & \dots & 	-\eta_{d-1}\eta_{d-2} & 		\sum_{j=1}^{d-2}{\eta_j^2}+1 
		\\
		-\eta_1 & -\eta_2 & \dots & -\eta_{d-2} & -\eta_{d-1}
		\end{pmatrix},
		\end{align*}
		i.e., $(\textsf{D}\bfPhi_d)(0)^\top=(\mathbb{I}_{d-1},\textbf{0})\in \setR^{(d-1)\times d}$ 
		and thus $\det((\textsf{D}\bfPhi_d)(0)^\top (\textsf{D}\bfPhi_d)(0))=1$. 
		Since \linebreak${(\textsf{D}\bfPhi_i)(\eta)=\bfE_i^\top (\textsf{D}\bfPhi_d)(\eta)}$ 
		for $\bfE_i=(\mathbf{e}_1,...,\mathbf{e}_{i-1},-\mathbf{e}_i,\mathbf{e}_{i+1},...,\mathbf{e}_d)\in \setR^{d\times d}$, $i=1,...,d$, ~we~even~have 
		\begin{align*}
		\det((\textsf{D}\bfPhi_i)(0)^\top (\textsf{D}\bfPhi_i)(0))
		=\det((\textsf{D}\bfPhi_d)(0)^\top \bfE_i\bfE_i^\top (\textsf{D}\bfPhi_d)(0))
		=\det((\textsf{D}\bfPhi_d)(0)^\top (\textsf{D}\bfPhi_d)(0))=1
		\end{align*}
		for every $i=1,...,d$. As $\det:\setR^{(d-1)\times (d-1)}\to \setR$ 
		and $\textsf{D}\bfPhi_i^\top \textsf{D}\bfPhi_i:\setR^{d-1}\to \setR^{(d-1)\times (d-1)}$, 
		$i=1,...,d$, are continuous, there exists a constant 
		$\alpha_0>0$, such that for every $\eta\in B^{d-1}_{\alpha_0}(0)$
		\begin{align*}
		\det((\textsf{D}\bfPhi_i)(\eta)^\top (\textsf{D}\bfPhi_i)(\eta))\ge \frac{1}{2}.
		\end{align*}
		Thus, $\text{rk}((\textsf{D}\bfPhi_i)(\eta))\!=\!d-1$ for every $\eta\!\in\! B^{d-1}_{\alpha_0}(0)$ 
		and $i\!=\!1,...,d$, i.e., ${\bfPhi_i\!:\!B^{d-1}_{\alpha_0}(0)\!\to\! \bfPhi_i(B^{d-1}_{\alpha_0}(0))}$,
		$i=1,...,d$, are immersions into $\mathbb{S}^{d-1}$. In addition, if we set 
		${\boldsymbol{\Upphi}=(\bfPhi_1,...,\bfPhi_d)^\top :\setR^{d-1}\to \setR^{d\times d}}$, 
	one easily sees that $	\det(\boldsymbol{\Upphi}(\eta))\neq 0$ for every $\eta\in (\setR\setminus\{0\})^{d-1}$, see e.g. Figure \ref{fig2}.
		Hence, for every $\eta \in (\setR\setminus\{0\})^{d-1}$
                we have ${\boldsymbol{\Upphi}(\eta) \in \text{GL}(d,\setR)}$.
		 If we set for $\alpha\in \left(0,\alpha_0\right)$
                 \begin{align*}
		Q_\alpha:=\Big\{\eta=(\eta_1,...,\eta_{d-1})^\top 
		\in B^{d-1}_\alpha(0)\;\Big|\;\vert\eta_i\vert>\frac{\alpha}{2d}\text{ for }i=1,...,d-1\Big\},
		\end{align*} 
		then, as $\vert\det\vert:\text{GL}(d,\setR)\to \setR_{> 0}$ and $\boldsymbol{\Upphi}:\overline{Q_\alpha}\subset(\setR\setminus\{0\})^{d-1} \to \text{GL}(d,\setR)$ 
		are continuous and for each $\alpha\in \left(0,\alpha_0\right)$ the image 
		$(\vert\det\vert\circ\boldsymbol{\Upphi})(\overline{Q_\alpha})$ is a compact subset of $\setR_{>0}$, 
		there exists a constant $\mu_\alpha>0$, depending on $\alpha\in \left(0,\alpha_0\right)$, 
		such that for each $\alpha\in \left(0,\alpha_0\right)$ and for every $\eta\in Q_\alpha$  it holds
		\begin{align}
		\vert \det(\boldsymbol{\Upphi}(\eta))\vert\ge \mu_\alpha.\label{eq:pp13}
		\end{align}
		On the other hand, there holds for every $\eta\in
                Q_\alpha$, $\alpha\in \left(0,\alpha_0\right)$ and
                $i=1,\ldots, d$
		\begin{align*}
		\angle(\bfPhi_i(\eta), \bfe_d)
		=\arccos((\vert\eta\vert^2+1)^{-\frac{1}{2}})
		<\arccos((\alpha^2+1)^{-\frac{1}{2}})\,.
		\end{align*}
		Hence, the sets $\bfPhi_i(Q_\alpha)$, $i=1,...,d$,
                are contained in the hyperspherical cap of
                $\mathbb{S}^{d-1}$ with centre $\bfe_d$ and opening
                angle $\arccos((\alpha^2+1)^{-1/2})$. Thus, for
                $\alpha_1\in \left(0,\alpha_0\right)$ sufficiently small the sets
                $\bfPhi_i(Q_{\alpha_1})$, $i=1,...,d$, are contained
                in the hyperspherical cap of $\mathbb{S}^{d-1}$ with
                centre $\bfe_d$ and surface area greater than $c_1(2r(x))^{d-1}$ (cf.~step
                \textbf{1}). Rescaling and translating
                $\mathcal{K}(x,\overline{x})$, it follows from
                \eqref{eq:pp9} and \eqref{eq:pp8} that 
		\begin{align*}
		\frac{1}{2r(x)}(\mathcal{K}(x,\overline{x})-x)
		\subseteq \mathbb{S}^{d-1},\qquad 
		\mathscr{H}^{d-1}\bigg(\frac{1}{2r(x)}(\mathcal{K}(x,\overline{x})-x)\bigg)
		=c_1.
		\end{align*}
		Hence, we can find an orthogonal matrix $\bfS_x\in \mathbb{R}^{d\times d}$, such that for every~${i=1,...,d}$~there~holds
		\begin{align}
		\bfS_x(\bfPhi_i(Q_{\alpha_1}))\subseteq 
		\frac{1}{2r(x)}(\mathcal{K}(x,\overline{x})-x).\label{eq:pp12}
		\end{align}
		But \eqref{eq:pp12} in conjunction with \eqref{eq:pp8} implies for 
		$\bfPhi_i^x:=\bfS_x\circ\bfPhi_i:Q_{\alpha_1}\to \mathbb{S}^{d-1}$, $i=1,...,d$, that
		\begin{align}
		x+2r(x)\bfPhi_i^x(Q_{\alpha_0})&
		\subseteq \mathcal{K}(x,\overline{x})
		\subseteq \pa B_{2r(x)}(x)\cap \Omega^c.\label{eq:pp15}
		\end{align}
		Apart from this, we have for every $\eta\in Q_{\alpha_1}$, $i=1,...,d$  and 
		$t\in \left[0,2r(x)\right]$ 
		\begin{align}
		\begin{split}
		\frac{d}{dt}\left[\bfu(x+t\bfPhi^x_i(\eta))\cdot\bfPhi^x_i(\eta)\right]
		&=(\nb \bfu)(x+t\bfPhi^x_i(\eta)) : \bfPhi^x_i(\eta)\otimes\bfPhi^x_i(\eta)
		\\&=\bfvarepsilon(\bfu)(x+t\bfPhi^x_i(\eta)): \bfPhi^x_i(\eta)\otimes\bfPhi^x_i(\eta),
		\end{split}\label{eq:pp15.0}
		\end{align}
		where we used that
                $\bfPhi^x_i(\eta)\otimes\bfPhi^x_i(\eta) \in \mathbb
                M^{d \times d}_\sym$ for every $\eta\in Q_{\alpha_1}$ and  $i=1,...,d$. 
	     We integrate \eqref{eq:pp15.0} with respect to $t\in \left[0,2r(x)\right]$ to obtain for every $\eta\in Q_{\alpha_1}$ and $i=1,...,d$
		\begin{align}
		\begin{split}
		\bfu(x)\cdot\bfPhi^x_i(\eta)=\bfu(x+2r(x)\bfPhi^x_i(\eta))
		\cdot\bfPhi^x_i(\eta)-\int_0^{2r(x)}
		{\!\!\bfvarepsilon(\bfu)(x+s\bfPhi^x_i(\eta)): \bfPhi^x_i(\eta)\otimes
			\bfPhi^x_i(\eta)\,ds}.
		\end{split}\label{eq:pp15.1}
		\end{align}
		Hence, using \eqref{eq:pp15}, i.e., $\bfu(x+2r(x)\bfPhi^x_i(\eta))=\mathbf{0}$ for every $\eta\in Q_{\alpha_1}$ and $i=1,...,d$, we infer from \eqref{eq:pp15.1} for every $\eta\in Q_{\alpha_1}$ and $i=1,...,d$ that
		\begin{align}
		\vert \bfu(x)\cdot\bfPhi^x_i(\eta)\vert\leq\int_0^{2r(x)}
		{\vert\bfvarepsilon(\bfu)(x+s\bfPhi^x_i(\eta))\vert\,ds}.\label{eq:pp15.2}
		\end{align}
		Furthermore, setting $c_{\alpha_1}
		:=\|\boldsymbol{\Upphi}\|_{L^\infty(Q_{\alpha_1},\setR^{d\times d})}^{d-1}$, 
		there holds for every $\eta\in Q_{\alpha_1}$ 
		\begin{align}
		\vert\det(\boldsymbol{\Upphi}(\eta))\vert
		\leq \lambda_{\text{min}}(\boldsymbol{\Upphi}(\eta)^\top \boldsymbol{\Upphi}(\eta))\|\boldsymbol{\Upphi}(\eta)\|_{\setR^{d\times d}}^{d-1}
		\leq c_{\alpha_1}\lambda_{\text{min}}(\boldsymbol{\Upphi}(\eta)^\top \boldsymbol{\Upphi}(\eta)),\label{eq:pp15.3}
		\end{align}
		where $\lambda_{\text{min}}(\bfA)\in \setR$ denotes the smallest 
		eigenvalue of a matrix $\bfA\in\setR^{d\times d}$. 
		With this we deduce for every $\eta\in Q_{\alpha_1}$ that
		\begin{align}
		\sum_{i=1}^d{\vert \bfu(x)\cdot\bfPhi_i^x(\eta)\vert^2}&
		=\sum_{i=1}^d{\vert (\bfS_x^\top \bfu(x))\cdot \bfPhi_i(\eta)\vert^2}
		=\sum_{i=1}^d{\vert (\boldsymbol{\Upphi}(\eta)\bfS_x^\top \bfu(x))_i\vert^2}\label{eq:pp15.4}
		\\&=(\boldsymbol{\Upphi}(\eta)\bfS_x^\top \bfu(x))\cdot (\boldsymbol{\Upphi}(\eta)\bfS_x^\top \textbf{u}(x))
		\ge \lambda_{\text{min}}(\boldsymbol{\Upphi}(\eta)^\top \boldsymbol{\Upphi}(\eta))\vert \bfS_x^\top \bfu(x))\vert^2\notag
		\\&= \lambda_{\text{min}}(\boldsymbol{\Upphi}(\eta)^\top \boldsymbol{\Upphi}(\eta))\vert \bfu(x)\vert^2
		\ge \frac{1}{c_{\alpha_1}}\vert \text{det}(\boldsymbol{\Upphi}(\eta))\vert\vert \bfu(x)\vert^2
		\ge \frac{\mu_{\alpha_1}}{c_{\alpha_1}}\vert \bfu(x)\vert^2,\notag
		\end{align}
		where we exploited the min-max theorem of Courant and Fischer 
		in the first inequality, \eqref{eq:pp15.3} in the second inequality and \eqref{eq:pp13} in the last inequality.
		In consequence, also using \eqref{eq:pp15.2}, we infer from \eqref{eq:pp15.4} for every $\eta\in Q_{\alpha_1}$
		\begin{align*}
		\begin{split}
		\bigg[\frac{\mu_{\alpha_1}}{c_{\alpha_1}}\bigg]^{\frac{1}{2}}\vert \bfu(x)\vert\leq \bigg[\sum_{i=1}^d{\vert \bfu(x)\cdot\bfPhi_i^x(\eta)\vert^2}\bigg]^{\frac{1}{2}}
		\leq 
		\sum_{i=1}^d{\vert \bfu(x)\cdot\bfPhi_i^x(\eta)\vert}
		\leq   
		\sum_{i=1}^d{\int_0^{2r(x)}
			{\vert\bfvarepsilon(\bfu)(x+s\bfPhi^x_i(\eta))\vert\, ds}}.
		\end{split}
		\end{align*}
		We integrate with respect to $\eta\in Q_{\alpha_1}$, 
		divide by $\vert Q_{\alpha_1}\vert$, apply the 
		transformation theorem, which is allowed since 
		$\bfPhi_i^x:Q_{\alpha_1}\to \Phi_i^x(Q_{\alpha_1})$, 
		$i=1,...,d$, are immersions, and the so-called 
		''onion formula'' to obtain
		\begin{align}
		\begin{split}
		\bigg[\frac{\mu_{\alpha_1}}{c_{\alpha_1}}\bigg]^{\frac{1}{2}}\vert \bfu(x)\vert
		&\leq 
		\sum_{i=1}^d{\fint_{Q_{\alpha_1}}{\int_0^{2r(x)}
				{\vert\bfvarepsilon(\bfu)(x+s\bfPhi^x_i(\eta))\vert \,ds}d\eta}}
		\\
		&\leq 2\sum_{i=1}^d{\fint_{Q_{\alpha_1}}
			{\int_0^{2r(x)}
				{\vert\bfvarepsilon(\bfu)(x+s\bfPhi^x_i(\eta))\vert
					\vert \det((\textsf{D}\bfPhi_i^x)(\eta)^\top (\textsf{D}\bfPhi_i^x)(\eta ))\vert \,ds}d\eta}}
		\\&= \frac{2}{\vert Q_{\alpha_1}\vert}
		\sum_{i=1}^d{\int_{\bfPhi_i^x(Q_{\alpha_1})}{\int_0^{2r(x)}
				{\vert\bfvarepsilon(\bfu)(x+s\xi)\vert \,ds}do(\xi)}}
		\\&\leq \frac{2d }{\vert Q_{\alpha_1}\vert}
		\int_{\mathbb{S}^{n-1}}
		{\int_0^{2r(x)}{s^{d-1}
				\frac{\vert\bfvarepsilon(\bfu)(x+s\xi)\vert}{s^{d-1}}\,ds}do(\xi)}
		\\&= \frac{2d }{\vert Q_{\alpha_1}\vert}
		\int_{B_{2r(x)}(x)\cap \Omega}
		{\frac{\vert\bfvarepsilon(\bfu)(y)\vert}{\vert x-y\vert^{d-1}}\,dy},
		\end{split}\label{final}
		\end{align}
		where we used \eqref{eq:pp12} and that 
		$\vert \det((\textsf{D}\bfPhi_i^x)(\eta)^\top (\textsf{D}\bfPhi_i^x)(\eta))\vert
		=\vert \det((\textsf{D}\bfPhi_i)(\eta)^\top (\textsf{D}\bfPhi_i)(\eta))\vert\ge \frac{1}{2}$ 
		for every $\eta\in Q_{\alpha_1}\subseteq B_{\alpha_0}^{d-1}(0)$ 
		in the third line. Eventually, observing that all constants in \eqref{final} solely depend on the Lipschitz characteristics of $\Omega$, we conclude the assertion.\hfill$\qed$
	\end{proof}

	\section{Variable exponent Bochner--Lebesgue spaces}
	\label{sec:4}
	In this section we introduce variable exponent Bochner--Lebesgue spaces, the appropriate substitute of usual Bochner--Lebesgue spaces for the treatment of unsteady problems in variable exponents spaces, such as the model problem \eqref{eq:model}. 
	
	Throughout the entire section, if nothing else is stated, then we assume that  $\Omega\subseteq\setR^d$, $d\ge 2$, 
	is a bounded domain, $I=\left(0,T\right)$, $T<\infty$,  $Q_T=I\times \Omega$,
	and $q,p\in \mathcal{P}^\infty(Q_T)$.

	\begin{defn}[Time slice spaces]\label{4.1} We define for almost every $t\in I$ the 
		\textbf{time slice spaces}
		\begin{align*}
		X^{q,p}_{\nb}(t):=X^{q(t,\cdot),p(t,\cdot)}_{\nb}(\Omega)^d,\qquad
		X^{q,p}_{\bfvarepsilon}(t):= X^{q(t,\cdot),p(t,\cdot)}_{\bfvarepsilon}(\Omega).
		\end{align*}
		Furthermore, we define the \textbf{limiting time slice spaces}
		\begin{align*}
				X^{q,p}_+:= X^{q^+,p^+}_{\nb}(\Omega)^d,\qquad X^{q,p}_-:= X^{q^-,p^-}_{\nb}(\Omega)^d.
		\end{align*}
	\end{defn}

	\begin{rmk}
		\begin{description}[(iii)]
                \item[(i)] For time independent variable exponents
                  $q,p\in \mathcal{P}^\infty(\Omega)$, we define
                  $X^{q,p}_\nb:=X^{q(\cdot),p(\cdot)}_{\nb}(\Omega)^d$
                  and
                  $X^{q,p}_\bfvarepsilon:=X^{q(\cdot),p(\cdot)}_{\bfvarepsilon}(\Omega)$.
			\item[(ii)] Recall that if $p\in \mathcal{P}^{\log}(Q_T)$ with $p^->1$,
			in virtue of Korn's inequality (cf.~Proposition~\ref{2.4}),
			the spaces $X^{q,p}_\nb(t)$ and $X^{q,p}_\bfvarepsilon(t)$
			coincide for every $t\in I$, with a possibly
                        on $t\in I$
 			depending norm equivalence. Thus, for 
			$p\in \mathcal{P}^{\log}(Q_T)$ with $p^->1$, we set $X^{q,p}(t):=X^{q,p}_\bfvarepsilon(t)=X^{q,p}_\nb(t)$. For time independent variable exponents $q\in \mathcal{P}^\infty(\Omega)$ and $p\in \mathcal{P}^{\log}(\Omega)$ with $p^->1$, we set likewise $X^{q,p}:=X^{q,p}_\bfvarepsilon=X^{q,p}_\nabla$.
			\item[(iii)] There hold for almost every $t\in I$ the dense embeddings $X^{q,p}_+\hookrightarrow  X^{q,p}_\bfvarepsilon(t)\hookrightarrow X^{q,p}_\nb(t)\hookrightarrow X^{q,p}_-$.
		\end{description}
    \end{rmk}

	By means of the time slice spaces $X^{q,p}_\nb(t)$, $t\in I$, and $X^{q,p}_\bfvarepsilon(t)$, $t\in I$, we next introduce variable exponent Bochner--Lebesgue spaces.
		
	\begin{defn}\label{4.3}We define the \textbf{variable exponent Bochner--Lebesgue spaces}
		\begin{alignat*}{2}
		\bscal{X}_{\nb}^{q,p}(Q_T)
		&:=\{\bsu\in L^{q(\cdot,\cdot)}(Q_T)^d
		\fdg \nb\bsu\in L^{p(\cdot,\cdot)}(Q_T)^{d\times d},
		\bsu(t)\in X^{q,p}_{\nb}(t)\text{ for a.e. }t\in I\},\\
		\bscal{X}_{\bfvarepsilon}^{q,p}(Q_T)
		&:=\{\bsu\in L^{q(\cdot,\cdot)}(Q_T)^d
		\fdg \bfvarepsilon(\bsu)\in L^{p(\cdot,\cdot)}(Q_T,\mathbb{M}_{\sym}^{d\times d}),\;\bsu(t)\in X^{q,p}_{\bfvarepsilon}(t)\text{ for a.e. }t\in I\}.
		\end{alignat*}
		Moreover, we define the \textbf{limiting Bochner--Lebesgue spaces}
		\begin{align*}
		\bscal{X}^{q,p}_+(Q_T):=L^{\max\{p^+,q^+\}}(I,X^{q,p}_+),\qquad
		\bscal{X}^{q,p}_-(Q_T):=L^{\min\{p^-,q^-\}}(I,X^{q,p}_-).
		\end{align*} 
	\end{defn}
	
	Before we equip $\bscal{X}_{\nabla}^{q,p}(Q_T)$ and $\bscal{X}_{\bfvarepsilon}^{q,p}(Q_T)$ with appropriate norms, we deal
	with the question, whether there  exists a wide range of variable exponents for
	which the spaces $\bscal{X}_{\nabla}^{q,p}(Q_T)$ and $\bscal{X}_{\bfvarepsilon}^{q,p}(Q_T)$ coincide, i.e., whether a distinction of $\bscal{X}_{\nabla}^{q,p}(Q_T)$ and $\bscal{X}_{\bfvarepsilon}^{q,p}(Q_T)$ becomes unnecessary, if e.g. the variable exponents are smooth or not depending on time. Recall that for a $\log$-Hölder continuous exponent $p\in \mathcal{P}^{\log}(\Omega)$ with $p^->1$, the spaces $X^{q,p}_{\nabla}$ and $X^{q,p}_{\bfvarepsilon}$, i.e., the steady counterpart to $\bscal{X}_{\nabla}^{q,p}(Q_T)$ and $\bscal{X}_{\bfvarepsilon}^{q,p}(Q_T)$, respectively,  coincide in virtue of Korn's inequality (cf.~Proposition~\ref{2.4}). In addition, in the case of a constant exponent $p\in\left(1,\infty\right)$, Korn's inequality implies the norm equivalence $\|\nabla\cdot\|_{L^p(Q_T)^{d\times d}}
	\!\sim \!\|\bfvarepsilon(\cdot)\|_{L^p(Q_T)^{d\times d}}$ on
	$\bscal{X}_{\bfvarepsilon}^{q,p}(Q_T)$, and the spaces
	$\bscal{X}_{\nabla}^{q,p}(Q_T)$ and $\bscal{X}_{\bfvarepsilon}^{q,p}(Q_T)$ 
	coincide anew. Thus, one may wonder, whether this circumstance also holds true in the case of a non-constant exponent
	$p\in\mathcal{P}^{\log}(Q_T)$~with~$p^->1$.  
	Regrettably,  even if the variable exponent is smooth, and not depending on time, the answer is  negative. 
	This  can be traced back to the following phenomenon, 
	which occurs in variable exponent Bochner--Lebesgue spaces, and leads to the invalidity of a variety~of~inequalities.
	
	
		\begin{prop}[Wet blanket]\label{4.4}
		Let $\Omega\subseteq \setR^d$, $d\ge 2$, be an arbitrary bounded 
		 domain and $k,l,m,n\in\setN$. Moreover, 
		let $\bscal{F}\subseteq C_0^\infty(\Omega)^{m\times n}\times 
		C_0^\infty(\Omega)^{k\times l}$ be a family of function couples
		containing a couple $(\mathbf{G}_0,\bfF_0)^\top \in\bscal{F}$, such that $\bfF_0\not\equiv \mathbf{0}$ and
		$\textup{int}(\textup{supp}(\mathbf{G}_0))\setminus\textup{supp}(\bfF_0)\neq\emptyset$. 
		Then, for every $1<\alpha<\beta<\infty $ there exists a smooth exponent
		$p\in C^\infty(\setR^d)$, with $p|_{Q_T}\in \mathcal{P}^{\log}(Q_T)$\footnote{Here, we extend $p\in C^\infty(\setR^d)$ constantly in time, i.e., we set $p(t,x):=p(x)$ for all $(t,x)^\top\in Q_T$.}, and 
		$p^-=\alpha$, $p^+=\beta$, such that for every  $\varphi\in L^\alpha(I)\setminus L^\beta(I)$ it holds $\varphi\mathbf{F}_0\in L^{p(\cdot)}(Q_T)^{k\times l}$ and  $\varphi\mathbf{G}_0\notin L^{p(\cdot)}(Q_T)^{m\times n}$. In addition, we have
		\begin{align*}
		\sup_{\varphi\in C_0^\infty(I)}{\frac{\|\varphi \mathbf{G}_0\|_{L^{p(\cdot)}(Q_T)^{m\times n}}}
			{\|\varphi \bfF_0\|_{L^{p(\cdot)}(Q_T)^{k\times l}}}}=\infty,
		\end{align*}
		or equivalently, there is no constant $c>0$, such that $	\|\varphi \mathbf{G}_0\|_{L^{p(\cdot)}(Q_T)^{m\times n}}
		\leq c\|\varphi \bfF_0\|_{L^{p(\cdot)}(Q_T)^{k\times l}}$ for every 
		$\varphi\in C_0^\infty(I)$.
	\end{prop}

	\begin{proof}
		Let us set $\Omega_1:=\text{int}(\text{supp}(\bfF_0))\neq \emptyset$,
		$\Omega_1^\vep:=\Omega_1+B_\vep^d(0)$ and $\Omega_2^\vep
		:=\Omega\setminus\overline{\Omega_1^\vep}$ for ${\vep>0}$. Since
		$\Omega_1\subset\subset\Omega$, we infer for $\vep>0$ sufficiently small that
		$\Omega_1^\vep\subset\subset\Omega$ and thus 
		$\Omega_2^\vep\neq \emptyset$. Moreover, since
		$\text{int}(\text{supp}(\mathbf{G}_0))\setminus\text{supp}(\bfF_0)\neq\emptyset$, 
		we obtain for a possibly smaller $\vep>0$ that 
		${\Omega_2^\vep\cap \text{int}(\text{supp}(\mathbf{G}_0))\neq \emptyset}$. Let us fix such an $\vep>0$ and define the variable exponent
		\begin{align*}
		p:=
		\alpha(\omega^d_{\frac{\vep}{2}}\ast \chi_{\Omega_1^{\vep/2}})+
		\beta[1-(\omega^d_{\frac{\vep}{2}}\ast \chi_{\Omega_1^{\vep/2}})],
		\end{align*}
		where $\omega^d_{\varepsilon}\in C_0^\infty(B_{\varepsilon}^d(0))$, $\varepsilon>0$, are the scaled standard mollifiers from Proposition \ref{2.3}.
		Then, we have $p\in C^\infty(\setR^d)$ with 
		$p(x)=\alpha$ if $x \in \Omega_1$, 
		$p(x)=\beta$ if $x \in \setR^d\setminus\overline{\Omega_1^\vep}$ and 
		$\alpha\leq p(x)\leq \beta$ if 
		$x\in\Omega_1^\vep\setminus \overline{\Omega_1}$. 
		In particular, there holds $p|_{Q_T}\in \mathcal{P}^{\log}(Q_T)$ and 
		for every $\varphi\in C_0^\infty(I)$
		\begin{align*}
		\rho_{p(\cdot)}(\varphi \bfF_0)
		=\int_I{\int_{\Omega_1}{\vert\varphi(t) \bfF_0(x)\vert^\alpha\,dx}\,dt}
		=\|\varphi\|_{L^\alpha(I)}^\alpha\|\bfF_0\|_{L^\alpha(\Omega_1)^{k\times l}}^\alpha,
		\end{align*}
		i.e., $\|\varphi \bfF_0\|_{L^{p(\cdot)}(Q_T)^{k\times l}}
		=\|\varphi\|_{L^\alpha(I)}\|\bfF_0\|_{L^\alpha(\Omega_1)^{k\times l}}$. 
		On the other hand, it holds~for~every~${\varphi\in C_0^\infty(I)}$
		\begin{align*}
		\rho_{p(\cdot)}(\varphi \mathbf{G}_0)
		\ge\int_I{\int_{\Omega_2^\vep}{\vert\varphi(t) \mathbf{G}_0(x)\vert^\beta\,dx}\,dt}
		=\|\varphi\|_{L^\beta(I)}^\beta\|\mathbf{G}_0\|_{L^\beta(\Omega_2^\vep)^{m\times n}}^\beta,
		\end{align*}
		i.e., $\|\varphi \mathbf{G}_0\|_{L^{p(\cdot)}(Q_T)^{m\times n}}
		\ge\|\varphi\|_{L^\beta(I)}\|\mathbf{G}_0\|_{L^\beta(\Omega_2^\vep)^{m\times n}}$. 
		Now,  take $\varphi\in L^\alpha(I)\setminus L^\beta(I)$~and~$(\varphi_n)_{n\in\setN}\subseteq C_0^\infty(I)$, such that 
		$\varphi_n\to\varphi$ in $L^\alpha(I)$ $(n\to\infty)$. Then, $\varphi\mathbf{F}_0\in L^{p(\cdot)}(Q_T)^{k\times l}$ and  $\varphi\mathbf{G}_0\notin L^{p(\cdot)}(Q_T)^{m\times n}$.
		Moreover, we have $\|\varphi_n\|_{L^\beta(I)}\to \infty$~$(n\to\infty)$ and thus
		\begin{align*}
		\frac{\|\varphi_n \mathbf{G}_0\|_{L^{p(\cdot)}(Q_T)^{m\times n}}}
		{\|\varphi_n \bfF_0\|_{L^{p(\cdot)}(Q_T)^{k\times l}}}
		\ge\frac{\|\varphi_n\|_{L^\beta(I)}\|\mathbf{G}_0\|_{L^\beta(\Omega_2^\vep)^{m\times n}}}
		{\|\varphi_n\|_{L^\alpha(I)}\|\bfF_0\|_{L^\alpha(\Omega_1)^{k\times l}}}
		\overset{n\to \infty}{\to }\infty,
		\end{align*}
		where we used that $\|\mathbf{G}_0\|_{L^\beta(\Omega_2^\vep)^{m\times n}}> 0$, 
		since $\Omega_2^\vep\cap \text{int}(\text{supp}(\mathbf{G}_0))\neq \emptyset$ is open. \hfill$\qed$
	\end{proof} 
	
	By means of Proposition \ref{4.4} we are able to prove the invalidity of a Korn type inequality in the framework of variable exponent Bochner--Lebesgue spaces, which entails a variety of difficulties.
	
	\begin{rmk}[Invalidity of Korn's inequality on $\bscal{X}^{q,p}_{\bfvarepsilon}(Q_T)$]\label{4.5}
		Let $\Omega\subseteq \setR^d$, $d\ge 2$, be an
                arbitrary bounded domain, $I=(0,T)$, $T<\infty$, a bounded interval
		and $Q_T=I\times \Omega$. Moreover, let 
		\begin{align*}
		\bscal{F}_{\textsf{Korn}}:=\{(\nb \bfu,\bfvarepsilon(\bfu))^\top 
		\fdg \bfu\in C_0^\infty(\Omega)^d\}\subseteq C_0^\infty(\Omega)^{d\times d}\times
		C_0^\infty(\Omega)^{d\times d} .
		\end{align*}
		Let $\eta\in C_0^\infty(\Omega)$ with $\eta\equiv 1$ on $G$, where $G\subset\subset \Omega$ is a domain,
		and $\bfA\in \mathbb{R}^{d\times d}\setminus\{\mathbf{0}\}$ be skew symmetric. 
		If we set $\bfu(x):=\eta(x)\bfA x$ for every $x\in \Omega$, then 
		$\bfu\in C_0^\infty(\Omega)^d$ with $\nb\bfu=\bfA$ and
		${\bfvarepsilon(\bfu)=\mathbf{0}}$ in $G$. In other words, we have
		$(\nb\bfu,\bfvarepsilon(\bfu))^\top \in \bscal{F}_{\textsf{Korn}}$ 
		with  $\bfvarepsilon(\bfu)\not\equiv  \mathbf{0}$ and  $\textup{int}(\textup{supp}(\nb\bfu))\setminus \textup{supp}(\bfvarepsilon(\bfu))\neq\emptyset$.
		In consequence, according to Proposition \ref{4.4}, there exists a smooth, time independent exponent $p\in C^\infty(\setR^d)$ with  
		$p^->1$, which does not admit a constant $c>0$, such that for every 
		${\bfphi\in C_0^\infty(Q_T)^d}$ it holds
		\begin{align*}
		\|\nb\bfphi\|_{L^{p(\cdot)}(Q_T)^{d\times d}}
		\leq c\|\bfvarepsilon(\bfphi)\|_{L^{p(\cdot)}(Q_T)^{d\times d}}.
		\end{align*}
		In particular, for every $\varphi\in L^{p^-}(I)\setminus L^{p^+}(I)$ there holds $\varphi\bfu\in \bscal{X}^{p^-,p}_{\bfvarepsilon}(Q_T)\setminus \bscal{X}^{p^-,p}_{\nabla}(Q_T)$, i.e., 
		we have $\bscal{X}^{p^-,p}_{\bfvarepsilon}(Q_T)\neq \bscal{X}^{p^-,p}_{\nabla}(Q_T)$.
                
                The situation is illustrated  in Figure
                \ref{Korn} for
                $\Omega:=B_{2.5}^2(0)$, 
                $G:=B_{0.6}^2(0)$,  $\bfu\in C_0^\infty(\Omega)^d$, given via $\bfu(x):=\eta(x)\bfA x$ for all $x\in \Omega$, where $\bfA:=\begin{psmallmatrix}0 & -1\\1 &
                0\end{psmallmatrix}$ and
                ${\eta:=\chi_{B_1^2(0)}\ast\omega_\varepsilon^2\in
                C_0^\infty(\Omega)}$ for ${\vep=0.4}$.
                \begin{figure}[ht]
                	\centering
                	\hspace*{-7mm}\includegraphics[width=15.5cm]{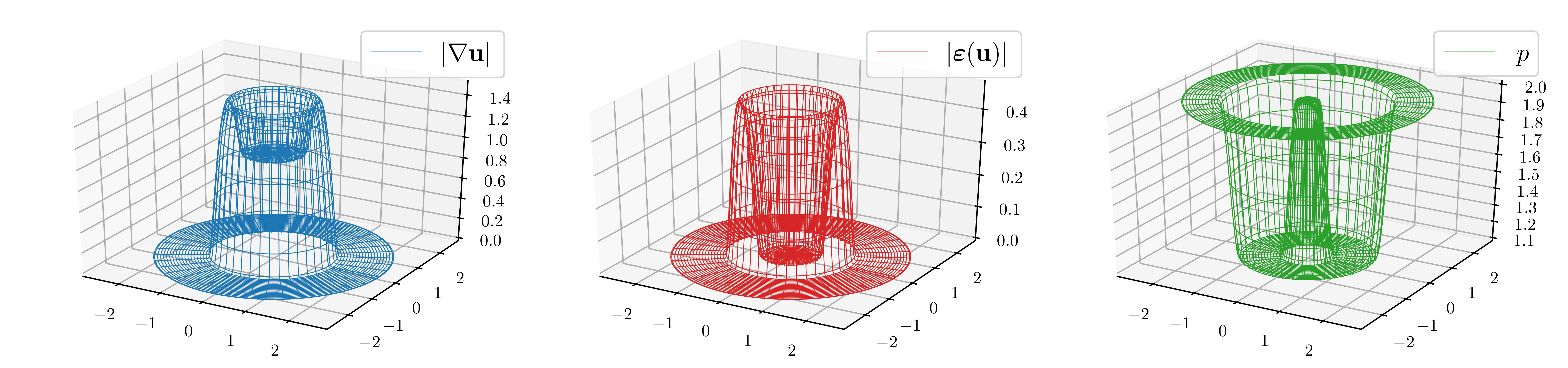}
                	\caption{Plots of $\vert \nb\bfu\vert\in C_0^\infty(\Omega)$ (blue/left), $\vert\bfvarepsilon(\bfu)\vert\in C_0^\infty(\Omega)$ (red/middle) and the exponent $p\in C^\infty(\mathbb{R}^2)$ (green/right) constructed according to Proposition \ref{4.4} for $d=2$,  $\alpha=1.1$ and $\beta=2$.}
                	\label{Korn}
                \end{figure}

            In Figure
            \ref{varphi} one can see  for $I:=(-1.5,1.5)$, the function
            $\varphi\in L^{1.1}(I)\setminus L^2(I)$, given via ${\varphi(t):=\chi_{\left(-1,1\right)}(t)\big(\vert t\vert^{-\frac{1}{2}}-1\big)}$ for $t\in I$, and its approximations ${(\varphi_n)_{n\in \setN}:=(\varphi\ast\omega_{2^{-n}}^1)_{n\in \setN}\subseteq C_0^\infty(I)}$. 
            According to Proposition \ref{4.4}, the sequence ${(\bfphi_n)_{n\in \setN}\subseteq C_0^\infty(Q_T)^d}$,  for every $(t,x)^\top\in Q_T$ and $n\in \setN$ given via $\bfphi_n(t,x):=\varphi_n(t)\bfu(x)$, satisfies 
            \begin{align*}
            	\frac{\|\nb\bfphi_n\|_{L^{p(\cdot)}(Q_T)^{d\times d}}}{\|\bfvarepsilon(\bfphi_n)\|_{L^{p(\cdot)}(Q_T)^{d\times d}}}\overset{n\to\infty}{\to }\infty.
            \end{align*}
            
            \vspace*{-5mm}
            \begin{figure}[ht]
            	\centering
            	\includegraphics[width=14.5cm]{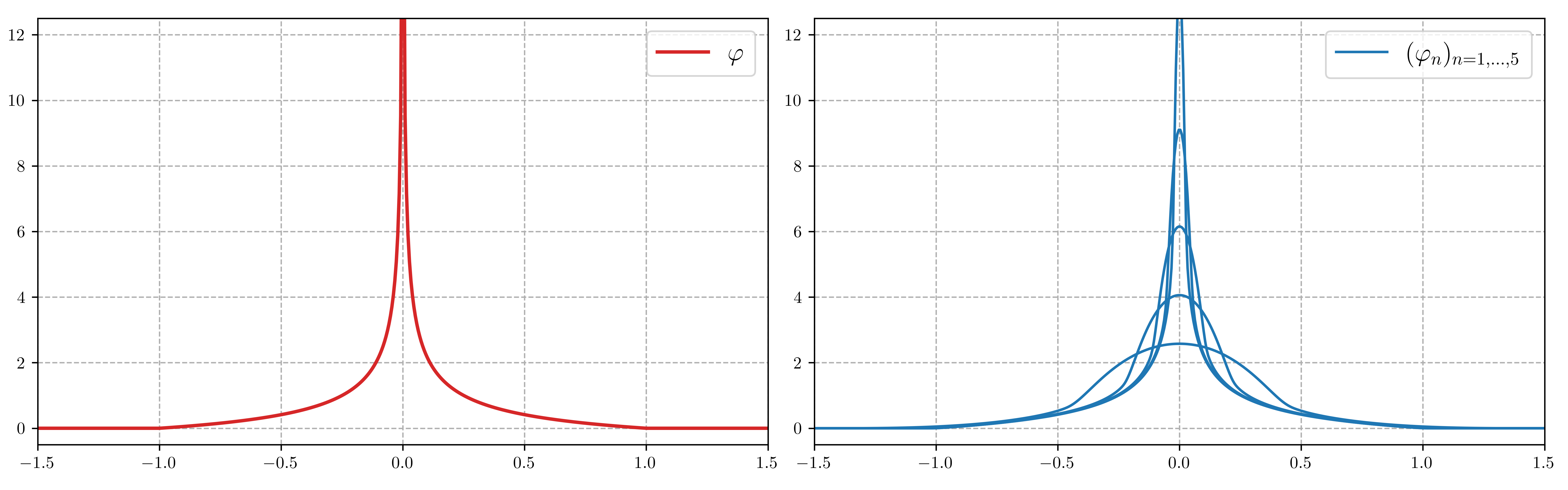}
            	\caption{Plots of $\varphi\in L^{1.1}(I)\setminus L^2(I)$ (red/left) and $(\varphi_n)_{n=1,...,5}\subseteq C_0^\infty(I)$ (blue/right).}
            	\label{varphi}
            \end{figure}
	\end{rmk}	
	
	In consequence, we have to distinguish between the spaces $\bscal{X}_{\nb}^{q,p}(Q_T)$ 
	and $\bscal{X}_{\bfvarepsilon}^{q,p}(Q_T)$, i.e., we have to equip
	$\bscal{X}_{\nb}^{q,p}(Q_T)$ and $\bscal{X}_{\bfvarepsilon}^{q,p}(Q_T)$ 
	with different norms, in order to guarantee at least the well-definedness of these norms.
	
	\begin{prop}\label{4.6}
		The spaces $\bscal{X}_{\nb}^{q,p}(Q_T)$ and $\bscal{X}_{\bfvarepsilon}^{q,p}(Q_T)$ equipped with the norms
		\begin{align*}
		\|\cdot\|_{\bscal{X}_{\nb}^{q,p}(Q_T)}
		&:=\|\cdot\|_{L^{q(\cdot,\cdot)}(Q_T)^d}
		+\|\nb\cdot\|_{L^{p(\cdot,\cdot)}(Q_T)^{d\times d}},
		\\\|\cdot\|_{\bscal{X}_{\bfvarepsilon}^{q,p}(Q_T)}
		&:=\|\cdot\|_{L^{q(\cdot,\cdot)}(Q_T)^d}
		+\|\bfvarepsilon(\cdot)\|_{L^{p(\cdot,\cdot)}(Q_T)^{d\times d}}
		\end{align*}
		are Banach spaces. In addition, we have $\bscal{X}_{+}^{q,p}(Q_T)\embedding \bscal{X}_{\nb}^{q,p}(Q_T)
		\embedding\bscal{X}_{\bfvarepsilon}^{q,p}(Q_T)\embedding\bscal{X}_{-}^{q,p}(Q_T)$.
	\end{prop}
	
	\begin{proof}
			We only give a proof for $\bscal{X}_{\bfvarepsilon}^{q,p}(Q_T)$, 
		since all arguments transfer to $\bscal{X}_{\nb}^{q,p}(Q_T)$. 
		
		The norm properties follow from the definition. Hence, it remains to prove completeness. To this end, let $(\bsu_n)_{n\in\setN}
		\subseteq\bscal{X}_{\bfvarepsilon}^{q,p}(Q_T)$ be a Cauchy sequence, 
		i.e., also $(\bsu_n)_{n\in\setN}\subseteq L^{q(\cdot,\cdot)}(Q_T)^d$ 
		and \linebreak $(\bfvarepsilon(\bsu_n))_{n\in\setN}\subseteq
		L^{p(\cdot,\cdot)}(Q_T,\mathbb{M}^{d\times d}_{\sym})$ are Cauchy sequences. 
		Thus, as $L^{q(\cdot,\cdot)}(Q_T)^d$ and 
		$L^{p(\cdot,\cdot)}(Q_T,\mathbb{M}^{d\times d}_{\sym})$ are Banach spaces, there exist 
		elements $\bsu\in L^{q(\cdot,\cdot)}(Q_T)^d$ 
		and $\bscal{E}\in L^{p(\cdot,\cdot
			)}(Q_T,\mathbb{M}^{d\times d}_{\sym})$, such that
		\begin{align}
		\begin{split}
		\begin{alignedat}{2}
		\bsu_n&\overset{n\to\infty}{\to} \bsu&&\quad
		\text{ in }L^{q(\cdot,\cdot)}(Q_T)^d,\\
		\bfvarepsilon(\bsu_n)&
		\overset{n\to\infty}{\to} \bscal{E}&&\quad
		\text{ in }L^{p(\cdot,\cdot)}(Q_T,\mathbb{M}^{d\times d}_{\sym}).
		\end{alignedat}
		\end{split}\label{eq:4.6.1}
		\end{align}
		Since norm convergence implies modular convergence, we infer from \eqref{eq:4.6.1}
		\begin{align}
		\begin{split}
		\int_I\rho_{q(t,\cdot)}(\bsu_n(t)-\bsu(t))
		&+\rho_{p(t,\cdot)}(\bfvarepsilon(\bsu_n)(t)-\bscal{E}(t))\,dt
		\\&=\rho_{q(\cdot,\cdot)}(\bsu_n-\bsu)
		+\rho_{p(\cdot,\cdot)}(\bfvarepsilon(\bsu_n)-\bscal{E})
		\overset{n\to \infty}{\to }0.
		\end{split}\label{eq:4.6.2}
		\end{align}
		Thus, \eqref{eq:4.6.2} yields a subsequence 
		$(\bsu_n)_{n\in\Lambda}\subseteq 
		\bscal{X}^{q,p}_{\bfvarepsilon}(Q_T)$, with 
		$\Lambda\subseteq \setN$, such that for~almost~every~${t\in I}$
		\begin{align}
		\rho_{q(t,\cdot)}(\bsu_n(t)-\bsu(t))
		+\rho_{p(t,\cdot)}(\bfvarepsilon(\bsu_n)(t)-\bscal{E}(t))
		\overset{n\to\infty}{\to }0\quad(n\in\Lambda).\label{eq:4.6.3}
		\end{align}
	As conversely modular convergence implies norm convergence, 
		\eqref{eq:4.6.3} yields for almost every~${t\in I}$
		\begin{align}
		\begin{split}
		\begin{alignedat}{2}
		\bsu_n(t)&\overset{n\to\infty}{\to}\bsu(t)&&\quad
		\text{ in }L^{q(t,\cdot)}(\Omega)^d\quad\quad(n\in\Lambda),
		\\\bfvarepsilon(\bsu_n)(t)
		&\overset{n\to\infty}{\to}\bscal{E}(t)&&\quad
		\text{ in }L^{p(t,\cdot)}(\Omega,\mathbb{M}^{d\times d}_{\sym})\quad(n\in\Lambda).
		\end{alignedat}
		\end{split}\label{eq:4.6.4}
		\end{align}
		From \eqref{eq:4.6.4} we derive that the sequence $(\bsu_n(t))_{n\in\Lambda}\subseteq X^{q,p}_{\bfvarepsilon}(t)$ 
		is  for almost every $t\in I$ a Cauchy sequence. Since for almost every $t\in I$ 
		the space $X^{q,p}_{\bfvarepsilon}(t)$ is a Banach space, we obtain for almost every $t\in I$ that $\bsu(t)\in X^{q,p}_{\bfvarepsilon}(t)$ with $\bfvarepsilon(\bsu)(t)=\bscal{E}(t)$ in $L^{p(t,\cdot)}(\Omega,\mathbb{M}^{d\times d}_{\sym})$ and $\bsu_n(t)\to\bsu(t)$ in $X^{q,p}_{\bfvarepsilon}(t)$ $(\Lambda\ni n\to \infty)$.
		In other words, we have 
		$\bsu\in \bscal{X}^{q,p}_{\bfvarepsilon}(Q_T)$ 
		and $\bsu_n\to\bsu$ in
		$\bscal{X}^{q,p}_{\bfvarepsilon}(Q_T)$ $(n\to \infty)$, i.e., $\bscal{X}^{q,p}_{\bfvarepsilon}(Q_T)$ is complete. 
		
		The embedding $\bscal{X}_{+}^{q,p}(Q_T)\embedding \bscal{X}_{\nb}^{q,p}(Q_T)$ follows from the embeddings $L^{q^+}(Q_T)^d\embedding  L^{q(\cdot,\cdot)}(Q_T)^d$ and $L^{p^+}(Q_T)^{d\times d}\embedding  L^{p(\cdot,\cdot)}(Q_T)^{d\times d}$ (cf.~\cite[Theorem 3.3.1.]{DHHR11}). The embedding $\bscal{X}_{\nb}^{q,p}(Q_T)\embedding  \bscal{X}^{q,p}_{\bfvarepsilon}(Q_T)$ is evident. 
		The embedding $\bscal{X}^{q,p}_{\bfvarepsilon}(Q_T)\embedding \bscal{X}_{-}^{q,p}(Q_T)$ follows from the embeddings $  L^{q(\cdot,\cdot)}(Q_T)^d\embedding L^{q^-}(Q_T)^d$ and $ L^{p(\cdot,\cdot)}(Q_T,\mathbb{M}^{d\times d}_{\sym})^d\embedding L^{p^-}(Q_T,\mathbb{M}^{d\times d}_{\sym})^d$ as well as Korn's inequality for the constant exponent $p^-\in \left(1,\infty\right)$.
		\hfill$\qed$
	\end{proof}
	
	Henceforth, 
	 this article will primarily treat the space $\bscal{X}_{\bfvarepsilon}^{q,p}(Q_T)$, as this
	 space is of greater interest with regard to the model problem \eqref{eq:model}. 
	 However, we emphasize that all following definitions and results admit congruent adaptations 
	 to $\bscal{X}_{\nb}^{q,p}(Q_T)$.
	
	\begin{cor}\label{4.7}
		Let
		$(\bsu_n)_{n\in\setN}
		\subseteq \bscal{X}^{q,p}_{\bfvarepsilon}(Q_T)$ 
		be a sequence,  such that $\bsu_n\to\bsu\text{ in }\bscal{X}^{q,p}_{\bfvarepsilon}(Q_T)$ ${(n\to\infty)}$.
		Then, there exists a subsequence $(\bsu_n)_{n\in\Lambda}\subseteq \bscal{X}^{q,p}_{\bfvarepsilon}(Q_T)$, 
		with $\Lambda\subseteq \setN$, such that 
		$\bsu_n(t)\to \bsu(t)$ in $X^{q,p}_{\bfvarepsilon}(t)$ $(\Lambda\ni n\to\infty)$
		for almost every $t\in I$.
	\end{cor}

	\begin{prop}\label{4.8}
		The mapping $\bfPi_{\bfvarepsilon}:\bscal{X}^{q,p}_{\bfvarepsilon}(Q_T)
		\to L^{q(\cdot,\cdot)}(Q_T)^d\times L^{p(\cdot,\cdot)}(Q_T,\mathbb{M}_{\sym}^{d\times d})$,
		 given via
		\begin{align*}
			\bfPi_{\bfvarepsilon}\bsu
			:=(\bsu,\bfvarepsilon(\bsu))^\top 
			\quad\text{ in }L^{q(\cdot,\cdot)}(Q_T)^d\times L^{p(\cdot,\cdot)}(Q_T,\mathbb{M}_{\sym}^{d\times d})
		\end{align*}
		for every $\bsu\in \bscal{X}^{q,p}_{\bfvarepsilon}(Q_T)$, 
		is an isometric isomorphism into its range $R(\bfPi_{\bfvarepsilon})$. In particular, $\bscal{X}_{\bfvarepsilon}^{q,p}(Q_T)$ is separable. In addition, if $q^-,p^->1$, then $\bscal{X}_{\bfvarepsilon}^{q,p}(Q_T)$ is reflexive.
	\end{prop}
	
	\begin{proof}
		Since $\bfPi_{\bfvarepsilon}:
		\bscal{X}_{\bfvarepsilon}^{q,p}(Q_T)\to 
		L^{q(\cdot,\cdot)}(Q_T)^d\times 
		L^{p(\cdot,\cdot)}(Q_T,\mathbb{M}_{\sym}^{d\times d})$ 
		is an isometry and $	\bscal{X}_{\bfvarepsilon}^{q,p}(Q_T)$ 
		a Banach space (cf.~Proposition~\ref{4.6}),  $R(\bfPi_{\bfvarepsilon})$ 
		is a closed subspace of ${L^{q(\cdot,\cdot)}(Q_T)^d\times L^{p(\cdot,\cdot)}(Q_T,\mathbb{M}_{\textup{sym}}^{d\times d})}$, i.e.,
		$R(\bfPi_{\bfvarepsilon})$ is a Banach space, 
		and therefore $\bfPi_{\bfvarepsilon}:	\bscal{X}_{\bfvarepsilon}^{q,p}(Q_T)\to R(\bfPi_{\bfvarepsilon})$ 
		an isometric isomorphism. 
		As $L^{q(\cdot,\cdot)}(Q_T)^d$ and $L^{p(\cdot,\cdot)}(Q_T,\mathbb{M}_{\sym}^{d\times d})$ are separable, the product 
		${L^{q(\cdot,\cdot)}(Q_T)^d\times L^{p(\cdot,\cdot)}(Q_T,\mathbb{M}_{\sym}^{d\times d})}$ is separable as well. Thus, $R(\bfPi_{\bfvarepsilon})$ inherits the separability of 
		${L^{q(\cdot,\cdot)}(Q_T)^d\times L^{p(\cdot,\cdot)}(Q_T,\mathbb{M}_{\sym}^{d\times d})}$, 
		and in virtue of the isomorphism $\bfPi_{\bfvarepsilon}:
		\bscal{X}_{\bfvarepsilon}^{q,p}(Q_T)\to R(\bfPi_{\bfvarepsilon})$ also $	\bscal{X}_{\bfvarepsilon}^{q,p}(Q_T)$. If additionally  $q^-,p^->1$, then the spaces $L^{q(\cdot,\cdot)}(Q_T)^d$ and $L^{p(\cdot,\cdot)}(Q_T,\mathbb{M}_{\sym}^{d\times d})$, and thus also the product $L^{q(\cdot,\cdot)}(Q_T)^d\times L^{p(\cdot,\cdot)}(Q_T,\mathbb{M}_{\sym}^{d\times d})$, are reflexive (cf.~Proposition~\ref{riesz}). 
		Hence, the range $R(\bfPi_{\bfvarepsilon})$ is reflexive and, in virtue of the  isomorphism 
		${\bfPi_{\bfvarepsilon}:
			\bscal{X}_{\bfvarepsilon}^{q,p}(Q_T)\to R(\bfPi_{\bfvarepsilon})}$, also $	\bscal{X}_{\bfvarepsilon}^{q,p}(Q_T)$.~\hfill$\qed$
	\end{proof}
	
	\begin{prop}[Characterization of 
		$\bscal{X}^{q,p}_{\bfvarepsilon}(Q_T)^*$]\label{4.9} Let $q^-,p^->1$. Then, the operator
	    $\bscal{J}_{\bfvarepsilon}:
	   L^{q'(\cdot,\cdot)}(Q_T)^d\times 
	   L^{p'(\cdot,\cdot)}(Q_T,\mathbb{M}_{\sym}^{d\times d})\to \bscal{X}^{q,p}_{\bfvarepsilon}(Q_T)^*$,
		given via
		\begin{align*}
		\langle	\bscal{J}_{\bfvarepsilon}(\bsf,
		\bsF),\bsu\rangle
		_{\bscal{X}^{q,p}_{\bfvarepsilon}(Q_T)}&:=( \bsf,\bsu)_{L^{q(\cdot,\cdot)}(Q_T)^d}
	 	+(\bsF,
	 	\bfvarepsilon(\bsu))_{L^{p(\cdot,\cdot)}(Q_T)^{d\times d}}
		\end{align*}
		for every $\bsf\in L^{q'(\cdot,\cdot)}(Q_T)^d$, 
		$\bsF\in L^{p'(\cdot,\cdot)}(Q_T,\mathbb{M}_{\sym}^{d\times d})$ 
		and $\bsu\in\bscal{X}^{q,p}_{\bfvarepsilon}(Q_T)$, 
		is well-defined, linear and Lipschitz continuous with constant $2$.  In addition, for every $\bsu^*\in \bscal{X}^{q,p}_{\bfvarepsilon}(Q_T)^*$ there exist $\bsf\in L^{q'(\cdot,\cdot)}(Q_T)^d$, 
		$\bsF\in L^{p'(\cdot,\cdot)}(Q_T,\mathbb{M}_{\sym}^{d\times d})$, such that $\bsu^*=\bscal{J}_{\bfvarepsilon}(\bsf,\bsF)$ in $\bscal{X}^{q,p}_{\bfvarepsilon}(Q_T)^*$ and
		\begin{align}
		\frac{1}{2}\|\bsu^*\|_{\bscal{X}^{q,p}_{\bfvarepsilon}(Q_T)^*}\leq \|\bsf\|_{L^{q'(\cdot,\cdot)}(Q_T)^d}+\|\bsF\|_{L^{p'(\cdot,\cdot)}(Q_T)^{d\times d}}\leq 2\|\bsu^*\|_{\bscal{X}^{q,p}_{\bfvarepsilon}(Q_T)^*}.\label{eq:4.9}
		\end{align}
	\end{prop}

	\begin{rmk}\label{4.9help}
		The operator $\bscal{J}_{\bfvarepsilon}:
		L^{q'(\cdot,\cdot)}(Q_T)^d\times 
		L^{p'(\cdot,\cdot)}(Q_T,\mathbb{M}_{\sym}^{d\times
                  d})\to \bscal{X}^{q,p}_{\bfvarepsilon}(Q_T)^*$ is
                closely related to the adjoint operator  $\bfPi_{\bfvarepsilon}^*:(L^{q(\cdot,\cdot)}(Q_T)^d\times L^{p(\cdot,\cdot)}(Q_T,\mathbb{M}_{\sym}^{d\times d}))^*\to \bscal{X}^{q,p}_{\bfvarepsilon}(Q_T)^*$. In fact, consider  $\bsJ_{q,p}:L^{q'(\cdot,\cdot)}(Q_T)^d\times 
		L^{p'(\cdot,\cdot)}(Q_T,\mathbb{M}_{\sym}^{d\times d})\to (L^{q(\cdot,\cdot)}(Q_T)^d\times 
		L^{p(\cdot,\cdot)}(Q_T,\mathbb{M}_{\sym}^{d\times d}))^*$, given via
		\begin{align*}
			\langle \bsJ_{q,p}(\bsf,\bsF),(\bsg,\bsG)^\top\rangle_{L^{q(\cdot,\cdot)}(Q_T)^d\times 
				L^{p(\cdot,\cdot)}(Q_T,\mathbb{M}_{\sym}^{d\times d})}:=(\bsf,\bsg)_{L^{q(\cdot,\cdot)}(Q_T)^d}+(\bsF,\bsG)_{L^{p(\cdot,\cdot)}(Q_T)^{d\times d}}
		\end{align*}
		for every $\bsf \in L^{q'(\cdot,\cdot)}(Q_T)^d$, $\bsg\in L^{q(\cdot,\cdot)}(Q_T)^d$, $\bsF\in L^{p'(\cdot,\cdot)}(Q_T,\mathbb{M}_{\sym}^{d\times d})$ and $\bsG\in L^{p(\cdot,\cdot)}(Q_T,\mathbb{M}_{\sym}^{d\times d})$. Using Proposition \ref{riesz}, we deduce for every $\bsf \in L^{q'(\cdot,\cdot)}(Q_T)^d$ and $\bsF\in L^{p'(\cdot,\cdot)}(Q_T,\mathbb{M}_{\sym}^{d\times d})$ that
		\begin{align}
		\begin{split}
			\frac{1}{2} \|(\bsf,\bsF)^\top\|_{L^{q'(\cdot,\cdot)}(Q_T)^d\times L^{p'(\cdot,\cdot)}(Q_T)^{d\times d}}
			&\leq \| \bsJ_{q,p}(\bsf,\bsF)\|_{(L^{q(\cdot,\cdot)}(Q_T)^d\times 
				L^{p(\cdot,\cdot)}(Q_T,\mathbb{M}_{\sym}^{d\times d}))^*}\\&\leq 2\|(\bsf,\bsF)^\top\|_{L^{q'(\cdot,\cdot)}(Q_T)^d\times L^{p'(\cdot,\cdot)}(Q_T)^{d\times d}},
			\end{split}\label{eq:4.9help}
		\end{align}
		i.e., $\bsJ_{q,p}:L^{q'(\cdot,\cdot)}(Q_T)^d\times 
		L^{p'(\cdot,\cdot)}(Q_T,\mathbb{M}_{\sym}^{d\times d})\to (L^{q(\cdot,\cdot)}(Q_T)^d\times 
		L^{p(\cdot,\cdot)}(Q_T,\mathbb{M}_{\sym}^{d\times d}))^*$ is an isomorphism. Moreover, we have $\bscal{J}_{\bfvarepsilon}=\bfPi_{\bfvarepsilon}^*\circ\bsJ_{q,p}$, i.e., $\bscal{J}_{\bfvarepsilon}$ and $\bfPi_{\bfvarepsilon}^*$ coincide up to~the~isomorphism~$\bsJ_{q,p}$.
	\end{rmk}

	\begin{proof}(of Proposition \ref{4.9})
		Well-definedness, linearity and Lipschitz continuity
                with constant $2$ follow from Hölder's inequality. So,
                let us prove the surjectivity including
                \eqref{eq:4.9}. For ${\bsu^*\in
                \bscal{X}^{q,p}_{\bfvarepsilon}(Q_T)^*}$ we define
                $\widetilde \bsu \in (R(\bfPi_{\bfvarepsilon}))^*$ via
                $\langle \widetilde \bsu ,
                \bfPi_{\bfvarepsilon}\bsu\rangle _{R(\bfPi_{\bfvarepsilon})}:=\langle  \bsu^* ,
                \bsu\rangle_{\bscal{X}^{q,p}_{\bfvarepsilon}(Q_T)}
                $ for every $\bsu\in \bscal{X}^{q,p}_{\bfvarepsilon}(Q_T)$. Since $\bfPi_{\bfvarepsilon}$ is an isometric
                isomorphism, we have $\norm{\widetilde \bsu }_{(R(\bfPi_{\bfvarepsilon}))^*}=\norm{\bsu ^*}_{\bscal{X}^{q,p}_{\bfvarepsilon}(Q_T)^*}$.
                The Hahn--Banach theorem yields a norm preserving
                extension $\widehat \bsu \in ( L^{q(\cdot,\cdot)}(Q_T)^d\times 
		L^{p(\cdot,\cdot)}(Q_T,\mathbb{M}_{\sym}^{d\times d}))^*$ of $\widetilde \bsu \in
                (R(\bfPi_{\bfvarepsilon}))^*$. From
                Remark~\ref{4.9help} follows that there exist
                functions $\bsf\in L^{q'(\cdot,\cdot)}(Q_T)^d$ and
                $\bsF\in
                L^{p'(\cdot,\cdot)}(Q_T,\mathbb{M}_{\sym}^{d\times
                  d})$, such that $\bsJ_{q,p}(\bsf,\bsF) =\widehat
                \bsu$ in $(L^{q(\cdot,\cdot)}(Q_T)^d\times 
		L^{p(\cdot,\cdot)}(Q_T,\mathbb{M}_{\sym}^{d\times d}))^*$. 
                Consequently, there holds for every $\bsu\in\bscal{X}^{q,p}_{\bfvarepsilon}(Q_T)$
		\begin{align*}
		\langle \bsu^*,\bsu
		\rangle_{\bscal{X}^{q,p}_{\bfvarepsilon}(Q_T)}
                  &=\langle \widetilde \bsu, \bfPi_{\bfvarepsilon}\bsu
                    \rangle_{R(\bfPi_{\bfvarepsilon})}
                  \\
                  &=\langle \widehat \bsu, \bfPi_{\bfvarepsilon}\bsu
                    \rangle_{L^{q(\cdot,\cdot)}(Q_T)^d\times 
			L^{p(\cdot,\cdot)}(Q_T,\mathbb{M}_{\sym}^{d\times d})}
		\\&=
		\langle 	\bsJ_{q,p}
		(\bsf,\bsF)
		,\bfPi_{\bfvarepsilon}\bsu
		\rangle_{L^{q(\cdot,\cdot)}(Q_T)^d\times 
			L^{p(\cdot,\cdot)}(Q_T,\mathbb{M}_{\sym}^{d\times d})}
		\\&=(\bsf,\bsu)_{L^{q(\cdot,\cdot)}(Q_T)^d}+(\bsF,\bfvarepsilon(\bsu))_{L^{p(\cdot,\cdot)}(Q_T)^{d\times d}}
		\\&=\langle \bscal{J}_{\bfvarepsilon}
		(\bsf,\bsF),\bsu\rangle
		_{\bscal{X}^{q,p}_{\bfvarepsilon}(Q_T)},
		\end{align*}
		i.e., $\bsu^*=\bscal{J}_{\bfvarepsilon}
		(\bsf,\bsF)$ in $\bscal{X}^{q,p}_{\bfvarepsilon}(Q_T)^*$. 
		Thus, ${\bscal{J}_{\bfvarepsilon}:L^{q'(\cdot,\cdot)}(Q_T)^d\times
		L^{p'(\cdot,\cdot)}(Q_T,\mathbb{M}_{\sym}^{d\times d})\to
		\bscal{X}^{q,p}_{\bfvarepsilon}(Q_T)^*}$ is surjective. 
		In addition, from the above follows that $\norm{\bsu
                  ^*}_{\bscal{X}^{q,p}_{\bfvarepsilon}(Q_T)^*}=\norm{\widetilde
                  \bsu }_{(R(\bfPi_{\bfvarepsilon}))^*}=$ \linebreak$\norm{\widehat \bsu }_{(L^{q(\cdot,\cdot)}(Q_T)^d\times 
			L^{p(\cdot,\cdot)}(Q_T,\mathbb{M}_{\sym}^{d\times d}))^*}=\norm{\bsJ_{q,p}
		(\bsf,\bsF)}_{(L^{q(\cdot,\cdot)}(Q_T)^d\times 
			L^{p(\cdot,\cdot)}(Q_T,\mathbb{M}_{\sym}^{d\times
                          d}))^*}$ and \eqref{eq:4.9} follows from \eqref{eq:4.9help}.\hfill$\qed$
\end{proof}

	In the same manner, we can characterize the dual space of $X_{\bfvarepsilon}^{q,p}=X^{q(\cdot),p(\cdot)}_{\bfvarepsilon}(\Omega)$.
	
	\begin{cor}\label{3.8}
		Let $q,p\in \mathcal{P}^\infty(\Omega)$. Then,  the mapping ${\Pi_{\bfvarepsilon}:X_{\bfvarepsilon}^{q,p}\to L^{q(\cdot)}(\Omega)^d\times L^{p(\cdot)}(\Omega,\mathbb{M}^{d\times d}_{\sym})}$, for every $\bfu\in X_{\bfvarepsilon}^{q,p}$  given via
		\begin{align*}
		\Pi_{\bfvarepsilon}\bfu:=(\bfu,\bfvarepsilon(\bfu))^\top\quad\textup{ in }L^{q(\cdot)}(\Omega)^d\times L^{p(\cdot)}(\Omega,\mathbb{M}^{d\times d}_{\sym}),
		\end{align*}
		is an isometric isomorphism. In particular,  $X_{\bfvarepsilon}^{q,p}$ is separable. Furthermore, if $q^-,p^->1$, then $X_{\bfvarepsilon}^{q,p}$ is reflexive and the operator  $\mathcal{J}_{\bfvarepsilon}:L^{q'(\cdot)}(\Omega)^d\times L^{p'(\cdot)}(\Omega,\mathbb{M}^{d\times d}_{\sym})\to (X_{\bfvarepsilon}^{q,p})^*$, for every ${\bff\in L^{q'(\cdot)}(\Omega)^d}$, $\bfF\in L^{p'(\cdot)}(\Omega,\mathbb{M}^{d\times d}_{\sym})$ and $\bfu\in X_{\bfvarepsilon}^{q,p}$ given via 
		\begin{align*}
		\langle\mathcal{J}_{\bfvarepsilon}(\bff,\bfF),\bfu\rangle_{X_{\bfvarepsilon}^{q,p}}=(\bff,\bfu)_{L^{q(\cdot)}(\Omega)^d}+( \bfF,\bfvarepsilon(\bfu))_{L^{p(\cdot)}(\Omega)^{d\times d}}
		\end{align*}
		is well-defined, linear and Lipschitz continuous with constant $2$. In additon, for every ${\bfu^*\in (X_{\bfvarepsilon}^{q,p})^*}$ there exist $\bff\in L^{q'(\cdot)}(\Omega)^d$ and $\bfF\in L^{p'(\cdot)}(\Omega,\mathbb{M}^{d\times d}_{\sym})$, such that $\bfu^*=\mathcal{J}_{\bfvarepsilon}(\bff,\bfF)$ in $(X_{\bfvarepsilon}^{q,p})^*$ and
		\begin{align*}
		\frac{1}{2}\|\bfu^*\|_{(X_{\bfvarepsilon}^{q,p})^*}\leq\|\bff\|_{L^{q'(\cdot)}(\Omega)^d}+\|\bfF\|_{L^{p'(\cdot)}(\Omega)^{d\times d}}\leq 2	\|\bfu^*\|_{(X_{\bfvarepsilon}^{q,p})^*}.
		\end{align*}
	\end{cor}
	
	The following remark deals with the question of a well-posed definition
	of a time evaluation of functionals in $\bscal{X}^{q,p}_{\bfvarepsilon}(Q_T)^*$.
	\begin{rmk}
		[Time slices of functionals $\bsu^*\in \bscal{X}^{q,p}_{\bfvarepsilon}(Q_T)^*$]
		\label{3.9}Let $q^-,p^->1$.
		\begin{description}[{(ii)}]
			\item[(i)] Since for functions $\bsf\in
                          L^{q'(\cdot,\cdot)}(Q_T)^d$ and $\bsF\in
                          L^{p'(\cdot,\cdot)}(Q_T,\mathbb{M}_{\sym}^{d\times
                            d})$ it holds $\bsf(t)\in
                          L^{q'(t,\cdot)}(\Omega)^d$ and $\bsF(t)\in
                          L^{p'(t,\cdot)}(\Omega,\mathbb{M}_{\sym}^{d\times
                            d})$ for almost every $t\in I$ by Fubini's
                          theorem, Corollary \ref{3.8} yields
                          $\mathcal{J}_{\bfvarepsilon}(\bsf(t),\bsF(t))\in
                          X^{q,p}_{\bfvarepsilon}(t)^*$ for almost
                          every $t\in I$. Fubini's theorem also
                          implies, due to Hölder's inequality, for every
                          $\bsu\in
                          \bscal{X}^{q,p}_{\bfvarepsilon}(Q_T)$, i.e.,
                          $\bsu(t)\in X^{q,p}_{\bfvarepsilon}(t)$ for
                          almost every $t\in I$, that
                          $(t\mapsto\langle
                         \mathcal{J}_{\bfvarepsilon}(\bsf(t),\bsF(t)),\bsu(t)\rangle_{X^{q,p}_{\bfvarepsilon}(t)})=(t\mapsto(\bsf(t),\bsu(t))_{L^{q(t,\cdot)}(\Omega)^d}+(\bsF(t),\bfvarepsilon(\bsu(t)))_{L^{p(t,\cdot)}(\Omega)^{d\times
                              d}})$ belongs to $L^1(I)$. More precisely, we have
			\begin{align}
			\begin{split}
				\langle \bscal{J}_{\bfvarepsilon}(\bsf,\bsF),\bsu\rangle_{\bscal{X}^{q,p}_{\bfvarepsilon}(Q_T)}&=\int_I{(\bsf(t),\bsu(t))_{L^{q(t,\cdot)}(\Omega)^d}+(\bsF(t),\bfvarepsilon(\bsu(t)))_{L^{p(t,\cdot)}(\Omega)^{d\times d}}\,dt}\\&
				=\int_I{\langle\mathcal{J}_{\bfvarepsilon}(\bsf(t),\bsF(t)),\bsu(t)\rangle_{X^{q,p}_{\bfvarepsilon}(t)}\,dt}.
				\end{split}\label{eq:3.9.1}
			\end{align}
			\item[(ii)] At first sight, it is not clear, whether $\bsu^*\in \bscal{X}^{q,p}_{\bfvarepsilon}(Q_T)^*$ possesses time slices, in the sense that $\bsu^*(t)\in X^{q,p}_{\bfvarepsilon}(t)^*$ for almost every $t\in I$.  But since due to Proposition \ref{4.9} there exist $\bsf\in L^{q'(\cdot,\cdot)}(Q_T)^d$ and $\bsF\in L^{p'(\cdot,\cdot)}(Q_T,\mathbb{M}_{\sym}^{d\times d})$, such that $\bsu^*=\bscal{J}_{\bfvarepsilon}(\bsf,\bsF)$ in $\bscal{X}^{q,p}_{\bfvarepsilon}(Q_T)^*$, we are apt to define
			\begin{align}
			\bsu^*(t):=\mathcal{J}_{\bfvarepsilon}(\bsf(t),\bsF(t))\quad\textup{ in }X^{q,p}_{\bfvarepsilon}(t)^*\quad\text{ for a.e. }t\in I.\label{eq:2.22.1}
			\end{align}
			To this end, we have to clarify, whether \eqref{eq:2.22.1} is well-defined, i.e., independent of the choice of $\bsf\in L^{q'(\cdot,\cdot)}(Q_T)^d$ and $\bsF\in L^{p'(\cdot,\cdot)}(Q_T,\mathbb{M}_{\sym}^{d\times d})$. So, let us consider a further representation $\bsg\in L^{q'(\cdot,\cdot)}(Q_T)^d$ and $\bsG\in L^{p'(\cdot,\cdot)}(Q_T,\mathbb{M}_{\sym}^{d\times d})$, such that $\bsu^*=\bscal{J}_{\bfvarepsilon}(\bsg,\bsG)$ in $\bscal{X}^{q,p}_{\bfvarepsilon}(Q_T)^*$. Thus, testing by $\bsu:=\bfu\varphi\in \bscal{X}^{q,p}_{\bfvarepsilon}(Q_T)$, where $\bfu\in X_+^{q,p}$ and $\varphi\in C^\infty(I)$ are arbitrary, also using \eqref{eq:3.9.1}, we obtain
			\begin{align*}
			\int_I{\langle\mathcal{J}_{\bfvarepsilon}(\bsf(t),\bsF(t)),\bfu\rangle_{X^{q,p}_{\bfvarepsilon}(t)}\varphi(t)\,dt}=\int_I{\langle\mathcal{J}_{\bfvarepsilon}(\bsg(t),\bsG(t)),\bfu\rangle_{X^{q,p}_{\bfvarepsilon}(t)}\varphi(t)\,dt}, 
			\end{align*}
			i.e., $\langle\mathcal{J}_{\bfvarepsilon}(\bsf(t),\bsF(t)),\bfu\rangle_{X^{q,p}_{\bfvarepsilon}(t)}=\langle\mathcal{J}_{\bfvarepsilon}(\bsg(t),\bsG(t)),\bfu\rangle_{X^{q,p}_{\bfvarepsilon}(t)}$ for almost every $t\in I$. Since $X_+^{q,p}$ is for almost every $t\in I$ dense in $X^{q,p}_{\bfvarepsilon}(t)$, we conclude $\mathcal{J}_{\bfvarepsilon}(\bsf(t),\bsF(t))=\mathcal{J}_{\bfvarepsilon}(\bsg(t),\bsG(t))$ in $X^{q,p}_{\bfvarepsilon}(t)^*$ for almost every $t\in I$. Therefore, the time slices \eqref{eq:2.22.1} are well-defined.
		\end{description}
	\end{rmk}

	\section{Smoothing in $\bscal{X}^{q,p}_{\bfvarepsilon}(Q_T)$ and $\bscal{X}^{q,p}_{\bfvarepsilon}(Q_T)^*$}
	\label{sec:5}
	
	This section is concerned with the construction of an appropriate 
	smoothing operator for $\bscal{X}^{q,p}_{\bfvarepsilon}(Q_T)$ and $\bscal{X}^{q,p}_{\bfvarepsilon}(Q_T)^*$. We emphasize that the presented smoothing method  is not new and 
	is based on ideas from \cite{DNR12}. However, in \cite{DNR12}
        it is solely shown that the smoothing method applies to $\bscal{X}^{q,p}_{\nb}(Q_T)$. 
	An essential component of the argumentation was the pointwise Poincar\'e inequality \eqref{eq:pw}. We extend the method to the greater space $\bscal{X}^{q,p}_{\bfvarepsilon}(Q_T)$, by means of the pointwise Poincar\'e inequality for the symmetric gradient near the boundary of a bounded Lipschitz domain (cf.~Theorem~\ref{2.31}).

	In what follows, let 
	$\Omega\subseteq \setR^d$, $d\ge 2$, be a bounded Lipschitz domain, $I=\left(0,T\right)$,
	with ${0<T<\infty}$, $Q_T=I\times \Omega$ and $q,p\in \mathcal{P}^{\log}(Q_T)$ 
	with $q^-,p^->1$.

        For the constants $c_0,h_0>0$, provided by Theorem~\ref{2.31},
        we set $h_1:=\frac{h_0}{4}$ and 
	$(\eta_h)_{h>0}:=(        \omega_{h/2}^d\ast \chi_{\Omega_{5/2h}}
        )_{h>0}\subseteq C^\infty_0(\Omega)$. Thus, for every $h>0$ holds 
	$\text{supp}(\eta_h)\subset \Omega_{2h}$, $0\leq \eta_h\leq 1\text{ in }\Omega $, $\eta_h=1\text{ in }\Omega_{3h}$ and $\|\nb\eta_h\|
		_{L^\infty(\setR^d)^d}\leq \frac{c_\eta}{h}$,
	where $c_\eta>0$ is a constant not depending on $h>0$. Henceforth, we will always denote by $\bfomega:=\omega^{d+1}\in C_0^\infty(B_1^{d+1}(0))$
	the standard mollifier from Proposition \ref{2.3}. Likewise we
        define the family of scaled mollifiers $\bfomega_h\in
        C_0^\infty(B_h^{d+1}(0))$, $h>0$, by
        $\bfomega_h(t,x):=\frac{1}{h^{d+1}}\bfomega(\frac{t}{h},\frac{x}{h})$
        for every $(t,x)^\top\in \setR^{d+1}$. By $c_d>0$ we always denote a constant which depends
        only on the dimension $d\in \setN$ and $c_\eta>0$. 
        
        We frequently make use of the zero extension operator $\bscal{F}_{Q_T}:L^1(Q_T)^d\to L^1(\setR^{d+1})^d$,
	which is for every $\bsu\in L^1(Q_T)^d$ given via $\bscal{F}_{Q_T}\bsu:=\bsu$ in $Q_T$ and ${\bscal{F}_{Q_T}\bsu:=\mathbf{0}}$~in~$Q_T^c$. Apparently, $\bscal{F}_{Q_T}:L^1(Q_T)^d\to L^1(\setR^{d+1})^d$ is linear and Lipschitz continuous with constant $1$. It is not difficult to see that also $\bscal{F}_{Q_T}:\bscal{X}^{q,p}_{\bfvarepsilon}(Q_T)\to \bscal{X}^{\tilde{q},\tilde{p}}_{\bfvarepsilon}(\widetilde{Q}_T)$, where $\widetilde{Q}_T:=\tilde{I}\times \widetilde{\Omega}$ is an extended cylinder, with a bounded interval $\tilde{I}\supseteq I$ and a bounded domain $\widetilde{\Omega}\supset\Omega$, and $\tilde{q},\tilde{p}\in \mathcal{P}^{\log}(\widetilde{Q}_T)$ are extended exponents, with $\tilde{q}|_{Q_T}=q$, $\tilde{p}|_{Q_T}=p$, $q^-\leq \tilde{q}\leq q^+$ and   $p^-\leq \tilde{p}\leq p^+$ in $\widetilde{Q}_T$, is well-defined, linear and Lipschitz continuous with constant $1$. Consequently, the adjoint operator ${\bscal{F}_{Q_T}^*:\bscal{X}^{\tilde{q},\tilde{p}}_{\bfvarepsilon}(\widetilde{Q}_T)^*\to \bscal{X}^{q,p}_{\bfvarepsilon}(Q_T)^*}$  is also well-defined, linear and Lipschitz continuous with constant~$1$.

	\begin{prop}\label{5.1}  For $\bsu\in L^{q(\cdot,\cdot)}(Q_T)^d$ and $h>0$, we define the smoothing operator
		\begin{align}\label{def:Rh}
		\bscal{R}^h_{Q_T}\bsu:=\bfomega_{h}\ast(\eta_h\bscal{F}_{Q_T}\bsu)\in C^\infty(\mathbb{R}^{d+1})^d.
		\end{align}
		Then, for every ${\bsu\in L^{q(\cdot,\cdot)}(Q_T)^d}$ it holds:
		\begin{description}[{(iii)}]
			\item[(i)] $(\bscal{R}^h_{Q_T}\bsu)_{h>0}\subseteq C_0^\infty(\mathbb{R}^{d+1})^d$ with $\text{supp}(\bscal{R}^h_{Q_T}\bsu)\subseteq \left(-h,T+h\right)\times \Omega_h$~for~every~$h>0$.
			\item[(ii)] $\sup_{h>0}{\vert \bscal{R}^h_{Q_T}\bsu\vert} \leq 2 M_{d+1}(\bscal{F}_{Q_T}\bsu)$ almost everywhere in $\setR^{d+1}$.
			\item[(iii)] For an extension $\overline{q}\in\mathcal{P}^{\log}(\setR^{d+1})$ of $q\in\mathcal{P}^{\log}(Q_T)$, with $q^-\leq \overline{q}\leq q^+$ in $\setR^{d+1}$, there exist a constant $c_{\overline{q}}>0$ (depending on $\overline{q}\in\mathcal{P}^{\log}(\setR^{d+1})$), such that 
			\begin{align*}
			\sup_{h>0}{\| \bscal{R}^h_{Q_T}\bsu\|_{L^{{q}(\cdot,\cdot)}(Q_T)^d}}\leq \sup_{h>0}{\| \bscal{R}^h_{Q_T}\bsu\|_{L^{\overline{q}(\cdot,\cdot)}(\setR^{d+1})^d}}\leq c_{\overline{q}}\|\bsu\|_{L^{q(\cdot,\cdot)}(Q_T)^d}.
			\end{align*}
			In particular, $\bscal{R}^h_{Q_T}:L^{q(\cdot,\cdot)}(Q_T)^d\to L^{q(\cdot,\cdot)}(Q_T)^d$ is for every $h>0$ linear and bounded.
			\item[(iv)] $\bscal{R}^h_{Q_T}\bsu\to\bsu$ in $L^{q(\cdot,\cdot)}(Q_T)^d$ $(h\to 0)$.
		\end{description}
	\end{prop}

	\begin{rmk}\label{5.1.1}
		The smoothing operator in Proposition \ref{5.1} admits congruent extensions
		to not relabelled smoothing operators for scalar functions, i.e., $\bscal{R}^h_{Q_T}:L^{q(\cdot,\cdot)}(Q_T)\to L^{\overline{q}(\cdot,\cdot)}(\mathbb{R}^{d+1})$, and tensor-valued functions, i.e., $\bscal{R}^h_{Q_T}:L^{q(\cdot,\cdot)}(Q_T)^{d\times d}\to L^{\overline{q}(\cdot,\cdot)}(\mathbb{R}^{d+1})^{d\times d}$.
	\end{rmk}

	\begin{proof}(of Proposition \ref{5.1})
		\textbf{ad (i)}
		As $\bscal{F}_{Q_T}\bsu\in L^1(\setR^{d+1})^d$, we have $(\bscal{R}^h_{Q_T}\bsu)_{h>0}\subseteq C^\infty(\setR^{d+1})^d$ by the standard theory of mollification, see e.g. \cite{AF03}. 
		In particular, we have for every $h>0$
		\begin{align*}
		\text{supp}(\bscal{R}^h_{Q_T}\bsu)\subseteq \text{supp}(\eta_h\bscal{F}_{Q_T}\bsu)+B_{h}^{d+1}(0)\subseteq (I\times \Omega_{2h})+B_{h}^{d+1}(0)\subseteq \left(-h,T+h\right)\times \Omega_h.
		\end{align*}
		
		\textbf{ad (ii)} From Proposition \ref{2.3} (ii) and $\sup_{h>0}{\| \eta_h\|_{L^\infty(\Omega)}}\leq 1$,~we~infer 
		\begin{align}\label{eq:max}
			\sup_{h>0}{\vert\bscal{R}^h_{Q_T}\bsu\vert}\leq \sup_{h>0}{2M_{d+1}( \eta_h\bscal{F}_{Q_T}\bsu)}\leq 2 M_{d+1}(\bscal{F}_{Q_T}\bsu) \quad\text{ a.e. in }\setR^{d+1}.
		\end{align}
		
		\textbf{ad (iii)} Let $\overline{q}\in
                \mathcal{P}^{\log}(\setR^{d+1})$ be an extension of
                $q\in \mathcal{P}^{\log}(Q_T)$ with $q^-\leq
                \overline{q}\leq q^+$ in $\setR^{d+1}$
                (cf.~Proposition~\ref{2.1.1}). In particular, we have
                $\bscal{F}_{Q_T}\bsu\in
                L^{\overline{q}(\cdot,\cdot)}(Q_T)^d$ and
                $\overline{q}^->1$. Therefore, assertion
                \textbf{(iii)}  follows from Proposition \ref{2.2} and
                \eqref{eq:max}.\\[-3mm]
		
		\textbf{ad (iv)} By Proposition \ref{2.3} (i) \& (iii)
                and Lebesgue's theorem on dominated convergence,~we~obtain
		\begin{align*}
		\|\bscal{R}^h_{Q_T}\bsu-\bsu\|_{L^{q(\cdot,\cdot)}(Q_T)^d}&\leq \|\bfomega_{h}\ast[\eta_h\bscal{F}_{Q_T}\bsu-\bscal{F}_{Q_T}\bsu]\|_{L^{q(\cdot,\cdot)}(\setR^{d+1} )^d}\\&\quad+\|\bfomega_h\ast (\bscal{F}_{Q_T}\bsu)-\bscal{F}_{Q_T}\bsu\|_{L^{q(\cdot,\cdot)}(\setR^{d+1})^d}\\&\leq K\|(1-\eta_h)\bscal{F}_{Q_T}\bsu\|_{L^{q(\cdot,\cdot)}(\setR^{d+1} )^d}\\&\quad+\|\bfomega_{h}\ast(\bscal{F}_{Q_T}\bsu)-\bscal{F}_{Q_T}\bsu\|_{L^{q(\cdot,\cdot)}(\setR^{d+1})^d}
		\overset{h\to 0}{\to}0.\tag*{$\qed$}
		\end{align*}
	\end{proof}

	We need another smoothing operator. Since for $q\in  \mathcal{P}^{\log}(Q_T)$ with $q^->1$ also
        $q'\in  \mathcal{P}^{\log}(Q_T)$ with $(q')^->1$,
        Proposition~\ref{5.1}~(iii) showed for every $h>0$ that
        ${\bscal{R}^h_{Q_T}:L^{q'(\cdot,\cdot)}(Q_T)^d\to
        L^{q'(\cdot,\cdot)}(Q_T)^d}$ is well-defined, linear and
        continuous.  As a consequence, we define for every $h>0$ the quasi adjoint operator ${(\bscal{R}^h_{Q_T})^\star:L^{q(\cdot,\cdot)}(Q_T)^d\to L^{q(\cdot,\cdot)}(Q_T)^d}$ via
        \begin{align}\label{def:R*}
( (\bscal{R}^h_{Q_T})^\star\bsu,\bsv)_{L^{q'(\cdot,\cdot)}(Q_T)^d}:=(\bsu, \bscal{R}^h_{Q_T}\bsv)_{L^{q'(\cdot,\cdot)}(Q_T)^d}
	\end{align}
	for every $\bsu\in L^{q(\cdot,\cdot)}(Q_T)^d$ and $\bsv\in
        L^{q'(\cdot,\cdot)}(Q_T)^d$.
        \begin{rmk}
          The operator $(\bscal{R}^h_{Q_T})^\star$ is related to the
          adjoint operator $(\bscal{R}^h_{Q_T})^*$ of the operator $\bscal{R}^h_{Q_T}:L^{q'(\cdot,\cdot)}(Q_T)^d\to
        L^{q'(\cdot,\cdot)}(Q_T)^d$ via $(\bscal{R}^h_{Q_T})^\star= J^{-1}
        \circ (\bscal{R}^h_{Q_T})^*\circ J$, where the isomorphism $J :L^{q(\cdot,\cdot)}(Q_T)^d\to
        (L^{q'(\cdot,\cdot)}(Q_T)^d)^*$ is defined via $\langle
        J\bsu,\bsv\rangle
        _{L^{q'(\cdot,\cdot)}(Q_T)^d}:=(\bsu,\bsv)_{L^{q'(\cdot,\cdot)}(Q_T)^d}$. This
        can be seen by straightforward calculations. Thus, the
        operator
        ${(\bscal{R}^h_{Q_T})^\star:L^{q(\cdot,\cdot)}(Q_T)^d\to
          L^{q(\cdot,\cdot)}(Q_T)^d}$ is well defined, linear and bounded.
        \end{rmk}
        
        The following proposition shows that $(\bscal{R}^h_{Q_T})^\star$ possesses similar properties {as $\bscal{R}^h_{Q_T}$.}

	\begin{prop}\label{5.2}  For every $\bsu\in L^{q(\cdot,\cdot)}(Q_T)^d$ it holds:
          \begin{description}[{(iii)}]
			\item[(i)] $(\bscal{R}^h_{Q_T})^\star\bsu=(\bfomega_{h}\ast (\bscal{F}_{Q_T}\bsu)) \eta_h$ almost everywhere in $Q_T$ for every $h>0$. In particular, setting $(\bscal{R}^h_{Q_T})^\star\bsu:=(\bfomega_{h}\ast (\bscal{F}_{Q_T}\bsu)) \eta_h$ almost everywhere in $Q_T^c$ for every $h>0$, we have $((\bscal{R}^h_{Q_T})^\star\bsu)_{h>0} \subseteq C^\infty_0(\setR^{d+1})^d$ with $\text{supp}((\bscal{R}^h_{Q_T})^\star\bsu)\subseteq\left(-h,T+h\right)\times \Omega_{2h}$~for~every~${h>0}$.
			\item[(ii)] $\sup_{h>0}{\vert (\bscal{R}^h_{Q_T})^\star\bsu\vert }\leq 2M_{d+1}(\bscal{F}_{Q_T}\bsu)$ almost everywhere in $\setR^{d+1}$.
			\item[(iii)] For an extension $\overline{q}\in\mathcal{P}^{\log}(\setR^{d+1})$ of $q\in\mathcal{P}^{\log}(Q_T)$, with $q^-\leq \overline{q}\leq q^+$ in $\setR^{d+1}$, there exist a constant $c_{\overline{q}}>0$ (depending on $\overline{q}\in\mathcal{P}^{\log}(\setR^{d+1})$), such that 
			\begin{align*}
			 \sup_{h>0}{\| (\bscal{R}^h_{Q_T})^\star\bsu\|_{L^{\overline{q}(\cdot,\cdot)}(\setR^{d+1})^d}}\leq c_{\overline{q}}\|\bsu\|_{L^{q(\cdot,\cdot)}(Q_T)^d}.
			\end{align*}
			\item[(iv)] $(\bscal{R}^h_{Q_T})^\star\bsu\to\bsu$ in $L^{q(\cdot,\cdot)}(Q_T)^d$ $(h\to 0)$.
		\end{description}
	\end{prop}
	
	\begin{rmk}\label{5.2.1}
		The smoothing operator $ (\bscal{R}^h_{Q_T})^\star$ defined in \eqref{def:R*} admits congruent extensions~to~not relabelled smoothing operators for scalar functions, i.e.,  ${(\bscal{R}^h_{Q_T})^\star\!:\!L^{q(\cdot,\cdot)}(Q_T)\!\to\! L^{\overline{q}(\cdot,\cdot)}(\mathbb{R}^{d+1})}$, and tensor-valued functions, i.e.,  $(\bscal{R}^h_{Q_T})^\star:L^{q(\cdot,\cdot)}(Q_T)^{d\times d}\to L^{\overline{q}(\cdot,\cdot)}(\mathbb{R}^{d+1})^{d\times d}$.
	\end{rmk}

	\begin{proof}(of Proposition \ref{5.2})
		\textbf{ad (i)} For every $\bsu\in L^{q(\cdot,\cdot)}(Q_T)^d$ and $\bsv\in L^{q'(\cdot,\cdot)}(Q_T)^d$, we infer 
		\begin{align*}
		((\bscal{R}^h_{Q_T})^\star\bsu,\bsv)_{L^{q'(\cdot,\cdot)}(Q_T)^d}=(\bsu,
                  \bfomega_h\ast(\eta_h\bscal{F}_{Q_T}\bsv))_{L^{q'(\cdot,\cdot)}(Q_T)^d}
                  =((\bfomega_h\ast(\bscal{F}_{Q_T}\bsu)) \eta_h,\bsv)_{L^{q'(\cdot,\cdot)}(Q_T)^d},
		\end{align*}
		which implies $(\bscal{R}^h_{Q_T})^\star\bsu=(\bfomega_h\ast(\bscal{F}_{Q_T}\bsu)) \eta_h\in C^\infty_0(\setR^{d+1})^d$. 
		Moreover, we get for every $h>0$ that $\text{supp}((\bscal{R}^h_{Q_T})^\star\bsu)\subseteq \left(-h,T+h\right)\times \Omega_{2h}$. \\[-3mm]
		
		\textbf{ad (ii)} Immediate consequence of Proposition \ref{2.3} (ii) and $\sup_{h>0}{\|\eta_h\|_{L^\infty(\Omega)^d}}\leq 1$. \\[-3mm]
		
		\textbf{ad (iii)} Follows as in Proposition \ref{5.1} (iii) from \textbf{(ii)} and Proposition \ref{2.2}. \\[-3mm]
		
		\textbf{ad (iv)} By Proposition \ref{2.3} (i) \& (iii) and Lebesgue's theorem on dominated convergence,~we~obtain
		\begin{align*}
		\|(\bscal{R}^h_{Q_T})^\star\bsu-\bsu\|_{L^{q(\cdot,\cdot)}(Q_T)^d}&\leq \|\eta_h[\bfomega_{h}\ast (\bscal{F}_{Q_T}\bsu)-\bscal{F}_{Q_T}\bsu]\|_{L^{q(\cdot,\cdot)}(\setR^{d+1})^d}\\&\quad+\|(1-\eta_h)\bscal{F}_{Q_T}\bsu\|_{L^{q(\cdot,\cdot)}(\setR^{d+1} )^d}\\&\leq \|\bfomega_{h}\ast(\bscal{F}_{Q_T}\bsu)-\bscal{F}_{Q_T}\bsu\|_{L^{q(\cdot,\cdot)}(\setR^{d+1})^d}\\&\quad+\|(1-\eta_h)\bscal{F}_{Q_T}\bsu\|_{L^{q(\cdot,\cdot)}(\setR^{d+1} )^d}
		\overset{h\to 0}{\to}0.\tag*{$\qed$}
		\end{align*}
	\end{proof}

	Let us now state the smoothing properties of $\bscal{R}^h_{Q_T}$ with respect to the space $\bscal{X}^{q,p}_{\bfvarepsilon}(Q_T)$.

	\begin{prop}[Smoothing in $\bscal{X}^{q,p}_{\bfvarepsilon}(Q_T)$]\label{5.3} For every $\bsu\in \bscal{X}^{q,p}_{\bfvarepsilon}(Q_T)$ it holds:
		\begin{description}[{(iii)}]
			\item[(i)] There exists a constant $c_d>0$ (not depending on $q,p\in \mathcal{P}^{\log}(Q_T)$), such that 
			\begin{align*}
				\sup_{h\in \left(0,h_1\right)}{\vert \bfvarepsilon(\bscal{R}^h_{Q_T}\bsu)\vert} \leq c_dM_{d+1}(\bscal{F}_{Q_T}\bfvarepsilon(\bsu))\quad\text{ a.e. in }\setR^{d+1}.
			\end{align*}
			\item[(ii)] For an extension $\overline{p}\in\mathcal{P}^{\log}(\setR^{d+1})$ of $p\in\mathcal{P}^{\log}(Q_T)$, with $p^-\leq \overline{p}\leq p^+$ in $\setR^{d+1}$, there exist constants $c_{\overline{p}}>0$ (depending on $\overline{p}\in\mathcal{P}^{\log}(\setR^{d+1})$), such that
			\begin{align*}
		\sup_{h\in \left(0,h_1\right)}{\| \bfvarepsilon(\bscal{R}^h_{Q_T}\bsu)\|_{L^{\overline{p}(\cdot,\cdot)}(\setR^{d+1})^{d\times d}}} \leq c_{\overline{p}}\| \bfvarepsilon(\bsu)\|_{L^{p(\cdot,\cdot)}(Q_T)^{d\times d}}.
			\end{align*}
			In particular, we have $\sup_{h\in \left(0,h_1\right)}{\|\bscal{R}^h_{Q_T}\bsu\|_{\bscal{X}^{q,p}_{\bfvarepsilon}(Q_T)}} \leq (c_{\overline{q}}+c_{\overline{p}})\|\bsu\|_{\bscal{X}^{q,p}_{\bfvarepsilon}(Q_T)}$, where $c_{\overline{q}}>0$ is from Proposition~\ref{5.1}~(iii).
			\item[(iii)] $\bscal{R}^h_{Q_T}\bsu\to \bsu$ in $\bscal{X}^{q,p}_{\bfvarepsilon}(Q_T)$ $(h\to 0)$, i.e., $C^\infty(\overline{I},C_0^\infty(\Omega)^d)$ is dense in $\bscal{X}^{q,p}_{\bfvarepsilon}(Q_T)$.
		\end{description}
	\end{prop}

	\begin{proof}
		\textbf{ad (i)} We first calculate the symmetric
                gradient of $\bscal{R}^h_{Q_T}\bsu\in
                C_0^\infty(\setR^{d+1})^d$. Using differentiation under
                the integral sign and the product rule
                $\bfvarepsilon(\eta
                \bsv)=\eta\bfvarepsilon(\bsv)+[\bsv\otimes\nb\eta]^{\sym}$
                in $C^0(\setR^{d+1})^{d\times d}$ valid for every $\bsv\in C^1(\setR^{d+1})^d$ and $\eta\in C^1(\setR^d)$, we obtain for every $h >0$
		\begin{align}
		\bfvarepsilon(\bscal{R}^h_{Q_T}\bsu)=\bscal{R}^h_{Q_T}(\bfvarepsilon(\bsu))+\bfomega_{h}\ast[(\bscal{F}_{Q_T}\bsu)\otimes\nb\eta_h]^{\sym}\quad\text{ in }\setR^{d+1}.\label{eq:5.4.1}
		\end{align}
		Proposition \ref{5.1} (ii) and Remark \ref{5.1.1} immediately provide that
		\begin{align*}
			\sup_{h>0}{\vert \bscal{R}^h_{Q_T}(\bfvarepsilon(\bsu))\vert}\leq 2M_{d+1}(\bscal{F}_{Q_T}\bfvarepsilon(\bsu))\quad\text{ a.e. in }\setR^{d+1}.
		\end{align*}
		For the remaining term note that $\bfomega_{h}\ast[(\bscal{F}_{Q_T}\bsu)\otimes\nb\eta_h]^{\sym}=\boldsymbol{0}$ in $\setR\times \Omega_{4h}$, since $\nb\eta_h=\textbf{0}$~in~$ \Omega_{3h}$. On the other hand, using Theorem \ref{2.31}, $\|\bfomega_h\|_{L^\infty(\setR^{d+1})}\leq \frac{c_d}{h^{d+1}}$ and $\|\nb\eta_h\|_{L^\infty(\setR^d)^d}\leq \frac{c_d}{h}$~for~${h>0}$, we deduce that for every $t\in \setR$, $x\in\Omega_{4h}^c$ and $h\in \left(0,h_1\right)$, i.e., $r(x)=\textup{dist}(x,\pa\Omega)\leq 4h \leq 4h_1=h_0$, that
		\begin{align}
		\begin{split}
		\vert \bfomega_{h}\ast&[(\bscal{F}_{Q_T}\bsu)\otimes\nb\eta_h]^{\sym}(t,x)\vert\\&\leq \int_{B_{h}^{d+1}(t,x)}{\bfomega_{h}(t-s,x-y)\vert \nb\eta_h(s,y)\vert\vert (\bscal{F}_{Q_T}\bsu)(s,y)\vert\,dsdy}
		\\&\leq \frac{c_d}{h^{d+2}}\int_{B_{h}^1(t)}{\int_{B_{h}^{d}(x)}{\int_{B^d_{2r(y)}(y)}{\frac{\vert(\bscal{F}_{Q_T}\bfvarepsilon(\bsu))(s,z)\vert}{\vert y-z\vert^{d-1}}\,dz}\,dy}\,ds}
		\\&\leq \frac{c_d}{h^{d+2}}\int_{B_{h}^1(t)}{\int_{B^d_{11h}(x)}{\vert(\bscal{F}_{Q_T}\bfvarepsilon(\bsu))(s,z)\vert\bigg[\int_{B^d_{12h}(z)}{\frac{1}{\vert y-z\vert^{d-1}}\,dy}\bigg]\,dz}\,ds}
		\\&\leq \frac{c_d}{h^{d+1}}\int_{B_{11h}^{d+1}(t,x)}{\vert(\bscal{F}_{Q_T}\bfvarepsilon(\bsu))(s,z)\vert\,dsdz}\\&\leq c_dM_{d+1}(\bscal{F}_{Q_T}\bfvarepsilon(\bsu))(t,x),
		\end{split}\label{eq:5.4}
		\end{align}
	where we used that $B^d_{2r(y)}(y)\subseteq B^d_{11h}(x)$ and $B^d_{h}(x)\subseteq B^d_{12h}(z)$ for all $z\in B^d_{2r(y)}(y)$, $y\in B_{h}^d(x)$ and $x\in\Omega_{4h}^c$, as for $\tilde{y}\in B^d_{2r(y)}(y)$ it holds
	${\vert \tilde{y}-x\vert \leq 	\vert x-y\vert +	\vert
        y-\tilde{y}\vert \leq h+2r(y) \leq 3h+2r(x)\leq 11h}$ since
        $\operatorname{dist}(\cdot, \partial \Omega)$ is Lipschitz continuous with
        constant $1$, and likewise for $\tilde{z}\in B^d_{h}(x)$ it holds
        $\vert \tilde{z}-z\vert \leq \vert \tilde{z}-x\vert+\vert
        x-z\vert\leq h+\vert x-y\vert +\vert y-z\vert\leq 2h+2r(y)\leq
        12h$, as well as $	\int_{B_{12h}^d(z)}{\vert y-z\vert^{1-d}dy}=12h\mathscr{H}^{d-1}(\mathbb{S}^{d-1})$.
	Eventually, combining \eqref{eq:5.4.1}--\eqref{eq:5.4},~we~conclude~\textbf{(i)}. \\[-3mm]

	\textbf{ad (ii)} Using \textbf{(i)} 
        we can proceed as in the proof of  Proposition \ref{5.1} (iii). \\[-3mm]
	
		\textbf{ad (iii)} Proposition \ref{5.1} (iv) and Remark \ref{5.1.1} immediately provide
		\begin{align}
		\begin{split}
		\begin{alignedat}{2}
		\bscal{R}^h_{Q_T}\bsu&\overset{h\to 0}{\to }\bsu&&\quad\text{ in }L^{q(\cdot,\cdot)}(Q_T)^d,\\
		\bscal{R}^h_{Q_T}\bfvarepsilon(\bsu)&\overset{h\to 0}{\to }\bfvarepsilon(\bsu)&&\quad\text{ in }L^{p(\cdot,\cdot)}(Q_T,\mathbb{M}^{d\times d}_{\sym}).
		\end{alignedat}
		\end{split}\label{eq:5.4.3}
		\end{align}
		Due to $\bfomega_{h}\ast[(\bscal{F}_{Q_T}\bsu)\otimes\nb\eta_h]^{\sym}=\boldsymbol{0}$ in $\setR\times \Omega_{4h}$ for every $h>0$ and \eqref{eq:5.4}, there holds 
		\begin{align}
		\begin{split}
		\bfomega_{h}\ast[(\bscal{F}_{Q_T}\bsu)\otimes\nb\eta_h]^{\sym}\overset{h\to 0}{\to }\boldsymbol{0}\quad\text{ a.e. in }Q_T,\\\sup_{h\in \left(0,h_1\right)}\vert \bfomega_{h}\ast[(\bscal{F}_{Q_T}\bsu)\otimes\nb\eta_h]^{\sym}\vert \leq c_dM_{d+1}(\bscal{F}_{Q_T}\bfvarepsilon(\bsu))\quad\text{ a.e. in }Q_T.
		\end{split}\label{eq:5.4.5}
		\end{align}
		Hence, since $M_{d+1}(\bscal{F}_{Q_T}\bfvarepsilon(\bsu)) \in L^{\overline{p}(\cdot,\cdot)}(\setR^{d+1})$ due to Proposition \ref{2.2}, as ${\vert\bscal{F}_{Q_T}\bfvarepsilon(\bsu)\vert\!\in\! L^{\overline{p}(\cdot,\cdot)}(\setR^{d+1})}$ and $\overline{p}\ge p^->1$, we conclude from \eqref{eq:5.4.5} by Lebegue's dominated convergence theorem that
		\begin{align}
		\bfomega_{h}\ast[(\bscal{F}_{Q_T}\bsu)\otimes\nb\eta_h]^{\sym}\overset{h\to 0}{\to }\boldsymbol{0}\quad\text{ in } L^{p(\cdot,\cdot)}(Q_T,\mathbb{M}^{d\times d}_{\sym}).\label{eq:5.4.6}
		\end{align} 
		From \eqref{eq:5.4.1}, \eqref{eq:5.4.3} and
                \eqref{eq:5.4.6}  follows that $\bscal{R}^h_{Q_T}\bsu\to\bsu$ in $\bscal{X}^{q,p}_{\bfvarepsilon}(Q_T)$ $(h\to 0)$. \hfill$\qed$
	\end{proof}
	
	The same method as in the proofs of Proposition \ref{5.1} and Proposition \ref{5.3} can be used, in order to construct a smoothing operator for $X^{q,p}=X^{q(\cdot),p(\cdot)}_{\bfvarepsilon}(\Omega)= X^{q(\cdot),p(\cdot)}_{\nb}(\Omega)^d$.
	
	\begin{cor}\label{5.5} Let $\Omega\subseteq \setR^d$, $d\ge 2$, be a bounded Lipschitz domain and $q,p\in \mathcal{P}^{\log}(\Omega)$ with $q^-,p^->1$.  For $\bfu\in X^{q,p}$ and $h>0$, we define the smoothing operator
		\begin{align*}
		R^h_{\Omega}\bfu:=\omega_{h}^d\ast(\eta_hE_{\Omega}\bfu)\in C^\infty(\setR^d)^d,
		\end{align*}
		where $\omega^d\in C_0^\infty(B_1^d(0))$ is a standard mollifier from Proposition \ref{2.3} and  $E_{\Omega}\bfu\in L^1(\setR^d)^d$ is given via $E_{\Omega}\bfu:=\bfu$  in $\Omega$ and $E_{\Omega}\bfu:=\mathbf{0}$ in $\Omega^c$. Then, for every ${\bfu\in X^{q,p}}$ it holds:
		\begin{description}[{(iii)}]
			\item[(i)] $(R^h_{\Omega}\bfu)_{h>0}\subseteq
                          C_0^\infty(\Omega)^d$ with $\text{supp}({R}^h_{\Omega}\bfu)\subseteq \Omega_h$~for~every~$h>0$.
			\item[(ii)] There exists a constant $c_d>0$ (not depending on $q,p\in \mathcal{P}^{\log}(\Omega)$), such that
			\begin{alignat*}{2}
			\sup_{h>0}{\vert R^h_{\Omega}\bfu\vert }&\leq 2M_d(E_{\Omega}\bfu)&&\quad\text{ a.e. in }\setR^d,\\
			\sup_{h\in \left(0,h_1\right)}{\vert \bfvarepsilon(R^h_{\Omega}\bfu)\vert}&\leq c_dM_d(E_{\Omega}\bfvarepsilon(\bfu))&&\quad\text{ a.e. in }\setR^d.
			\end{alignat*}
			\item[(iii)] For extensions $\overline{q},
                          \overline{p}\in\mathcal{P}^{\log}(\setR^d)$
                          of $q, p\in\mathcal{P}^{\log}(\Omega)$,
                          resp., with $q^-\leq \overline{q}\leq
                          q^+$ and $p^-\leq \overline{p}\leq p^+$ in
                          $\setR^{d}$, there exist constants
                          $c_{\overline{q}}, c_{\overline{p}}>0$
                          (depending on $\overline{q}, \overline{p} \in\mathcal{P}^{\log}(\setR^d)$, resp.), such that 
			\begin{align*}
			\sup_{h>0}{\| {R}^h_{\Omega}\bfu\|_{L^{{q}(\cdot)}(\Omega)^d}}&\leq \sup_{h>0}{\| {R}^h_{\Omega}\bfu\|_{L^{\overline{q}(\cdot)}(\setR^{d})^d}}\leq c_{\overline{q}}\|\bfu\|_{L^{q(\cdot)}(\Omega)^d}\\
			\sup_{h\in \left(0,h_1\right)}{\|\bfvarepsilon(
                          R^h_{\Omega}\bfu)\|_{L^{{p}(\cdot)}(\Omega)^{d\times
                          d}}}&\leq \sup_{h\in \left(0,h_1\right)}{\|
                                \bfvarepsilon(R^h_{\Omega}\bfu)\|_{L^{\overline{p}(\cdot)}(\setR^{d})^{d\times
                                d}}}\leq
                                c_{\overline{p}}\|\bfvarepsilon(\bfu)\|_{L^{p(\cdot)}(\Omega)^{d\times
                                d}}.
			\end{align*}
			In particular,
                        ${R}^h_{\Omega}:X^{q,p}\to X^{q,p}$ is for every $h\in \left(0,h_1\right)$
                        linear and bounded.
			\item[(iv)] $R^h_{\Omega}\bfu\to \bfu$ in $X^{q,p}$ $(h\to 0)$.
		\end{description}
	\end{cor}
The next proposition collects the properties of the smoothing operator
$(\bscal{R}^h_{Q_T})^\star$ with respect to the space $\bscal{X}^{q,p}_{\bfvarepsilon}(Q_T)$.
	\begin{prop}\label{5.6} For every $\bsu\in \bscal{X}^{q,p}_{\bfvarepsilon}(Q_T)$ it holds:
		\begin{description}[{(iii)}]
			\item[(i)] There exists a constant $c_d>0$ (not depending on $q,p\in \mathcal{P}^{\log}(Q_T)$), such that
			\begin{alignat*}{2}
		\sup_{h\in \left(0,h_1\right)}{\vert \bfvarepsilon((\bscal{R}^h_{Q_T})^\star\bsu)\vert} \leq c_dM_{d+1}(\bscal{F}_{Q_T}\bfvarepsilon(\bsu))\quad\text{ a.e. in }\setR^{d+1}.
			\end{alignat*}
			\item[(ii)] For an extension $\overline{p}\in\mathcal{P}^{\log}(\setR^{d+1})$ of $p\in\mathcal{P}^{\log}(Q_T)$, with $p^-\leq \overline{p}\leq p^+$ in $\setR^{d+1}$, there exist a constant $c_{\overline{p}}>0$ (depending on $\overline{p}\in\mathcal{P}^{\log}(\setR^{d+1})$), such that 
			\begin{align*}
		\sup_{h\in \left(0,h_1\right)}{\| \bfvarepsilon((\bscal{R}^h_{Q_T})^\star\bsu)\|_{L^{\overline{p}(\cdot,\cdot)}(\setR^{d+1})^{d\times d}}} \leq c_{\overline{p}}\| \bfvarepsilon(\bsu)\|_{L^{p(\cdot,\cdot)}(Q_T)^{d\times d}}.
			\end{align*}
			In particular, we have $\sup_{h\in \left(0,h_1\right)}{\|(\bscal{R}^h_{Q_T})^\star\bsu\|_{\bscal{X}^{q,p}_{\bfvarepsilon}(Q_T)}} \leq (c_{\overline{q}}+c_{\overline{p}})\|\bsu\|_{\bscal{X}^{q,p}_{\bfvarepsilon}(Q_T)}$, where $c_{\overline{q}}>0$ is from Proposition \ref{5.2} (iii).\\[-3mm]
				\item[(iii)] $(\bscal{R}^h_{Q_T})^\star\bsu\to \bsu$ in $\bscal{X}^{q,p}_{\bfvarepsilon}(Q_T)$ $(h\to 0)$.
		\end{description}
	\end{prop}

	\begin{proof}
		\textbf{ad (i)} We calculate the symmetric gradient
                of $(\bscal{R}^h_{Q_T})^\star\bsu\in
                C_0^\infty(\setR^{d+1})^d$ for $h>0$. Using anew the product rule
                $\bfvarepsilon(\eta
                \bfu)=\eta\bfvarepsilon(\bfu)+[\bfu\otimes\nb\eta]^{\sym}$
                in $C^0(\setR^{d+1})^d$ valid for every ${\bfu\in
                  C^1(\setR^{d+1})^{d\times d}}$ and $\eta\in
                C^1(\setR^d)$, as well as Proposition \ref{5.2} (i), we obtain for every $h>0$
		\begin{align*}
		\bfvarepsilon((\bscal{R}^h_{Q_T})^\star\bsu)=(\bscal{R}^h_{Q_T})^\star(\bfvarepsilon(\bsu))+[\nb\eta_h\otimes(\bfomega_{h}\ast(\bscal{F}_{Q_T}\bsu))]^{\sym}\quad\text{ in }\setR^{d+1}.
		\end{align*}
		Proposition \ref{5.2} (ii) and Remark \ref{5.2.1} provide
		\begin{align*}
		\sup_{h>0}{\vert (\bscal{R}^h_{Q_T})^\star(\bfvarepsilon(\bsu))\vert }\leq 2M_{d+1}(\bscal{F}_{Q_T}\bfvarepsilon(\bsu))\quad\textup{ a.e. in }\setR^{d+1}.
		\end{align*}
		 For the remaining term note that $[\nb\eta_h\otimes(\bfomega_{h}\ast(\bscal{F}_{Q_T}\bsu))]^{\sym}=\boldsymbol{0}$ in $\setR\times \Omega_{3h}$ for every $h>0$. On the other hand, using \eqref{eq:5.4}, we infer for every $t\in \setR$, $x\in \Omega_{3h}^c$ and $h\in \left(0,h_1\right)$, i.e., $r(x)=\textup{dist}(x,\partial\Omega)\leq 3h\leq 3h_1=\frac{3}{4}h_1<h_0$, that 
		\begin{align}
		\begin{split}
		\vert [\nb\eta_h\otimes(\bfomega_{h}&\ast(\bscal{F}_{Q_T}\bsu))]^{\sym}(t,x)\vert\\&\leq \vert \nb\eta_h(x)\vert\int_{B_{h}^{d+1}(t,x)}{\bfomega_{h}(t-s,x-y)\vert (\bscal{F}_{Q_T}\bsu)(s,y)\vert\,dsdy}\\&\leq \frac{c_d}{h^{d+2}}\int_{B_{h}^1(t)}{\int_{B_{h}^{d}(x)}{\int_{B^d_{2r(y)}(y)}{\frac{\vert(\bscal{F}_{Q_T}\bfvarepsilon(\bsu))(s,z)\vert}{\vert y-z\vert^{d-1}}\,dz}\,dy}\,ds}
		\\&\leq  c_dM_{d+1}(\bscal{F}_{Q_T}\bfvarepsilon(\bsu))(t,x).
		\end{split}\label{eq:5.7}
		\end{align}
		
		\textbf{ad (ii)} Follows from \textbf{(i)} as in the proof Proposition \ref{5.1}. \\[-3mm]
		
		\textbf{ad (iii)} Proposition \ref{5.2} (iv) and Remark \ref{5.2.1} provide $(\bscal{R}^h_{Q_T})^\star\bsu\to \bsu$ in $ L^{q(\cdot,\cdot)}(Q_T)^d$ and $(\bscal{R}^h_{Q_T})^\star\bfvarepsilon(\bsu)\to\bfvarepsilon(\bsu)$ in $ L^{p(\cdot,\cdot)}(Q_T,\mathbb{M}^{d\times d}_{\sym})$ $(h\to 0)$. Due to $[\nb\eta_h\otimes(\bfomega_{h}\ast(\bscal{F}_{Q_T}\bsu))]^{\sym}=\boldsymbol{0}$ in $\setR\times \Omega_{3h}$ for every $h>0$ and \eqref{eq:5.7}, we conclude, using the  Lebegue dominated convergence theorem
		\begin{align*}
		[\nb\eta_h\otimes(\bfomega_{h}\ast(\bscal{F}_{Q_T}\bsu))]^{\sym}\overset{h\to 0}{\to }\boldsymbol{0}\quad\text{ in } L^{p(\cdot,\cdot)}(Q_T,\mathbb{M}^{d\times d}_{\sym}).
		\end{align*} 
		All things considered, we proved $(\bscal{R}^h_{Q_T})^\star\bsu\to \bsu$ in $\bscal{X}^{q,p}_{\bfvarepsilon}(Q_T)$ $(h\to 0)$.\hfill$\qed$
	\end{proof}

Next, we show that even $C^\infty_0(Q_T)^d$ functions are dense in $\bscal{X}^{q,p}_{\bfvarepsilon}(Q_T)$.

	\begin{prop}\label{5.8} 
		Let $(\varphi_h)_{h>0}\subseteq C^\infty_0(I)$ a family of cut-off functions, such that for every $h>0$ it holds
		$\textup{supp}(\varphi_h)\subset I_{2h}:=\left(2h,T-2h\right)$, $0\leq \varphi_h\leq 1$ and $\varphi_h\to 1$ $(h\to 0)$ almost everywhere~in~$I$. For $\bsu\in \bscal{X}^{q,p}_{\bfvarepsilon}(Q_T)$ and $h>0$, we define the smoothing operator
		\begin{align*}
		\mathring{\bscal{R}}^h_{Q_T}\bsu:=\bscal{R}^h_{Q_T}(\varphi_h\bsu)\in C^\infty_0(\mathbb{R}^{d+1})^d.
		\end{align*}
		Then, for every $\bsu\in \bscal{X}^{q,p}_{\bfvarepsilon}(Q_T)$ it holds:
		\begin{description}[{(iii)}]
			\item[(i)]
                          $(\mathring{\bscal{R}}^h_{Q_T}\bsu)_{h>0}\subseteq
                          C^\infty_0(Q_T)^d$ with
                          $\text{supp}(\bscal{R}^h_{Q_T}\bsu)\subseteq
                          (h,T-h)\times \Omega_h$~for~every~$h>0$.
			\item[(ii)] There exists a constant $c_d>0$ (not depending on $q,p\in \mathcal{P}^{\log}(Q_T)$), such that
			\begin{alignat*}{2}
			\sup_{h>0}{\vert \mathring{\bscal{R}}^h_{Q_T}\bsu\vert }&\leq 2M_{d+1}( \bscal{F}_{Q_T}\bsu)&&\quad\text{ a.e. in  }Q_T,\\\sup_{h\in \left(0,h_1\right)}{\vert \bfvarepsilon(\mathring{\bscal{R}}^h_{Q_T}\bsu)\vert }&\leq c_dM_{d+1}( \bscal{F}_{Q_T}\bfvarepsilon(\bsu))&&\quad\textup{ a.e. in }Q_T.
			\end{alignat*}
			\item[(iii)] For extensions $\overline{q},
                          \overline{p}\in\mathcal{P}^{\log}(\setR^{d+1})$
                          of $q, p\in\mathcal{P}^{\log}(Q_T)$,
                          respectively, with $q^-\leq \overline{q}\leq
                          q^+$ and $p^-\leq \overline{p}\leq p^+$ in
                          $\setR^{d+1}$, there exist constants
                          $c_{\overline{q}}, c_{\overline{p}}>0$
                          (depending on $\overline{q}, \overline{p} \in\mathcal{P}^{\log}(\setR^{d+1})$), such that 
			\begin{align*}
                          \sup_{h>0}{\| \mathring {\bscal{R}}^h_{Q_T}\bsu\|_{L^{{q}(\cdot,\cdot)}(Q_T)^d}}
                          &\leq \sup_{h>0}{\|\mathring {\bscal{R}}^h_{Q_T}\bsu\|_{L^{\overline{q}(\cdot,\cdot)}(\setR^{d+1})^d}}\leq
                            c_{\overline{q}}\|\bsu\|_{L^{q(\cdot,\cdot)}(Q_T)^d}
                          \\
                          \sup_{h\in
                          	\left(0,h_1\right)}{\|\bfvarepsilon(\mathring{\bscal{R}}^h_{Q_T}\bsu)
                          \|_{L^{p(\cdot,\cdot)}(Q_T)^{d\times d}}}&\leq \sup_{h\in
                          \left(0,h_1\right)}{\|
                                \bfvarepsilon(\mathring{\bscal {R}}^h_{Q_T}\bsu)\|_{L^{p(\cdot,\cdot)}(\setR^{d+1})^{d\times d}}}\leq
                                c_{\overline{p}}\|\bfvarepsilon(\bsu)\|_{L^{p(\cdot,\cdot)}(Q_T)^{d\times d}}.
			\end{align*}
			In particular,
                        $\mathring {\bscal{R}}^h_{Q_T}: \bscal{X}^{q,p}_{\bfvarepsilon}(Q_T) \to \bscal{X}^{q,p}_{\bfvarepsilon}(Q_T) $ is for every $h\in \left(0,h_1\right)$
                        linear and bounded.
			\item[(iv)] $\mathring {\bscal{R}}^h_{Q_T}\bsu\to \bsu$ in $\bscal{X}^{q,p}_{\bfvarepsilon}(Q_T)$ $(h\to 0)$, i.e., $C^\infty_0(Q_T)^d$ is dense in $\bscal{X}^{q,p}_{\bfvarepsilon}(Q_T)$.
		\end{description}
	\end{prop}
	\begin{proof}
			\textbf{ad (i)} Proposition \ref{5.1} (i) guarantees that $(\mathring{\bscal{R}}^h_{Q_T}\bsu)_{h>0}\subseteq C^\infty_0(\setR^{d+1})^d$. Apart from that, there holds for every $h>0$
			\begin{align*}
					\text{supp}(\mathring{\bscal{R}}^h_{Q_T}\bsu)\subseteq \text{supp}(\varphi_h\eta_h(\bscal{F}_{Q_T}\bsu))+B_{h}^{d+1}(0)\subseteq (I_{2h}\times \Omega_{2h}) +B_{h}^{d+1}(0)\subset I_{h}\times \Omega_{h}\subset\subset Q_T,
			\end{align*}
			i.e., $(\mathring{\bscal{R}}^h_{Q_T}\bsu)_{h>0}\subseteq C^\infty_0(Q_T)^d$. \\[-3mm]
			
			\textbf{ad (ii) \& (iii)} Follow  from
                        Proposition \ref{5.3} and Proposition \ref{5.1} as $\sup_{h>0}{\| \varphi_h\|_{L^\infty(\setR)}}\leq 1$. \\[-3mm]
			
			\textbf{ad (iv)} Using Proposition \ref{5.3} (ii) \& (iii) and Lebesgue's theorem on dominated convergence, we infer
			\begin{align*}
				\|\mathring{\bscal{R}}^h_{Q_T}\bsu-\bsu\|_{\bscal{X}^{q,p}_{\bfvarepsilon}(Q_T)}&\leq\|\bscal{R}^h_{Q_T}[(1-\varphi_h)\bsu]\|_{\bscal{X}^{q,p}_{\bfvarepsilon}(Q_T)}+ \|\bscal{R}^h_{Q_T}\bsu-\bsu\|_{\bscal{X}^{q,p}_{\bfvarepsilon}(Q_T)}\\&\leq (c_{\overline{q}}+c_{\overline{p}})\|(1-\varphi_h)\bsu\|_{\bscal{X}^{q,p}_{\bfvarepsilon}(Q_T)}+ \|\bscal{R}^h_{Q_T}\bsu-\bsu\|_{\bscal{X}^{q,p}_{\bfvarepsilon}(Q_T)}\overset{h\to 0}{\to} 0.\tag*{$\qed$}
			\end{align*}
	\end{proof}
	Eventually, we extent the smoothing operator $\bscal{R}^h_{Q_T}$ by means of  Proposition \ref{4.9} to a smoothing operator for functionals in $\bscal{X}^{q,p}_{\bfvarepsilon}(Q_T)^*$.

	\begin{prop}[Smoothing in
          $\bscal{X}^{q,p}_{\bfvarepsilon}(Q_T)^*$]\label{5.9} For
          $h>0$, we define the smoothing operator
          $\bscal{R}^h_{Q_T}:\bscal{X}^{q,p}_{\bfvarepsilon}(Q_T)^*\to
          \bscal{X}^{q,p}_{\bfvarepsilon}(Q_T)^*$ via 
		\begin{align}\label{eq:R-bidual}
		\langle \bscal{R}^h_{Q_T}\bsu^*,\bsv\rangle
                  _{\bscal{X}^{q,p}_{\bfvarepsilon}(Q_T)}:=\langle
                  \bsu^*, (\bscal{R}^h_{Q_T})^\star\bsv \rangle
                  _{\bscal{X}^{q,p}_{\bfvarepsilon}(Q_T)} 
		\end{align}
		for every  $\bsu^*\in
                \bscal{X}^{q,p}_{\bfvarepsilon}(Q_T)^*$ and
                $\bsv \in \bscal{X}^{q,p}_{\bfvarepsilon}(Q_T)$. Then,  for every $\bsu^*\in \bscal{X}^{q,p}_{\bfvarepsilon}(Q_T)^*$ it holds:
		\begin{description}[{(iii)}]
			\item[(i)] There holds $\sup_{h\in
                            \left(0,h_1\right)}{\|\bscal{R}^h_{Q_T}\bsu^*\|_{\bscal{X}^{q,p}_{\bfvarepsilon}(Q_T)^*}}\leq
                          c_{\overline{q},\overline{p}}\|\bsu^*\|_{\bscal{X}^{q,p}_{\bfvarepsilon}(Q_T)^*}$,
                          where  $c_{\overline{q},\overline{p}}>0$ is the
                          constant from Proposition \ref{5.6} (ii).
			\item[(ii)] $(\bscal{R}^h_{Q_T}\bsu^*)_{h>0}\subseteq \bscal{C}^\infty_*(Q_T):=\bscal{J}_{\bfvarepsilon}(C^\infty(\overline{I},C_0^\infty(\Omega)^d)\times C_0^\infty(\overline{I},C^\infty(\Omega,\mathbb{M}^{d\times d}_{\sym})))$.
			\item[(iii)] $\bscal{R}^h_{Q_T}\bsu^*\to\bsu^*$ in $ \bscal{X}^{q,p}_{\bfvarepsilon}(Q_T)^*$ $(h\!\to\!0)$, i.e., $\bscal{C}^\infty_*(Q_T)$ is dense in~$\bscal{X}^{q,p}_{\bfvarepsilon}(Q_T)^*$.
		\end{description}
	\end{prop}

	\begin{rmk}
		The operator defined on the left-hand side of
                \eqref{eq:R-bidual} is actually the operator
                $((\bscal{R}^h_{Q_T})^\star)^*:\bscal{X}^{q,p}_{\bfvarepsilon}(Q_T)^*\to
          \bscal{X}^{q,p}_{\bfvarepsilon}(Q_T)^*$, i.e., the adjoint
                operator of ${(\bscal{R}^h_{Q_T})^\star :\bscal{X}^{q,p}_{\bfvarepsilon}(Q_T)\to
          \bscal{X}^{q,p}_{\bfvarepsilon}(Q_T)}$, examined in
          Proposition \ref{5.6}. Nevertheless,
                we denote this operator by $\bscal{R}^h_{Q_T}$ as the
                operator defined in \eqref{def:Rh} and studied in
                Propositions \ref{5.1} and \ref{5.3}. This is
                motivated by the fact that $\bscal{R}^h_{Q_T}:\bscal{X}^{q,p}_{\bfvarepsilon}(Q_T)^*\to
		\bscal{X}^{q,p}_{\bfvarepsilon}(Q_T)^*$ can be seen as an extension of the smoothing operator $\bscal{R}^h_{Q_T}:L^{q'(\cdot,\cdot)}(Q_T)^d\to L^{q'(\cdot,\cdot)}(Q_T)^d$, in the sense that for every $\bsu \in L^{q'(\cdot,\cdot)}(Q_T)^d$ there holds the identity $ \bscal{R}^h_{Q_T}\big[\bscal{J}_{\bfvarepsilon}(\bsu,\boldsymbol{0})\big]=\bscal{J}_{\bfvarepsilon}(\bscal{R}^h_{Q_T}\bsu,\boldsymbol{0})$ in $\bscal{X}^{q,p}_{\bfvarepsilon}(Q_T)^*$. More precisely, we have for every $\bsu \in L^{q'(\cdot,\cdot)}(Q_T)^d$ and $\bsv \in  \bscal{X}^{q,p}_{\bfvarepsilon}(Q_T)$
		\begin{align*}
			\langle \bscal{R}^h_{Q_T}\big[\bscal{J}_{\bfvarepsilon}(\bsu,\boldsymbol{0})\big],\bsv\rangle
			_{\bscal{X}^{q,p}_{\bfvarepsilon}(Q_T)}&=\langle
			\bscal{J}_{\bfvarepsilon}(\bsu,\boldsymbol{0}),(\bscal{R}^h_{Q_T})^\star\bsv \rangle
			_{\bscal{X}^{q,p}_{\bfvarepsilon}(Q_T)}=(\bsu,(\bscal{R}^h_{Q_T})^\star\bsv)_{L^{q(\cdot,\cdot)}(Q_T)^d}\\&=(\bscal{R}^h_{Q_T}\bsu,\bsv)_{L^{q(\cdot,\cdot)}(Q_T)^d}=\langle \bscal{J}_{\bfvarepsilon}(\bscal{R}^h_{Q_T}\bsu,\boldsymbol{0}),\bsv\rangle_{\bscal{X}^{q,p}_{\bfvarepsilon}(Q_T)}.
		\end{align*}
	\end{rmk}

	\begin{proof}(of Proposition \ref{5.9})
          \textbf{ad (i)} We have, using Proposition \ref{5.6} (ii) 
          \begin{align*}
            \|\bscal{R}^h_{Q_T}\bsu^*\|_{\bscal{X}^{q,p}_{\bfvarepsilon}(Q_T)^*}
            &=\sup_{\|\bsv\|_{\bscal{X}^{q,p}_{\bfvarepsilon}(Q_T)}\le 1}
              \langle \bsu^*, (\bscal{R}^h_{Q_T})^\star\bsv 
              \rangle_{\bscal{X}^{q,p}_{\bfvarepsilon}(Q_T)}
            \\
            &\le \sup_{\|\bsv\|_{\bscal{X}^{q,p}_{\bfvarepsilon}(Q_T)}\le 1}\|\bsu^*\|_{\bscal{X}^{q,p}_{\bfvarepsilon}(Q_T)^*} 
              \|(\bscal{R}^h_{Q_T})^\star\bsv\|_{\bscal{X}^{q,p}_{\bfvarepsilon}(Q_T)}
            \\
            &\le c_{\overline{q},\overline{p}} \|\bsu^*\|_{\bscal{X}^{q,p}_{\bfvarepsilon}(Q_T)^*}.
          \end{align*}

          \textbf{ad (ii)} Let $\bsf\in L^{q'(\cdot,\cdot)}(Q_T)^d$
          and  $\bsF\in L^{p'(\cdot,\cdot)}(Q_T,\mathbb{M}_{\sym}^{d\times
            d})$ be such that
          $\bsu^*=\bscal{J}_{\bfvarepsilon}(\bsf,\bsF)$ in
          $\bscal{X}^{q,p}_{\bfvarepsilon}(Q_T)^*$. Then, we infer for
          every $\bfphi\in C_0^\infty(Q_T)^d$ and $h>0$
		\begin{align*}
		\langle \bscal{R}^h_{Q_T}\bsu^*,\bfphi\rangle_{\bscal{X}^{q,p}_{\bfvarepsilon}(Q_T)}&=\langle \bscal{J}_{\bfvarepsilon}(\bsf,\bsF),(\bscal{R}^h_{Q_T})^\star\bfphi\rangle_{\bscal{X}^{q,p}_{\bfvarepsilon}(Q_T)}\\&=\langle \bscal{J}_{\bfvarepsilon}(\bscal{R}^h_{Q_T}\bsf,\bscal{R}^h_{Q_T}\bsF),\bfphi\rangle_{\bscal{X}^{q,p}_{\bfvarepsilon}(Q_T)}\\&\quad+(\bsF,[\nb\eta_h\otimes(\bfomega_h\ast(\bscal{F}_{Q_T}\bfphi))]^{\sym})_{L^{p(\cdot,\cdot)}(Q_T)^{d\times d}},
		\end{align*}
		i.e., for all $h>0$ holds
                \begin{align}\label{eq:rep}
			\bscal{R}^h_{Q_T}\bsu^*=\bscal{J}_{\bfvarepsilon}(\bscal{R}^h_{Q_T}\bsf+\bfomega_h\ast(\bsF\nb\eta_h),\bscal{R}^h_{Q_T}\bsF)\quad\text{ in }\bscal{X}^{q,p}_{\bfvarepsilon}(Q_T)^*
		\end{align}
                by the density of $C_0^\infty(Q_T)^d$ in $\bscal{X}^{q,p}_{\bfvarepsilon}(Q_T)$ (cf.~Proposition~\ref{5.8}), and hence \textbf{(ii)}.\\[-3mm]
		
		\textbf{ad (iii)} 
		Let $\bsf\in L^{q'(\cdot,\cdot)}(Q_T)^d$ and $\bsF\in
                L^{p'(\cdot,\cdot)}(Q_T,\mathbb{M}^{d\times
                  d}_{\sym})$ be such that
                $\bsu^*=\bscal{J}_{\bfvarepsilon}(\bsf,\bsF)$ in
                $\bscal{X}^{q,p}_{\bfvarepsilon}(Q_T)^*$.
		Using the continuity of $\bscal{J}_{\bfvarepsilon}:L^{q'(\cdot,\cdot)}(Q_T)^d\times L^{p'(\cdot,\cdot)}(Q_T)^{d\times d}\to \bscal{X}^{q,p}_{\bfvarepsilon}(Q_T)^*$ (cf.~Proposition~\ref{4.9}), Proposition~\ref{5.1}~(iv) and Remark~\ref{5.1.1}, we get
		 \begin{align}\label{eq:conv}
		 \bscal{J}_{\bfvarepsilon}(\bscal{R}^h_{Q_T}\bsf,\bscal{R}^h_{Q_T}\bsF)\overset{h\to 0}{\to }\bsu^*\quad\text{ in }\bscal{X}^{q,p}_{\bfvarepsilon}(Q_T)^*.
		 \end{align}
		 Due to ${\sup_{h\in \left(0,h_1\right)}{\vert [\nabla\eta_h\otimes(\boldsymbol{\omega}_h\ast(\bscal{F}_{Q_T}\bsu))]^{\text{sym}}\vert\leq c_dM_{d+1}(\bscal{F}_{Q_T}\boldsymbol{ \varepsilon}(\bsx))}}$~in~$\mathbb{R}^{d+1}$ for every $\bsu\in \bscal{X}^{q,p}_{\bfvarepsilon}(Q_T)$ (cf.~\eqref{eq:5.7}),~we~have for every $\bsu\in \bscal{X}^{q,p}_{\bfvarepsilon}(Q_T)$ 
		 \begin{align*}
		 \begin{split}
		 \sup_{h\in \left(0,h_1\right)}{\|[\nb\eta_h\otimes(\bfomega_h\ast(\bscal{F}_{Q_T}\bsu))]^{\sym}\|_{L^{p(\cdot,\cdot)}(Q_T)^{d\times d}}}\leq c_{\overline{p}}\|\bfvarepsilon(\bsu)\|_{L^{p(\cdot,\cdot)}(Q_T)^{d\times d}}\leq c_{\overline{p}}\|\bsu\|_{\bscal{X}^{q,p}_{\bfvarepsilon}(Q_T)}.
		 \end{split}
		 \end{align*}
		Hence, since $[\nb\eta_h\otimes(\bfomega_h\ast(\bscal{F}_{Q_T}\bsu))]^{\sym}=\textbf{0}$ in $I\times\Omega_{3h}$ for every $h>0$ and $\bsu\in \bscal{X}^{q,p}_{\bfvarepsilon}(Q_T)$, there holds
		\begin{align*}
					\|\bscal{J}_{\bfvarepsilon}(\omega_h\ast(\bsF\nb\eta_h),\boldsymbol{0})\|_{\bscal{X}^{q,p}_{\bfvarepsilon}(Q_T)^*}&=\sup_{\substack{\bsu\in \bscal{X}^{q,p}_{\bfvarepsilon}(Q_T)\\\|\bsu\|_{\bscal{X}^{q,p}_{\bfvarepsilon}(Q_T)}\leq 1}}{\!\!(\bsF\chi_{I\times\Omega_{3h}^c},\nb\eta_h\otimes(\bfomega_h\ast(\bscal{F}_{Q_T}\bsu)))_{L^{p(\cdot,\cdot)}(Q_T)^{d\times d}}}\\&\leq c_{\overline{p}}\|\bsF\chi_{I\times\Omega_{3h}^c}\|_{L^{p'(\cdot,\cdot)}(Q_T)^{d\times d}}\overset{h\to 0}{\to }0.
		\end{align*}
                This together with \eqref{eq:conv} and \eqref{eq:rep}
                yields \textbf{(iii)}.\hfill $\qed$
	\end{proof}

	\section{Generalized time derivative and formula of integration by parts}
	\label{sec:6}
	
	Since the introduction of variable exponent Bochner--Lebsgue
        spaces was driven by the inability of usual Bochner--Lebesgue
        spaces to properly encode unsteady problems with variable
        exponent non-linearity, such as the model problem
        \eqref{eq:model}, the notion of a generalized time derivative
        which lives in Bochner--Lebesgue spaces from Subsection
        \ref{subsec:2.2}  is  inappropriate as well. It is for this
        reason why we next introduce a special notion of a generalized
        time derivative, which is now supposed to be a functional in
        $\bscal{X}^{q,p}_{\bfvarepsilon}(Q_T)^*$ and which is aligned
        with the usual notion of distributional derivatives on the
        time-space cylinder $Q_T$.
	
	In what follows, let $\Omega\subseteq \setR^d$, $d\ge 2$, be a bounded Lipschitz domain, $I=\left(0,T\right)$, $T<\infty$, $Q_T=I\times \Omega$ and $q,p\in \mathcal{P}^{\log}(Q_T)$ with $q^-,p^->1$, such that $X^{q,p}_-\embedding Y: =L^2(\Omega)^d$, e.g., if $p^-\ge \frac{2d}{d+2}$.
	
	\begin{defn}\label{6.1}  A function $\bsu\in\bscal{X}^{q,p}_{\bfvarepsilon}(Q_T)$ possesses a \textbf{generalized time derivative in} $\bscal{X}^{q,p}_{\bfvarepsilon}(Q_T)^*$, if there exists a functional $\bsu^*\in \bscal{X}^{q,p}_{\bfvarepsilon}(Q_T)^*$,  such that for every $\bfphi\in C^\infty_0(Q_T)^d$ it holds
		\begin{align}
		-\int_I{(\bsu(t),\pa_t\bfphi(t))_Y\,dt}=\langle\bsu^*,\bfphi\rangle_{\bscal{X}^{q,p}_{\bfvarepsilon}(Q_T)}.\label{eq:6.1}
		\end{align}
	In this case, we set $\frac{\textbf{d}\bsu}{\textbf{dt}}:=\bsu^*$ in $\bscal{X}^{q,p}_{\bfvarepsilon}(Q_T)^*$.
	\end{defn}

	\begin{lem}\label{6.2} The generalized time derivative in the sense of Definition \ref{6.1} is unique.
	\end{lem}
	
	\begin{proof}
		Suppose that for $\bsu\in \bscal{X}^{q,p}_{\bfvarepsilon}(Q_T)$ there exist $\bsu^*_1,\bsu^*_2\in \bscal{X}^{q,p}_{\bfvarepsilon}(Q_T)^*$, such that for $i=1,2$ it holds \eqref{eq:6.1}. Thus, we have $\langle \bsu^*_1,\bfphi\rangle_{\bscal{X}^{q,p}_{\bfvarepsilon}(Q_T)}=\langle \bsu^*_2,\bfphi\rangle_{\bscal{X}^{q,p}_{\bfvarepsilon}(Q_T)}$ for every $\bfphi\in C^\infty_0(Q_T)^d$. Since the space $C^\infty_0(Q_T)^d$ is dense in $\bscal{X}^{q,p}_{\bfvarepsilon}(Q_T)$ (cf.~Proposition~\ref{5.8}), we infer  $\bsu^*_1=\bsu^*_2$ in $\bscal{X}^{q,p}_{\bfvarepsilon}(Q_T)^*$.\hfill$\qed$
	\end{proof}
	
	\begin{defn}\label{6.3}
		We define the \textbf{variable exponent Bochner--Sobolev space}
		\begin{align*}
		\bscal{W}^{q,p}_{\bfvarepsilon}(Q_T):=
		\bigg\{\bsu\in \bscal{X}^{q,p}_{\bfvarepsilon}(Q_T)\;\bigg|\; \exists\;\frac{\mathbf{d}\bsu}{\mathbf{dt}}\in 	\bscal{X}^{q,p}_{\bfvarepsilon}(Q_T)^*\bigg\}.
		\end{align*}
	\end{defn}
	
	\begin{prop}\label{6.4}
		The space $\bscal{W}_{\bfvarepsilon}^{q,p}(Q_T)$ equipped with the norm
		\begin{align*}
		\|\cdot\|_{\bscal{W}_{\bfvarepsilon}^{q,p}(Q_T)}:=\|\cdot\|_{\bscal{X}_{\bfvarepsilon}^{q,p}(Q_T)}+\left\|\frac{\mathbf{d}}{\mathbf{dt}}\cdot\right\|_{\bscal{X}_{\bfvarepsilon}^{q,p}(Q_T)^*}
		\end{align*}
		is a separable, reflexive Banach space.
	\end{prop}
	
	\begin{proof}
		
		\textbf{1. Completeness:} Let $(\bsu_n)_{n\in\setN}\!\subseteq\!\bscal{W}_{\bfvarepsilon}^{q,p}(Q_T)$ be a Cauchy sequence. Thus, $(\bsu_n)_{n\in\setN}\subseteq\bscal{X}_{\bfvarepsilon}^{q,p}(Q_T)$ and $(\frac{\textbf{d}\bsu_n}{\textbf{dt}})_{n\in\setN}\!\subseteq\!\bscal{X}_{\bfvarepsilon}^{q,p}(Q_T)^*$ are Cauchy sequences as well. Hence, there exist elements ${\bsu\!\in\! \bscal{X}_{\bfvarepsilon}^{q,p}(Q_T)}$ and $\bsu^*\!\in\! \bscal{X}_{\bfvarepsilon}^{q,p}(Q_T)^*$, such that
		\begin{align}
		\begin{split}
		\begin{alignedat}{2}
		\bsu_n\overset{n\to \infty}{\to }\bsu\quad\text{ in }\bscal{X}_{\bfvarepsilon}^{q,p}(Q_T),\qquad\frac{\textbf{d}\bsu_n}{\textbf{dt}}\overset{n\to \infty}{\to }\bsu^*\quad\text{ in }\bscal{X}_{\bfvarepsilon}^{q,p}(Q_T)^*.
		\end{alignedat}
		\end{split}\label{eq:6.4.1}
		\end{align}
		From \eqref{eq:6.4.1} we infer for every $\bfphi\in
                C_0^\infty(Q_T)^d$, using 
                $\bscal{X}_{\bfvarepsilon}^{q,p}(Q_T) \vnor L^{\min\{p^-,q^-\}}(I,Y)$, that
		\begin{align*}
		\begin{split}
		\langle \bsu^*,\bfphi\rangle_{\bscal{X}^{q,p}_{\bfvarepsilon}(Q_T)}\!\overset{n\to\infty}{\leftarrow}\!\left\langle \frac{\textbf{d}\bsu_n}{\textbf{dt}},\bfphi\right\rangle_{\bscal{X}^{q,p}_{\bfvarepsilon}(Q_T)}\!=\!-\int_I{(\bsu_n(s),\pa_t\bfphi(s))_Y\,ds}\!\overset{n\to\infty}{\to}\!-\int_I{(\bsu(s),\pa_t\bfphi(s))_Y\,ds},
		\end{split}
		\end{align*}
		i.e., $\bsu\!\in \!\bscal{W}^{q,p}_{\bfvarepsilon}(Q_T)$ with $\frac{\textbf{d}\bsu}{\textbf{dt}}\!=\!\bsu^*$ in $\bscal{X}^{q,p}_{\bfvarepsilon}(Q_T)^*$. In particular, \eqref{eq:6.4.1} now reads $\bsu_n\!\to \!\bsu$~in~$\bscal{W}^{q,p}_{\bfvarepsilon}(Q_T)$ $(n\to \infty)$. Altogether, $\bscal{W}^{q,p}_{\bfvarepsilon}(Q_T)$ is a Banach space. \\[-3mm]
		
		\textbf{2. Reflexivity and separability:} $\bscal{W}^{q,p}_{\bfvarepsilon}(Q_T)$ inherits the reflexivity and separability from the product space $\bscal{X}^{q,p}_{\bfvarepsilon}(Q_T)\times \bscal{X}^{q,p}_{\bfvarepsilon}(Q_T)^*$ in virtue of the isometric isomorphism
		\begin{align*}
		\boldsymbol{\Lambda}_{\bfvarepsilon}:\bscal{W}^{q,p}_{\bfvarepsilon}(Q_T)\to R(\boldsymbol{\Lambda}_{\bfvarepsilon})\subseteq\bscal{X}^{q,p}_{\bfvarepsilon}(Q_T)\times \bscal{X}^{q,p}_{\bfvarepsilon}(Q_T)^*,
		\end{align*}
		which is for every $\bsu\in \bscal{W}^{q,p}_{\bfvarepsilon}(Q_T)$ given via $\boldsymbol{\Lambda}_{\bfvarepsilon}\bsu:=(\bsu,\frac{\textbf{d}\bsu}{\textbf{dt}})^\top$ in $\bscal{X}^{q,p}_{\bfvarepsilon}(Q_T)\times \bscal{X}^{q,p}_{\bfvarepsilon}(Q_T)^*$.\hfill$\qed$
	\end{proof}
	
	Since by assumption $\mathbf{id}_{X_+^{q,p}}:X_+^{q,p}\to Y$ and $\mathbf{id}_{X_-^{q,p}}:X_-^{q,p}\to Y$  are embeddings, which are
	inevitably also dense, the corresponding  adjoint operators $(\mathbf{id}_{X_+^{q,p}})^*:Y^*\to (X_+^{q,p})^*$ and $(\mathbf{id}_{X_-^{q,p}})^*:Y^*\to (X_-^{q,p})^*$ are embeddings as well. Consequently, also the mappings
	\begin{align*}
	e_-&:=(\mathbf{id}_{X_+^{q,p}})^*R_Y\mathbf{id}_{X_-^{q,p}}:X_-^{q,p}\to (X_+^{q,p})^*,\\
	e_+&:=(\mathbf{id}_{X_-^{q,p}})^*R_Y\mathbf{id}_{X_+^{q,p}}:X_+^{q,p}\to (X_-^{q,p})^*,
	\end{align*}
	where $R_Y:Y\to Y^*$ is the Riesz isomorphism with respect to $Y$, are embeddings. This guarantees the well-posedness of  the following limiting Bochner--Sobolev spaces (cf.~Section~\ref{subsec:2.2}).
	
	\begin{defn}\label{6.5}
		We define the \textbf{limiting Bochner--Sobolev spaces}
		\begin{align*}
		\bscal{W}^{q,p}_+(Q_T)&:=W_{e_+}^{1,\max\{q^+,p^+\},\max\{(q^-)',(p^-)'\}}(I,X^{q,p}_+,X^{q,p}_-,Y),\\
		\bscal{W}^{q,p}_-(Q_T)&:=W_{e_-}^{1,\min\{q^-,p^-\},\min\{(q^+)',(p^+)'\}}(I,X^{q,p}_-,X^{q,p}_+,Y).
		\end{align*}
	\end{defn}

	The following proposition examines the exact relation between $\bscal{W}^{q,p}_{\bfvarepsilon}(Q_T)$ and $ \bscal{W}^{q,p}_-(Q_T)$.
	
	\begin{prop}\label{6.6}
		For $\bsu\in \bscal{X}^{q,p}_{\bfvarepsilon}(Q_T)$ and $\bsu^*\in \bscal{X}^{q,p}_{\bfvarepsilon}(Q_T)^*$ the following statements are equivalent:
		\begin{description}[(ii)]
			\item[(i)]  
			$\bsu\in \bscal{W}^{q,p}_{\bfvarepsilon}(Q_T)$ with $\frac{\mathbf{d}\bsu}{\mathbf{dt}}=\bsu^*$ in $\bscal{X}^{q,p}_{\bfvarepsilon}(Q_T)^*$.
			\item[(ii)] For every $\bfv\in X^{q,p}_+$ and
                          $\varphi\in C^\infty_0(I)$ there holds
			\begin{align*}
			-\int_I{(\bsu(s),\bfv)_Y\varphi^\prime(s)\,ds}=\int_I{\langle \bsu^*(s),\bfv\rangle_{X^{q,p}(s)}\varphi(s)\,ds},
			\end{align*}
i.e., $\bsu\in \bscal{W}^{q,p}_-(Q_T)$ with $\frac{d_{e_-}\bsu}{dt}=\bscal{J}_{X_+^{q,p}}^{-1}\big[(\mathbf{id}_{\bscal{X}^{q,p}_+(Q_T)})^*\bsu^*\big]$ in $L^{\min\{(q^+)',(p^+)'\}}(I,(X_+^{q,p})^*)$, where we denote by ${(\mathbf{id}_{\bscal{X}^{q,p}_+(Q_T)})^*}$ the adjoint  of ${\mathbf{id}_{\bscal{X}^{q,p}_+(Q_T)}:\bscal{X}^{q,p}_+(Q_T)\to \bscal{X}^{q,p}_{\bfvarepsilon}(Q_T)}$ and $\bscal{J}_{X_+^{q,p}}:L^{\min\{(q^+)',(p^+)'\}}(I,(X_+^{q,p})^*)\to \bscal{X}^{q,p}_+(Q_T)^*$ is the isomorphism from Proposition \ref{riesz-boch}.
		\end{description}
	\end{prop}

\begin{proof}
	\textsf{\textbf{(i) $\boldsymbol{\Rightarrow}$ (ii)}} Due to $ \bscal{X}^{q,p}_{\bfvarepsilon}(Q_T)\embedding \bscal{X}^{q,p}_-(Q_T)$ (cf.~Proposition~\ref{4.6}), we have $\bsu\in \bscal{X}^{q,p}_-(Q_T)$.
	Choosing $\bfphi:=\bfv\varphi\in C_0^\infty(Q_T)^d$ for $\bfv\in C_0^\infty(\Omega)^d$ and $\varphi\in C_0^\infty(I)$ in \eqref{eq:6.1}, we obtain
	\begin{align*}
			-\int_I{(\bsu(s),\bfv)_Y\varphi^\prime(s)\,ds}&=	-\int_I{(\bsu(s),\partial_t\bfphi(s))_Y\,ds}\\&=\langle \bsu^*,\bfphi\rangle_{\bscal{X}^{q,p}_{\bfvarepsilon}(Q_T)}
				=\langle (\mathbf{id}_{\bscal{X}^{q,p}_+(Q_T)})^*\bfu^*,\bfphi\rangle_{\bscal{X}^{q,p}_+(Q_T)}\\&=\int_I{\langle \bscal{J}_{X_+^{q,p}}^{-1}\big[(\mathbf{id}_{\bscal{X}^{q,p}_+(Q_T)})^*\bsu^*\big](s),\bfv\rangle_{X_+^{q,p}}\varphi(s)\,ds},
	\end{align*}
	i.e., $\bsu\in \bscal{W}^{q,p}_-(Q_T)$ with $\frac{d_{e_-}\bsu}{dt}=\bscal{J}_{X_+^{q,p}}^{-1}\big[(\mathbf{id}_{\bscal{X}^{q,p}_+(Q_T)})^*\bsu^*\big]$ in $L^{\min\{(q^+)',(p^+)'\}}(I,(X_+^{q,p})^*)$ (cf.~\eqref{eq:gtd}).\\[-3mm]
	
	\textsf{\textbf{(ii) $\boldsymbol{\Rightarrow}$ (i)}} Let
        $\bfphi\in C_0^\infty(Q_T)^d$. Since $\bfphi\in
        \bscal{W}_+^{q,p}(Q_T)$ with
        $\frac{d_{e_+}\bfphi}{dt}=e_+(\partial_t\bfphi)$ in
        $\bscal{X}_-^{q,p}(Q_T)^*$, we obtain using \eqref{eq:iden},
        the integration by parts formula in time according to
        \cite[Lemma A.1]{Wol17} and $\frac{d_{e_-}\bsu}{dt}=\bscal{J}_{X_+^{q,p}}^{-1}\big[(\mathbf{id}_{\bscal{X}^{q,p}_+(Q_T)})^*\bsu^*\big]$ in $L^{\max\{(q^-)',(p^-)'\}}(I,(X_+^{q,p})^*)$
	\begin{align*}
	-\int_I{(\bsu(s),\partial_t\bfphi(s))_Y\,ds}&=	-\int_I{\langle e_+(\partial_t\bfphi(s)), \bsu(s)\rangle_{X_-^{q,p}}\,ds}=-\int_I{\bigg\langle\frac{d_{e_+}\bfphi}{dt}(s),\bsu(s)\bigg\rangle_{X_-^{q,p}}\,ds}\\&=\int_I{\bigg\langle\frac{d_{e_-}\bsu}{dt}(s),\bfphi(s)\bigg\rangle_{X_+^{q,p}}\,ds}=\langle\bsu^*,\bfphi\rangle_{\bscal{X}^{q,p}_{\bfvarepsilon}(Q_T)},
	\end{align*}
	i.e., $\bsu\in \bscal{W}^{q,p}_{\bfvarepsilon}(Q_T)$ with $\frac{\textbf{d}\bsu}{\textbf{dt}}=\bsu^*$ in $\bscal{X}^{q,p}_{\bfvarepsilon}(Q_T)^*$ (cf.~Definition~\ref{6.1}).\hfill$\qed$
\end{proof}

	\begin{rmk}\label{compare}
		Proposition \ref{6.6} also states that $
                \bscal{W}^{q,p}_{\bfvarepsilon}(Q_T)\subseteq
                \bscal{W}^{q,p}_-(Q_T)$. In particular, comparing the
                norms on $ \bscal{W}^{q,p}_{\bfvarepsilon}(Q_T)$  and
                $\bscal{W}^{q,p}_-(Q_T)$, it is readily seen that even
                ${\bscal{W}^{q,p}_{\bfvarepsilon}(Q_T)\embedding
                  \bscal{W}^{q,p}_-(Q_T)}$. Moreover, similar
                arguments as for Proposition \ref{6.6} show that $
                {\bscal{W}^{q,p}_+(Q_T)\embedding
                  \bscal{W}^{q,p}_{\bfvarepsilon}(Q_T)}$, i.e., in
                total we have the chain of  embeddings  
		\begin{align}
		\bscal{W}^{q,p}_+(Q_T)\embedding \bscal{W}^{q,p}_{\bfvarepsilon}(Q_T)\embedding \bscal{W}^{q,p}_-(Q_T).\label{eq:emb}
		\end{align}
	\end{rmk}

	The spaces $\bscal{W}^{q,p}_+(Q_T)$ and $\bscal{W}^{q,p}_-(Q_T)$ allow us to embed $\bscal{W}^{q,p}_{\bfvarepsilon}(Q_T)$ into the standard theory of Bochner--Sobolev spaces.
	To be more precise, \eqref{eq:emb} ensures in many situations access to the theory of Bochner--Sobolev spaces, even though $ \bscal{W}^{q,p}_{\bfvarepsilon}(Q_T)$  does not meet this framework. Above all, $\bscal{W}^{q,p}_+(Q_T)$ and $\bscal{W}^{q,p}_-(Q_T)$, i.e., particularly Proposition \ref{6.6}, enable us to extend the common method of in time extension via reflection to the framework of variable exponent Bochner--Sobolev spaces, which is crucial for the extension of the smoothing operator from Proposition \ref{5.1} to a smoothing operator for $ \bscal{W}^{q,p}_{\bfvarepsilon}(Q_T)$.
	
	\begin{prop}[Time extension via reflection]\label{6.7}
		For $\bsu\in \bscal{X}^{q,p}_{\bfvarepsilon}(Q_T)$ we define the \textbf{in time extension via reflection} for every $t\in 3I:=\left(-T,2T\right)$ by
		\begin{align}
		(\widehat{\bscal{F}}_T\bsu)(t):=\begin{cases}
		\bsu(-t)&\text{ for }t\in\left(-T,0\right]\\
		\bsu(t)&\text{ for }t\in\left(0,T\right)\\
		\bsu(2T-t)&\text{ for }t\in\left[T,2T\right)
		\end{cases}\quad\textup{ in }X^{q_T,p_T}(t),\label{extintime}
		\end{align}
		where for $Q_{3T}:=3I\times \Omega$ we set $q_T:=\widehat{\bscal{F}}_Tq\in \mathcal{P}^{\log}(Q_{3T})$ and $p_T:=\widehat{\bscal{F}}_Tp\in \mathcal{P}^{\log}(Q_{3T})$.
		Then, it holds:
		\begin{description}
			\item[(i)] $\widehat{\bscal{F}}_T: \bscal{X}^{q,p}_{\bfvarepsilon}(Q_T)\to \bscal{X}^{q_T,p_T}_{\bfvarepsilon}\!(Q_{3T})$ is well-defined, linear and Lipschitz~continuous with constant~$3$.
			\item[(ii)] $\widehat{\bscal{F}}_T:\bscal{W}^{q,p}_{\bfvarepsilon}(Q_T)\to \bscal{W}^{q_T,p_T}_{\bfvarepsilon}(Q_{3T})$ is well-defined, linear and Lipschitz continuous with constant~$12$. To be more precise, if $\bsu\in \bscal{W}^{q,p}_{\bfvarepsilon}(Q_T)$ with $\frac{\mathbf{d}\bsu}{\mathbf{dt}}=\bscal{J}_{\bfvarepsilon}(\bsf,\bsF)$ in $\bscal{X}^{q,p}_{\bfvarepsilon}(Q_T)^*$ for $\bsf\in L^{q'(\cdot,\cdot)}(Q_T)^d$ and $\bsF\in L^{p'(\cdot,\cdot)}(Q_T,\mathbb{M}_{\sym}^{d\times d})$, then $\widehat{\bscal{F}}_T\bsu\in \bscal{W}^{q_T,p_T}_{\bfvarepsilon}(Q_{3T})$ with
			\begin{align}
			\frac{\mathbf{d}\widehat{\bscal{F}}_T\bsu}{\mathbf{dt}}=\bscal{J}_{\bfvarepsilon}(\widehat{\bscal{F}}_T^-\bsf,\widehat{\bscal{F}}_T^-\bsF)\quad\textup{ in }\bscal{X}^{q_T,p_T}_{\bfvarepsilon}(Q_{3T})^*,\label{eq:6.7.a}
			\end{align}
			where $\widehat{\bscal{F}}_T^-\bsf\in L^{q_T'(\cdot,\cdot)}(Q_{3T})^d$ is defined as $(\widehat{\bscal{F}}_T^-\bsf)(t):=\bsf(t)$ if $t\in I$ and $(\widehat{\bscal{F}}_T^-\bsf)(t):=-(\widehat{\bscal{F}}_T\bsf)(t)$ if $t\in 3I \setminus \overline{I}$,
			and $\widehat{\bscal{F}}_T^-\bsF\in L^{p_T'(\cdot,\cdot)}(Q_{3T},\mathbb{M}^{d\times d}_{\sym})$ analogously.
		\end{description}
	\end{prop}
	
	\begin{proof}
		\textbf{ad (i)} Let $\bsu\in \bscal{X}^{q,p}_{\bfvarepsilon}(Q_T)$. Then, we have $\widehat{\bscal{F}}_T\bsu\in L^{q_T(\cdot,\cdot)}(Q_{3T})^d$, $\bfvarepsilon(\widehat{\bscal{F}}_T\bsu)=\widehat{\bscal{F}}_T\bfvarepsilon(\bsu)\in L^{p_T(\cdot,\cdot)}(Q_{3T},\mathbb{M}^{d\times d}_{\sym})$ and $(\widehat{\bscal{F}}_T\bsu)(t)\in X^{q_T,p_T}(t)$ for almost every $t\in 3I$, i.e., $\widehat{\bscal{F}}_T\bsu\in \bscal{X}_{\bfvarepsilon}^{q_T,p_T}(Q_{3T})$ and $\|\widehat{\bscal{F}}_T\bsu\|_{\bscal{X}_{\bfvarepsilon}^{q_T,p_T}(Q_{3T})}\leq 3\|\bsu\|_{\bscal{X}_{\bfvarepsilon}^{q,p}(Q_T)}$.\\[-3mm]
		
		\textbf{ad (ii)} Let $\bsf\in
                L^{q'(\cdot,\cdot)}(Q_T)^d$ and $\bsF\in
                L^{p'(\cdot,\cdot)}(Q_T,\mathbb{M}_{\sym}^{d\times
                  d})$ be such that $\frac{\mathbf{d}\bsu}{\mathbf{dt}}=\bscal{J}_{\bfvarepsilon}(\bsf,\bsF)$ in $\bscal{X}^{q,p}_{\bfvarepsilon}(Q_T)^*$. Proposition
                \ref{6.6} yields $\bsu\in  \bscal{W}^{q,p}_-(Q_T)$
                satisfying
                $\frac{d_{e_-}\bsu}{dt}=\bscal{J}_{X_+^{q,p}}^{-1} (\mathbf{id}_{\bscal{X}^{q,p}_{+}(Q_T)})^*\bscal{J}_{\bfvarepsilon}(\bsf,\bsF)$ 
                in $L^{\min\{(q^+)',(p^+)'\}}(I,(X_+^{q,p})^*)$,
                where $\bscal{J}_{X_+^{q,p}}: L^{\min\{(q^+)',(p^+)'\}}(I,(X_+^{q,p})^*)\to
                \bscal{X}^{q,p}_+(Q_T)^*$ is the isomorphism from
                Proposition \ref{riesz-boch}.  Therefore, \cite[Proposition 2.3.1]{Dro01} additionally provides
		$\widehat{\bscal{F}}_T\bsu\in \bscal{W}^{q_T,p_T}_-(Q_{3T})$ with
		\begin{align}
		\frac{d_{e_{-}}\widehat{\bscal{F}}_T\bsu}{dt}&=\bscal{F}^-_T\frac{d_{e_{-}}\bsu}{dt}\label{eq:6.7.3.1}\\&=\bscal{J}_{X_+^{q,p}}^{-1}\big[(\mathbf{id}_{\bscal{X}^{q_T,p_T}_+(Q_{3T})})^*\bscal{J}_{\bfvarepsilon}(\widehat{\bscal{F}}_T^-\bsf,\widehat{\bscal{F}}_T^-\bsF)\big]\quad\text{ in }L^{\min\{(q_T^+)',(p_T^+)'\}}(3I,(X_+^{q_T,p_T})^*).\notag
		\end{align}
		By means of Proposition \ref{6.6} we conclude from \eqref{eq:6.7.3.1}  that $\widehat{\bscal{F}}_T\bsu\in \bscal{W}^{q_T,p_T}_{\bfvarepsilon}(Q_{3T})$ with \eqref{eq:6.7.a}. Exploiting the Lipschitz continuity of $\bscal{J}_{\bfvarepsilon}:L^{q'_T(\cdot,\cdot)}(Q_{3T})^d\times L^{p'_T(\cdot,\cdot)}(Q_{3T},\mathbb{M}_{\sym}^{d\times d})\to \bscal{X}_{\bfvarepsilon}^{q_T,p_T}(Q_{3T})^*$ with constant $2$ (cf.~Proposition~\ref{4.9}), we deduce from \eqref{eq:6.7.a} that 
		\begin{align}
		\begin{split}
		\bigg\|\frac{\mathbf{d} \widehat{\bscal{F}}_T\bsu}{\mathbf{dt}}\bigg\|_{\bscal{X}_{\bfvarepsilon}^{q_T,p_T}(Q_{3T})^*}&\leq 2\big[\|\widehat{\bscal{F}}_T^-\bsf\|_{L^{q'_T(\cdot,\cdot)}(Q_T)^d}+\|\widehat{\bscal{F}}_T^-\bsF\|_{L^{p'_T(\cdot,\cdot)}(Q_T,\mathbb{M}_{\sym}^{d\times d})}\big]\\&= 6\big[\|\bsf\|_{L^{q'(\cdot,\cdot)}(Q_T)^d}+\|\bsF\|_{L^{p'(\cdot,\cdot)}(Q_T,\mathbb{M}_{\sym}^{d\times d})}\big].
		\end{split}\label{eq:6.7.3}
		\end{align}
		Apart from that, Proposition~\ref{4.9} guarantees the existence of functions $\bsf_0\in L^{q'(\cdot,\cdot)}(Q_T)^d$ and $\bsF_0\in L^{p'(\cdot,\cdot)}(Q_T,\mathbb{M}_{\sym}^{d\times d})$ satisfying $\frac{\mathbf{d}\bsu}{\mathbf{dt}}=\bscal{J}_{\bfvarepsilon}(\bsf_0,\bsF_0)$ in $\bscal{X}^{q,p}_{\bfvarepsilon}(Q_T)^*$ and
		\begin{align}
		\begin{split}
		\frac{1}{2}\bigg\|\frac{\mathbf{d}\bsu}{\mathbf{dt}}\bigg\|_{\bscal{X}_{\bfvarepsilon}^{q,p}(Q_T)^*}
		\leq \|\bsf_0\|_{L^{q'(\cdot,\cdot)}(Q_T)^d}+\|\bsF_0\|_{L^{p'(\cdot,\cdot)}(Q_T,\mathbb{M}_{\sym}^{d\times d})}
		\leq 2\bigg\|\frac{\mathbf{d} \bsu}{\mathbf{dt}}\bigg\|_{\bscal{X}_{\bfvarepsilon}^{q,p}(Q_T)^*}.
		\end{split}\label{eq:6.7.4}
		\end{align}
		Thus, using \eqref{eq:6.7.4} in \eqref{eq:6.7.3}, we obtain
		\begin{align}
		\begin{split}
		\bigg\|\frac{\mathbf{d} \widehat{\bscal{F}}_T\bsu}{\mathbf{dt}}\bigg\|_{\bscal{X}_{\bfvarepsilon}^{q_T,p_T}(Q_{3T})^*}
		&\leq 6\big[\|\bsf_0\|_{L^{q'(\cdot,\cdot)}(Q_T)^d}+\|\bsF_0\|_{L^{p'(\cdot,\cdot)}(Q_T,\mathbb{M}_{\sym}^{d\times d})}\big]
		\\&\leq 12\bigg\|\frac{\mathbf{d} \bsu}{\mathbf{dt}}\bigg\|_{\bscal{X}_{\bfvarepsilon}^{q,p}(Q_T)^*}.
		\end{split}\label{eq:6.7.5}
		\end{align}
		Combining \eqref{eq:6.7.5} and \textbf{(i)}, we conclude the Lipschitz continuity of ${\widehat{\bscal{F}}_T:\bscal{W}^{q,p}_{\bfvarepsilon}(Q_T)\to \bscal{W}^{q_T,p_T}_{\bfvarepsilon}(Q_{3T})}$.~\hfill$\qed$
	\end{proof}

	Combining Proposition \ref{6.7} and the smoothing operator from Proposition \ref{5.1} yields a smoothing operator
	for $\bscal{W}^{q,p}_{\bfvarepsilon}(Q_T)$.
	
	\begin{prop}[Smoothing in $\bscal{W}^{q,p}_{\bfvarepsilon}(Q_T)$]\label{6.8}
		For every $\bsu\in \bscal{W}^{q,p}_{\bfvarepsilon}(Q_T)$ and $h>0$, we define 
		\begin{align*}
			\widehat{\bscal{R}}_{Q_T}^h\bsu:=\bscal{R}_{Q_{3T}}^h(\widehat{\bscal{F}}_T\bsu)\in C^\infty(\overline{I},C_0^\infty(\Omega)^d).
		\end{align*}
		Then, for every $\bsu\in \bscal{W}^{q,p}_{\bfvarepsilon}(Q_T)$  it holds:
		\begin{description}[{(iii)}]
			\item[(i)] $(\widehat{\bscal{R}}_{Q_T}^h\bsu)_{h>0}\subseteq  \bscal{W}^{q,p}_{\bfvarepsilon}(Q_T)$ with 
			\begin{align*}
			\frac{\mathbf{d}}{\mathbf{dt}}\widehat{\bscal{R}}_{Q_T}^h\bsu=\bscal{F}_{Q_T}^*\bscal{R}_{Q_{3T}}^h\bigg[\frac{\mathbf{d}\widehat{\bscal{F}}_T\bsu}{\mathbf{dt}}\bigg]\quad\textup{ in }\bscal{X}_{\bfvarepsilon}^{q,p}(Q_T)^*
			\end{align*}
			for every $h>0$, where $\bscal{F}_{Q_T}^*:\bscal{X}_{\bfvarepsilon}^{q_T,p_T}(Q_{3T})^*\to \bscal{X}_{\bfvarepsilon}^{q,p}(Q_T)^*$ denotes the adjoint operator of the zero extension operator $\bscal{F}_{Q_T}:\bscal{X}_{\bfvarepsilon}^{q,p}(Q_T)\to \bscal{X}_{\bfvarepsilon}^{q_T,p_T}(Q_{3T})$. 
			\item[(ii)] $	\widehat{\bscal{R}}_{Q_T}^h\bsu\to\bsu$ in $\bscal{W}^{q,p}_{\bfvarepsilon}(Q_T)$ $(h\to 0)$, i.e., $C^\infty(\overline{I},C_0^\infty(\Omega)^d)$ is dense in $\bscal{W}^{q,p}_{\bfvarepsilon}(Q_T)$.
			\item[(iii)] For extensions $\overline{q}_T,\overline{p}_T\in \mathcal{P}^{\log}(\setR^{d+1})$ of $q_T,p_T\in \mathcal{P}^{\log}(Q_{3T})$, with ${q^-\leq\overline{q}_T\leq q^+}$ and ${p^-\leq\overline{p}_T\leq p^+}$ in $\setR^{d+1}$, there exists a constant $c_{\overline{q}_T,\overline{p}_T}>0$, such that 
			\begin{align*}
			\sup_{h\in \left(0,h_1\right)\cap I}{\big\|	\widehat{\bscal{R}}_{Q_T}^h\bsu\big\|_{\bscal{W}^{q,p}_{\bfvarepsilon}(Q_T)}}\leq c_{\overline{q}_T,\overline{p}_T}\|\bsu\|_{\bscal{W}^{q,p}_{\bfvarepsilon}(Q_T)}.
			\end{align*}
		\end{description}
	\end{prop}
	
	\begin{proof}
		\textbf{ad (i)} Due to $(\bscal{R}_{Q_{3T}}^h(\widehat{\bscal{F}}_T\bsu))_{h>0}\!\subseteq\! C_0^{\infty}(\setR^{d+1})^d$ with ${\textup{supp}(\bscal{R}_{Q_{3T}}^h(\widehat{\bscal{F}}_T\bsu))\!\subseteq\! \left(-T-h,2T+h\right)\!\times\! \Omega_{h}}$ for every $h>0$, we have $(\widehat{\bscal{R}}_{Q_T}^h\bsu)_{h>0}\subseteq C^\infty(\overline{I},C_0^\infty(\Omega)^d)\subseteq \bscal{W}^{q,p}_{\bfvarepsilon}(Q_T)$ with 
		\begin{align}
		\frac{\textbf{d}}{\textbf{dt}}\widehat{\bscal{R}}_{Q_T}^h\bsu=\bscal{J}_{\bfvarepsilon}(\pa_t\widehat{\bscal{R}}_{Q_T}^h\bsu,\boldsymbol{0})\quad\text{ in }\bscal{X}^{q,p}_{\bfvarepsilon}(Q_T)^*\label{eq:6.8.1}
		\end{align}
		for every $h>0$.
		Using \eqref{eq:6.8.1} and the formula of integration by parts in time for smooth vector fields, we infer for every $\bfphi\in C^\infty_0(Q_T)^d$ and $h>0$
		\begin{align}
		\begin{split}
		\bigg\langle \frac{\textbf{d}}{\textbf{dt}}\widehat{\bscal{R}}_{Q_T}^h\bsu,\bfphi\bigg\rangle_{\bscal{X}_{\bfvarepsilon}^{q,p}(Q_T)}&=(\pa_t\widehat{\bscal{R}}_{Q_T}^h\bsu,\bfphi)_{L^{q(\cdot,\cdot)}(Q_T)^d}\\&=-(\pa_t\bfphi, \widehat{\bscal{R}}_{Q_T}^h\bsu)_{L^{q(\cdot,\cdot)}(Q_T)^d}\\&=-(\bscal{F}_{Q_T}\pa_t\bfphi, \bscal{R}_{Q_{3T}}^h(\widehat{\bscal{F}}_T\bsu))_{L^{q_T(\cdot,\cdot)}(Q_{3T})^d}\\&=-((\bscal{R}^h_{Q_{3T}})^\star(\bscal{F}_{Q_T}\pa_t\bfphi),\widehat{\bscal{F}}_T\bsu)_{L^{q_T(\cdot,\cdot)}(Q_{3T})^d}
		\\&=-(\pa_t[(\bscal{R}^h_{Q_{3T}})^\star(\bscal{F}_{Q_T}\bfphi)],\widehat{\bscal{F}}_T\bsu)_{L^{q_T(\cdot,\cdot)}(Q_{3T})^d}\\&=-\int_{3I}{((\widehat{\bscal{F}}_T\bsu)(t),\pa_t[(\bscal{R}^h_{Q_{3T}})^\star(\bscal{F}_{Q_T}\bfphi)](t))_Y\,dt}.
		\end{split}\label{eq:6.8.2}
		\end{align}
		Since $((\bscal{R}^h_{Q_{3T}})^\star(\bscal{F}_{Q_T}\bfphi))_{h\in I}\subseteq  C^\infty_0(Q_{3T})^d$, as we have $((\bscal{R}^h_{Q_{3T}})^\star(\bscal{F}_{Q_T}\bfphi))_{h\in I}\subseteq  C^\infty_0(\mathbb{R}^{d+1})^d$ with ${\textup{supp}((\bscal{R}^h_{Q_{3T}})^\star(\bscal{F}_{Q_T}\bfphi))\subseteq \textup{supp}(\bfphi)+B_h^{d+1}(0)\subset\subset Q_{3T}}$ for every $h\in I$ due to Proposition~\ref{5.2}~(i), we further deduce from \eqref{eq:6.8.2} by means of Proposition \ref{6.7}~(ii) and Definition \ref{6.1} that for every $\bfphi\in C_0^\infty(Q_T)^d$  and $h\in I$ it holds
		\begin{align}
		\begin{split}
		\bigg\langle \frac{\textbf{d}}{\textbf{dt}}\widehat{\bscal{R}}_{Q_T}^h\bsu,\bfphi\bigg\rangle_{\bscal{X}_{\bfvarepsilon}^{q,p}(Q_T)}&=\bigg\langle \frac{\textbf{d}\widehat{\bscal{F}}_T\bsu}{\textbf{dt}},(\bscal{R}_{Q_{3T}}^h)^\star(\bscal{F}_{Q_T}\bfphi)\bigg\rangle_{\bscal{X}_{\bfvarepsilon}^{q_T,p_T}(Q_{3T})}\\&=\bigg\langle \bscal{F}_{Q_T}^*\bscal{R}_{Q_{3T}}^h\bigg[\frac{\textbf{d}\widehat{\bscal{F}}_T\bsu}{\textbf{dt}}\bigg],\bfphi\bigg\rangle_{\bscal{X}_{\bfvarepsilon}^{q,p}(Q_T)}.
		\end{split}\label{eq:6.8.3}
		\end{align}
		Due to the density of $C_0^\infty(Q_T)^d$ in $\bscal{X}_{\bfvarepsilon}^{q,p}(Q_T)$ (cf.~Proposition \ref{5.8}), we  conclude \textbf{(i)} from \eqref{eq:6.8.3}.\\[-3mm]
		
		\textbf{ad (ii)} Proposition \ref{5.3} (i) immediately provides $\bscal{R}_{Q_{3T}}^h(\widehat{\bscal{F}}_T\bsu)\to \widehat{\bscal{F}}_T\bsu$ in $\bscal{X}_{\bfvarepsilon}^{q_T,p_T}(Q_{3T})$~${(h\to 0)}$, from which we immediately infer that $\widehat{\bscal{R}}_{Q_T}^h\bsu\to \bsu$ in $\bscal{X}_{\bfvarepsilon}^{q,p}(Q_T)$ $(h\to 0)$.
		From \textbf{(i)}, Proposition~\ref{5.9}~(iii) and the continuity of $\bscal{F}_{Q_T}^*:\bscal{X}_{\bfvarepsilon}^{q_T,p_T}(Q_{3T})^*\to \bscal{X}_{\bfvarepsilon}^{q,p}(Q_T)^*$, we deduce
		\begin{align*}
		\frac{\textbf{d}}{\textbf{dt}}\widehat{\bscal{R}}_{Q_T}^h\bsu\overset{h\to 0}{\to }\bscal{F}_{Q_T}^*\frac{\textbf{d}\widehat{\bscal{F}}_T\bsu}{\textbf{dt}}=\frac{\textbf{d}\bsu}{\textbf{dt}}\quad\text{ in }\bscal{X}_{\bfvarepsilon}^{q,p}(Q_T)^*.
		\end{align*}
		
		\textbf{ad (iii)} Proposition \ref{5.1}~(iv) and Proposition \ref{5.3}~(iii) immediately provide for~every~${h\in \left(0,h_1\right)}$ 
		\begin{align}
		\big\|\widehat{\bscal{R}}_{Q_T}^h\bsu\big\|_{\bscal{X}^{q,p}_{\bfvarepsilon}(Q_T)}\leq c_{\overline{q}_T,\overline{p}_T}\|\widehat{\bscal{F}}_T\bsu\|_{\bscal{X}^{q_T,p_T}_{\bfvarepsilon}(Q_{3T})}.\label{eq:6.8.5}
		\end{align}
		Apart from that, we deduce from \textbf{(i)}, Proposition \ref{5.9} (ii), and the Lipschitz continuity of  $\bscal{F}_{Q_T}^*:\bscal{X}^{q_T,p_T}_{\bfvarepsilon}(Q_{3T})^*\to \bscal{X}^{q,p}_{\bfvarepsilon}(Q_T)^*$ (with constant $1$) for every $h\in \left(0,h_1\right)\cap I$ 
		\begin{align}
		\begin{split}
		\bigg\|\frac{\textbf{d}}{\textbf{dt}}\widehat{\bscal{R}}_{Q_T}^h\bsu\bigg\|_{\bscal{X}^{q,p}_{\bfvarepsilon}(Q_T)^*}&=\bigg\|\bscal{F}_{Q_T}^*\bscal{R}_{Q_{3T}}^h\bigg[\frac{\textbf{d}\widehat{\bscal{F}}_T\bsu}{\textbf{dt}}\bigg]\bigg\|_{\bscal{X}^{q,p}_{\bfvarepsilon}(Q_T)^*}\\&\leq \bigg\|\bscal{R}_{Q_{3T}}^h\bigg[\frac{\textbf{d}\widehat{\bscal{F}}_T\bsu}{\textbf{dt}}\bigg]\bigg\|_{\bscal{X}^{q_T,p_T}_{\bfvarepsilon}(Q_{3T})^*}\\&\leq c_{\overline{q}_T,\overline{p}_T}\bigg\|\frac{\textbf{d}\widehat{\bscal{F}}_T\bsu}{\textbf{dt}}\bigg\|_{\bscal{X}^{q_T,p_T}_{\bfvarepsilon}(Q_{3T})^*}.
		\end{split}\label{eq:6.8.6}
		\end{align}
		Finally, combining \eqref{eq:6.8.5}, \eqref{eq:6.8.6} and the Lipschitz continuity of $\widehat{\bscal{F}}_T:\bscal{W}^{q,p}_{\bfvarepsilon}(Q_T)\to \bscal{W}^{q_T,p_T}_{\bfvarepsilon}(Q_{3T})$ (cf.~Proposition \ref{6.7}~(ii)), we conclude \textbf{(iii)}.\hfill$\qed$
	\end{proof}
	
	\begin{prop}\label{6.9}
		Let $\bscal{Y}^0(Q_T):=C^0(\overline{I},Y)$. The following statements hold true:
		\begin{description}
			\item[(i)] Any function $\bsu\in \bscal{W}_{\bfvarepsilon}^{q,p}(Q_T)$ (defined almost everywhere) possesses a unique representation $\bsu_{c}\in \bscal{Y}^0(Q_T)$, and the resulting mapping $(\cdot)_{c}:\bscal{W}_{\bfvarepsilon}^{q,p}(Q_T)\to
			\bscal{Y}^0(Q_T)$ is an embedding.
			\item[(ii)] For every $\bsu,\bsv\in \bscal{W}_{\bfvarepsilon}^{q,p}(Q_T)$ and
			$t,t'\in \overline{I}$ with $t'\leq t$ it holds
			\begin{align}
			\int_{t'}^{t}{\left\langle
				\frac{\mathbf{d}\bsu}{\mathbf{dt}}(s),\bsv(s)\right\rangle_{X^{q,p}(s)}\!\!ds}
			=\left[(\bsu_c(s), 
			\bsv_c(s))_Y\right]^{s=t}_{s=t'}-\int_{t'}^{t}{\left\langle
				\frac{\mathbf{d}\bsv}{\mathbf{dt}}(s),\bsu(s)\right\rangle_{X^{q,p}(s)}\!\!ds}.\label{eq:6.9.a}
			\end{align}
		\end{description}  
	\end{prop}
	
	\begin{proof} 
		\textbf{ad (i)} Let $\bsu\in \bscal{W}^{q,p}_{\bfvarepsilon}(Q_T)$. We choose a family of smooth approximations $(\bsu_h )_{h\in \left(0,h_1\right)}:=(\widehat{\bscal{R}}_{Q_T}^h\bsu)_{h\in \left(0,h_1\right)}\subseteq C^\infty(\overline{I},C_0^\infty(\Omega)^d)$ (cf.~Proposition \ref{6.8}), such that $\bsu_h\to\bsu$ in $\bscal{W}^{q,p}_{\bfvarepsilon}(Q_T)$ $(h\to 0)$.
		Then, for every $t,t'\in \overline I$ and $h',h\in \left(0,h_1\right)$ it holds
		\begin{align}
		\|\bsu_h(t)-\bsu_{h'}(t)\|_Y^2&=\|\bsu_h(t')-\bsu_{h'}(t')\|_Y^2\notag\\&\qquad+2\int_{t'}^{t}{\int_{\Omega}{\big(\partial_t\bsu_h(s,y)-\partial_t\bsu_{h'}(s,y)\big)\cdot(\bsu_h(s,y)-\bsu_{h'}(s,y))\,dy}\,ds}\notag\\&=\|\bsu_h(t')-\bsu_{h'}(t')\|_Y^2\label{eq:6.9.1}\\&\qquad+2\int_{t'}^{t}{\left\langle\frac{\textbf{d}\bsu_h}{\textbf{dt}}(s)-\frac{\textbf{d}\bsu_{h'}}{\textbf{dt}}(s),\bsu_h(s)-\bsu_{h'}(s)\right\rangle_{X^{q,p}(s)}\!\!ds}\notag\\&
		\leq \|\bsu_h(t')-\bsu_{h'}(t')\|_Y^2+2\bigg\|\frac{\textbf{d}\bsu_h}{\textbf{dt}}-\frac{\textbf{d}\bsu_{h'}}{\textbf{dt}}\bigg\|_{\bscal{X}^{q,p}_{\bfvarepsilon}(Q_T)^*}\|\bsu_h-\bsu_{h'}\|_{\bscal{X}^{q,p}_{\bfvarepsilon}(Q_T)}\notag,
		\end{align}
		where we used for the second equality \eqref{eq:6.8.1}
                as well as Proposition~\ref{4.9}, Remark \ref{3.9} and Corollary~\ref{3.8}.
		As $\sqrt{a^2+b^2}\leq a+b$ for every $a,b\ge 0$, we infer from \eqref{eq:6.9.1} for every $t,t'\in \overline I$ and $h',h\in \left(0,h_1\right)$
		\begin{align}
		\begin{split}
		\|\bsu_h(t)-\bsu_{h'}(t)\|_Y&\leq\|\bsu_h(t')-\bsu_{h'}(t')\|_Y\\&\quad+\sqrt{2}\bigg\|\frac{\textbf{d}\bsu_h}{\textbf{dt}}-\frac{\textbf{d}\bsu_{h'}}{\textbf{dt}}\bigg\|_{\bscal{X}^{q,p}_{\bfvarepsilon}(Q_T)^*}^{\frac{1}{2}}\|\bsu_h-\bsu_{h'}\|_{\bscal{X}^{q,p}_{\bfvarepsilon}(Q_T)}^{\frac{1}{2}}.
		\end{split}\label{eq:6.9.2}
		\end{align}
		Integrating \eqref{eq:6.9.2} with respect $t'\in I$ and dividing by $T>0$, yields further for every $t\in \overline I$ and $h',h\in \left(0,h_1\right)$
		\begin{align}
		\begin{split}
		\|\bsu_h(t)-\bsu_{h'}(t)\|_Y&\leq T^{-1}\|\bsu_h-\bsu_{h'}\|_{L^1(I,Y)}\\&\quad+\sqrt{2}\bigg\|\frac{\textbf{d}\bsu_h}{\textbf{dt}}-\frac{\textbf{d}\bsu_{h'}}{\textbf{dt}}\bigg\|_{\bscal{X}^{q,p}_{\bfvarepsilon}(Q_T)^*}^{\frac{1}{2}}\|\bsu_h-\bsu_{h'}\|_{\bscal{X}^{q,p}_{\bfvarepsilon}(Q_T)}^{\frac{1}{2}}.
		\end{split}
		\label{eq:6.9.3}
		\end{align}
		As $t\!\in\!\overline{I}$ was arbitrary in
                \eqref{eq:6.9.3} and
                $\bscal{X}^{q,p}_{\bfvarepsilon}(Q_T)\embedding
                L^1(I,Y)$, since
                $X^{q,p}_-\embedding Y$, we infer for all ${h',h \in \left(0,h_1\right)}$
		\begin{align}
		\begin{split}
		\|\bsu_h-\bsu_{h'}\|_{\bscal{Y}^0(Q_T)} &\leq T^{-\frac{1}{2}}2(1+\vert Q_T\vert)\|\bsu_h-\bsu_{h'}\|_{\bscal{X}^{q,p}_{\bfvarepsilon}(Q_T)}\\&\quad+\sqrt{2}\bigg\|\frac{\textbf{d}\bsu_h}{\textbf{dt}}-\frac{\textbf{d}\bsu_{h'}}{\textbf{dt}}\bigg\|_{\bscal{X}^{q,p}_{\bfvarepsilon}(Q_T)^*}^{\frac{1}{2}}\|\bsu_h-\bsu_{h'}\|_{\bscal{X}^{q,p}_{\bfvarepsilon}(Q_T)}^{\frac{1}{2}}.
		\end{split}\label{eq:6.9.4}
		\end{align}
		Hence, $(\bsu_h)_{h\in \left(0,h_1\right)}\subseteq \bscal{Y}^0(Q_T)$ is a Cauchy sequence, i.e., there exists $\bsu_c\in \bscal{Y}^0(Q_T)$, such that
		\begin{align}
		\bsu_h\overset{h\to 0}{\to}\bsu_c\quad\text{ in }\bscal{Y}^0(Q_T).\label{eq:6.9.5}
		\end{align}
		Since $\bscal{Y}^0(Q_T)\embedding L^1(I,Y)^d$, \eqref{eq:6.9.5} also gives us
		$\bsu_h\to\bsu_c$ in $L^1(I,Y)^d$ $(h\to 0)$. As simultaneously $\bsu_h\to \bsu$ in $\bscal{X}^{q,p}_{\bfvarepsilon}(Q_T)\embedding L^1(I,Y)^d$ $(h\to 0)$, we obtain $\bsu_c=\bsu$ almost everywhere in $Q_T$. Hence, each $\bsu\in \bscal{W}^{q,p}_{\bfvarepsilon}(Q_T)$ possesses a unique continuous representation $\bsu_c\in \bscal{Y}^0(Q_T)$, i.e., the mapping $(\cdot)_c:\bscal{W}^{q,p}_{\bfvarepsilon}(Q_T)\to \bscal{Y}^0(Q_T) $ is well-defined. If we repeat the steps \eqref{eq:6.9.1}--\eqref{eq:6.9.4} and replace $\bsu_h-\bsu_{h'}$ by $\bsu_h$ in doing so, we infer for every $h\in \left(0,h_1\right)$
		\begin{align}
		\|\bsu_h\|_{\bscal{Y}^0(Q_T)}\leq  T^{-\frac{1}{2}}2(1+\vert Q_T\vert)\|\bsu_h\|_{\bscal{X}^{q,p}_{\bfvarepsilon}(Q_T)}+\sqrt{2}\bigg\|\frac{\textbf{d}\bsu_h}{\textbf{dt}}\bigg\|_{\bscal{X}^{q,p}_{\bfvarepsilon}(Q_T)^*}^{\frac{1}{2}}\|\bsu_h\|_{\bscal{X}^{q,p}_{\bfvarepsilon}(Q_T)}^{\frac{1}{2}}.\label{eq:6.9.6}
		\end{align}
		By passing in \eqref{eq:6.9.6} for $h\to 0$, using $\bsu_h\to\bsu$ in $\bscal{W}^{q,p}_{\bfvarepsilon}(Q_T)$ $(h\to 0)$ and \eqref{eq:6.9.5}, we obtain 
		\begin{align}
		\|\bsu_c\|_{\bscal{Y}^0(Q_T)}\leq T^{-\frac{1}{2}}2(1+\vert Q_T\vert)\|\bsu\|_{\bscal{X}^{q,p}_{\bfvarepsilon}(Q_T)}+\sqrt{2}\bigg\|\frac{\textbf{d}\bsu}{\textbf{dt}}\bigg\|_{\bscal{X}^{q,p}_{\bfvarepsilon}(Q_T)^*}^{\frac{1}{2}}\|\bsu\|_{\bscal{X}^{q,p}_{\bfvarepsilon}(Q_T)}^{\frac{1}{2}},\label{eq:6.9.6.1}
		\end{align}
		i.e., $(\cdot)_c:\bscal{W}^{q,p}_{\bfvarepsilon}(Q_T)\to \bscal{Y}^0(Q_T) $ is an embedding.\\[-3mm]
		
		\textbf{ad (ii)} Let $\bsu,\bsv\in \bscal{W}^{q,p}_{\bfvarepsilon}(Q_T)$. We choose $(\bsu_h)_{h\in \left(0,h_1\right)}:=(\widehat{\bscal{R}}_{Q_T}^h\bsu)_{h\in \left(0,h_1\right)},(\bsv_h )_{h\in \left(0,h_1\right)}:=(\widehat{\bscal{R}}_{Q_T}^h\bsv)_{h\in \left(0,h_1\right)}\subseteq C^\infty(\overline{I},C_0^\infty(\Omega)^d)$, such that
		\begin{align}
		\bsu_h\overset{h\to 0}{\to }\bsu\quad\text{ in }\bscal{W}^{q,p}_{\bfvarepsilon}(Q_T),\qquad
		\bsv_h\overset{h\to 0}{\to }\bsv\quad\text{ in }\bscal{W}^{q,p}_{\bfvarepsilon}(Q_T).\label{eq:6.9.7}
		\end{align}
		Thanks to \textbf{(i)}, i.e., the continuity of $(\cdot)_c:\bscal{W}^{q,p}_{\bfvarepsilon}(Q_T)\to \bscal{Y}^0(Q_T) $, we even have 
		\begin{align}
		\bsu_h=(\bsu_h)_c\overset{h\to 0}{\to }\bsu_c\quad\text{ in } \bscal{Y}^0(Q_T),\qquad
		\bsv_h=(\bsv_h)_c\overset{h\to 0}{\to }\bsv_c\quad\text{ in } \bscal{Y}^0(Q_T).\label{eq:6.9.8}
		\end{align}
		Using the classical formula of integration by parts
                and
                $$
                \int_{t'}^{t}{\int_{\Omega}{\partial_t\bsu_h(s,y)\cdot
                    \bsv_h(s,y)\,dy}\,ds}=\int_{t'}^{t}{\left\langle\frac{\textbf{d}\bsu_h}{\textbf{dt}}(s),\bsv_h(s)\right\rangle_{X^{q,p}(s)}\,ds},
                $$
                (cf.~\eqref{eq:6.9.1}), we obtain for every $h\in \left(0,h_1\right)$
		\begin{align}
		\int_{t'}^{t}{\left\langle
			\frac{\textbf{d}\bsu_h}{\textbf{dt}}(s),\bsv_h(s)\right\rangle_{X^{q,p}(s)}\!\!\!ds}=\left[(\bsu_h(s),\bsv_h(s))_H\right]_{s=t'}^{s=t}-\int_{t'}^{t}{\left\langle
			\frac{\textbf{d}\bsv_h}{\textbf{dt}}(s),\bsu_h(s)\right\rangle_{X^{q,p}(s)}\!\!\!ds}.\label{eq:6.9.9}
		\end{align}
		Eventually, by passing in \eqref{eq:6.9.9} for $h\to 0$, using \eqref{eq:6.9.7} and \eqref{eq:6.9.8}, we obtain the desired formula of integration by parts  \eqref{eq:6.9.a}.\hfill$\qed$
	\end{proof}

	\section{Existence result}
	
	\label{sec:7}
	This section adresses an abstract existence result, which will be based on the theory of maximal monotone mappings. The basic idea consists in interpreting the generalized time derivative as a maximal monotone mapping and then to introduce a special notion of pseudo-monotonicity which is directly coupled to the time derivative. To the best of the author's knowledge this approach finds its origin in \cite{Lio69}, was extended by Brezis in \cite{Bre72}, and was first applied in the framework of variable exponent Bochner--Lebesgue spaces in the rather simplified setting of monotone operators by Alkhutov and Zhikov in \cite{AZ10}.
	
	In what follows, let $\Omega\subseteq \setR^d$ be a bounded Lipschitz domain, $I=\left(0,T\right)$, with $0<T<\infty$, and $q,p\in \mathcal{P}^{\log}(Q_T)$ with $q^-,p^-> 1$ such that $X^{q,p}_-\embedding Y =L^2(\Omega)^d$. 
	
	\begin{prop}\label{7.0}
		The   \textbf{initial trace operator} $\bfgamma_0:\bscal{W}^{q,p}_{\bfvarepsilon}(Q_T)\to Y$, given via
		${\bfgamma_0(\bsu):=\bsu_c(0)}$ in $Y$ for every $\bsu\in \bscal{W}^{q,p}_{\bfvarepsilon}(Q_T)$, is well-defined, linear and continuous. Moreover, it holds:
		\begin{description}[(iii)]
			\item[(i)] $R(\bfgamma_0)$ is dense in $Y$.
			\item[(ii)] $G((\frac{\bfd}{\bfd\bft},\bfgamma_0)^\top)$ is closed in $\bscal{X}^{q,p}_{\bfvarepsilon}(Q_T)\times\bscal{X}^{q,p}_{\bfvarepsilon}(Q_T)^*\times Y$.
			\item[(iii)] $\langle \frac{\bfd\bsu}{\bfd\bft},\bsu\rangle_{\bscal{X}^{q,p}_{\bfvarepsilon}(Q_T)}+\frac{1}{2}\|\bfgamma_0(\bsu)\|_Y^2\ge 0$ for every $\bsu\in  \bscal{W}^{q,p}_{\bfvarepsilon}(Q_T)$.
		\end{description}
	\end{prop}

	\begin{proof}
		Well-definedness, linearity and continuity result from
                Proposition~\ref{6.9}~(i)  together with the
                properties of the strong convergence in $\bscal{Y}^0(Q_T)$. So, let us verify \textbf{(i)}--\textbf{(iii)}:
		
		\textbf{ad (i)} For arbitrary $\bfu\in X_+^{q,p}$ the function $\bsu\in  \bscal{W}^{q,p}_{\bfvarepsilon}(Q_T)$, given via $\bsu(t):=\bfu$ in $X_+^{q,p}$ for every $t\in I$, satisfies $\bfgamma_0(\bsu)=\bfu$. Thus, $ X_+^{q,p}\subseteq R(\bfgamma_0)$, i.e., $R(\bfgamma_0)$ is dense in $Y$, as $X_+^{q,p}$ is dense in $Y$.
		
		\textbf{ad (ii)} Let $((\bsu_n,\frac{\bfd\bsu_n}{\bfd\bft},\bfgamma_0(\bsu_n))^\top)_{n\in \setN}\subseteq G((\frac{\bfd}{\bfd\bft},\bfgamma_0)^\top)$ be a sequence, such that 
		\begin{align}
			\bigg(\bsu_n,\frac{\bfd\bsu_n}{\bfd\bft},\bfgamma_0(\bsu_n)\bigg)^\top\overset{n\to\infty}{\to }(\bsu,\bsu^*,\bfu_0)^\top\quad\textup{ in }\bscal{X}^{q,p}_{\bfvarepsilon}(Q_T)\times\bscal{X}^{q,p}_{\bfvarepsilon}(Q_T)^*\times Y.\label{eq:7.0.1}
		\end{align}
		Apparently, \eqref{eq:7.0.1} also implies that
                $(\bsu_n,\frac{\bfd\bsu_n}{\bfd\bft})^\top\to
                (\bsu,\bsu^*)^\top$ in
                ${\bscal{X}^{q,p}_{\bfvarepsilon}(Q_T)\times\bscal{X}^{q,p}_{\bfvarepsilon}(Q_T)^*}$ 
                ${(n\to \infty)}$, which 
		due to the closedness of $G(\frac{\bfd}{\bfd\bft})=\bscal{W}^{q,p}_{\bfvarepsilon}(Q_T)$ (cf.~Proposition~\ref{6.4}) yields that
		$\bsu\in \bscal{W}^{q,p}_{\bfvarepsilon}(Q_T)$ with
                $\frac{\bfd\bsu}{\bfd\bft}=\bsu^*$ in
                $\bscal{X}^{q,p}_{\bfvarepsilon}(Q_T)^*$ and
                $\bsu_n\to \bsu$ in
                $\bscal{W}^{q,p}_{\bfvarepsilon}(Q_T)$ $(n\to
                \infty)$. Since ${\bfgamma_0:\bscal{W}^{q,p}_{\bfvarepsilon}(Q_T)\to Y}$ is continuous, we conclude that ${\bfgamma_0(\bsu_n)\to \bfgamma_0(\bsu)}$ in $Y$ $(n\to\infty)$, i.e., $\bfgamma_0(\bsu):=\bfu_0$ in $Y$, as simultaneously $\bfgamma_0(\bsu_n)\to \bfu_0$ in $Y$ $(n\to\infty)$. Altogether,  $(\bsu,\bsu^*,\bfu_0)^\top=(\bsu,\frac{\bfd\bsu}{\bfd\bft},\bfgamma_0(\bsu))\in G((\frac{\bfd}{\bfd\bft},\bfgamma_0)^\top)$, i.e., $G((\frac{\bfd}{\bfd\bft},\bfgamma_0)^\top)$ is closed in $\bscal{X}^{q,p}_{\bfvarepsilon}(Q_T)\times\bscal{X}^{q,p}_{\bfvarepsilon}(Q_T)^*\times Y$.\\[-3mm]
		
		\textbf{ad (iii)} Immediate consequence of \eqref{eq:6.9.a}.\hfill$\qed$
	\end{proof}
	
	\begin{prop}[Maximal monotonicity of $\frac{\bfd}{\bfd\bft}$]
		\label{7.1}
		The generalized time derivative  \linebreak${\frac{\bfd}{\bfd\bft}:\bscal{W}^{q,p}_{\bfvarepsilon}(Q_T)\subseteq \bscal{X}^{q,p}_{\bfvarepsilon}(Q_T)\to \bscal{X}^{q,p}_{\bfvarepsilon}(Q_T)^*}$ is linear,  densely defined and its restriction to \linebreak${N(\bfgamma_0):=\{\bsu\in \bscal{W}^{q,p}_{\bfvarepsilon}(Q_T)\fdg\bsu_c(0)=\mathbf{0}\text{ in }Y\}}$ is maximal monotone.
	\end{prop}
	
	\begin{proof}
		\textbf{1. Linearity:} Follows right from the definition of $\frac{\bfd}{\bfd\bft}:\bscal{W}^{q,p}_{\bfvarepsilon}(Q_T)\subseteq \bscal{X}^{q,p}_{\bfvarepsilon}(Q_T)\to \bscal{X}^{q,p}_{\bfvarepsilon}(Q_T)^*$ (cf.~Definition~\ref{6.1}).\\[-3mm]
		
		\textbf{2. Density of domain of definition:} Since $C_0^\infty(Q_T)^d\subseteq \bscal{W}^{q,p}_{\bfvarepsilon}(Q_T)$ is dense in $\bscal{X}^{q,p}_{\bfvarepsilon}(Q_T)$ (cf.~Proposition~\ref{5.8}),  $\frac{\bfd}{\bfd\bft}:\bscal{W}^{q,p}_{\bfvarepsilon}(Q_T)\subseteq \bscal{X}^{q,p}_{\bfvarepsilon}(Q_T)\to \bscal{X}^{q,p}_{\bfvarepsilon}(Q_T)^*$ is densely defined. \\[-3mm]
		
		\textbf{3. Maximal monotonicity:}
		From \eqref{eq:6.9.a}  we obtain for every ${\bsu\in N(\bfgamma_0)}$ that
		\begin{align*}
		\bigg\langle \frac{\textbf{d}\bsu}{\textbf{dt}},\bsu\bigg\rangle_{\bscal{X}^{q,p}_{\bfvarepsilon}(Q_T)}= \frac{1}{2}\|\bsu_c(T)\|_Y^2\ge 0,
		\end{align*}i.e.,  $\frac{\textbf{d}}{\textbf{dt}}:N(\bfgamma_0)\subseteq \bscal{X}^{q,p}_{\bfvarepsilon}(Q_T)\to \bscal{X}^{q,p}_{\bfvarepsilon}(Q_T)^*$ is positive semi-definite, and therefore monotone, as it is also linear. For the maximal monotonicity, assume that $(\bsu,\bsu^*)^\top\in \bscal{X}^{q,p}_{\bfvarepsilon}(Q_T)\times \bscal{X}^{q,p}_{\bfvarepsilon}(Q_T)^*$ satisfies for every $\bsv\in N(\bfgamma_0)$ 
		\begin{align}
		\bigg\langle \bsu^*-\frac{\textbf{d}\bsv}{\textbf{dt}},\bsu-\bsv\bigg\rangle_{\bscal{X}^{q,p}_{\bfvarepsilon}(Q_T)}\ge 0.\label{eq:7.1}
		\end{align}
		Choosing $\bsv=\bfu\varphi\in N(\bfgamma_0)$ for arbitrary  $\bfu\in X_+^{q,p}$ and $\varphi\in C_0^\infty(I)$ in \eqref{eq:7.1}, also using that $\langle \frac{\textbf{d}\bsv}{\textbf{dt}},\bsv\rangle_{\bscal{X}^{q,p}_{\bfvarepsilon}(Q_T)}=\frac{1}{2}\|\bfu\|_Y^2(\varphi(T)-\varphi(0))=0$ due to \eqref{eq:6.9.a}, yields
		\begin{align}
		\langle \bsu^*,\bsu-\bsv\rangle_{\bscal{X}^{q,p}_{\bfvarepsilon}(Q_T)}\ge \bigg\langle\frac{\textbf{d}\bsv}{\textbf{dt}},\bsu\bigg\rangle_{\bscal{X}^{q,p}_{\bfvarepsilon}(Q_T)}.\label{eq:7.2}
		\end{align}
		Choosing now $\bfu=\pm\tau \bfv\in X_+^{q,p}$ for arbitrary $\bfv\in X_+^{q,p}$ and $\tau>0$ in \eqref{eq:7.2},  exploiting that $\frac{\bfd\bsv}{\bfd\bft}=\bscal{J}_{\bfvarepsilon}(\bfu\varphi',\boldsymbol{0})$ in $\bscal{X}^{q,p}_{\bfvarepsilon}(Q_T)^*$ and Remark \ref{3.9}, 	dividing by $\tau$ and passing for $\tau\to \infty$, further provides for every $\bfv\in X_+^{q,p}$ and $\varphi\in C^\infty_0(I)$
		\begin{align*}
		-\int_I{ (\bsu(s),\bfv)_Y\varphi^\prime(s)\,ds}=\int_I{\langle \bsu^*(s),\bfv\rangle_{X^{q,p}(s)}\varphi(s)\,ds},
		\end{align*}
		i.e., $\bsu\in \bscal{W}^{q,p}_{\bfvarepsilon}(Q_T)$ with $\frac{\textbf{d}\bsu}{\textbf{dt}}=\bsu^*$ in $\bscal{X}^{q,p}_{\bfvarepsilon}(Q_T)^*$ due to Proposition \ref{6.6}.
		Finally, let ${(\bfu_T^n)_{n\in\setN}\subseteq
                  X^{q,p}_+}$ be such that $\bfu_T^n\to \bsu_c(T)$ in $Y$ $(n\to\infty)$ and set $(\bsv_n)_{n\in\setN}:=(\bfu_T^n\varphi)_{n\in\setN}\subseteq N(\bfgamma_0)$ for $\varphi\in C^\infty(\overline{I})$ with $\varphi(0)=0$ and $\varphi(T)=1$. Then, we gain for every $n\in \setN$
		\begin{align}
		0\leq \bigg\langle \frac{\textbf{d}\bsu}{\textbf{dt}}-\frac{\textbf{d}\bsv_n}{\textbf{dt}},\bsu-\bsv_n\bigg\rangle_{\bscal{X}^{q,p}_{\bfvarepsilon}(Q_T)}=\frac{1}{2}\|\bsu_c(T)-\bfu_T^n\|_Y^2-\frac{1}{2}\|\bsu_c(0)\|_Y^2.\label{eq:7.3}
		\end{align}
		By passing  in \eqref{eq:7.3} for $n\to\infty$, we
                infer $0\leq -\frac{1}{2}\|\bsu_c(0)\|_Y^2$, i.e.,
                $\bsu_c(0)=\mathbf{0}$ in $Y$~and~${\bsu\in
                  N(\bfgamma_0)}$. In other words,
                $\frac{\textbf{d}}{\textbf{dt}}:N(\bfgamma_0
                )\subseteq \bscal{X}^{q,p}_{\bfvarepsilon}(Q_T)\to \bscal{X}^{q,p}_{\bfvarepsilon}(Q_T)^*$ is maximal monotone.\hfill$\qed$
	\end{proof}
	
	\begin{defn}[$\frac{\mathbf{d}}{\mathbf{dt}}$-pseudo-monotonicity]\label{7.2}An
          operator $\bscal{A}\!:\!\bscal{W}^{q,p}_{\bfvarepsilon}(Q_T)
         \!\subseteq \!\bscal{X}^{q,p}_{\bfvarepsilon}(Q_T)\!\to \!\bscal{X}^{q,p}_{\bfvarepsilon}(Q_T)^ *$ is said to be \textbf{$\frac{\mathbf{d}}{\mathbf{dt}}$-pseudo-monotone}, if for a sequence $(\bsu_n)_{n\in \setN}\subseteq \bscal{W}^{q,p}_{\bfvarepsilon}(Q_T)$, from
		\begin{align}
		\bsu_n\overset{n\to\infty}{\weakto}\bsu\quad\textup{ in }\bscal{W}^{q,p}_{\bfvarepsilon}(Q_T),\label{eq:7.2.a}\\
		\limsup_{n\to\infty}{\langle \bscal{A}\bsu_n,\bsu_n-\bsu\rangle_{\bscal{X}^{q,p}_{\bfvarepsilon}(Q_T)}}\leq 0,\label{eq:7.2.b}
		\end{align}
		it follows $\langle \bscal{A}\bsu,\bsu-\bsv\rangle_{\bscal{X}^{q,p}_{\bfvarepsilon}(Q_T)}\leq \liminf_{n\to\infty}{\langle \bscal{A}\bsu_n,\bsu_n-\bsv\rangle_{\bscal{X}^{q,p}_{\bfvarepsilon}(Q_T)}}$ for every $\bsv\in \bscal{X}^{q,p}_{\bfvarepsilon}(Q_T)$.
	\end{defn}	

	\begin{prop}\label{sum}
		If
                $\bscal{A},\bscal{B}:\bscal{W}^{q,p}_{\bfvarepsilon}(Q_T)
                \subseteq \bscal{X}^{q,p}_{\bfvarepsilon}(Q_T)\to \bscal{X}^{q,p}_{\bfvarepsilon}(Q_T)^*$ are $\frac{\mathbf{d}}{\mathbf{dt}}$-pseudo-monotone, then the sum $\bscal{A}+\bscal{B}:\bscal{W}^{q,p}_{\bfvarepsilon}(Q_T) \subseteq \bscal{X}^{q,p}_{\bfvarepsilon}(Q_T)\to \bscal{X}^{q,p}_{\bfvarepsilon}(Q_T)^*$ is $\frac{\mathbf{d}}{\mathbf{dt}}$-pseudo-monotone.
	\end{prop}
	
	\begin{proof}
          Let $(\bsu_n)_{n\in \setN}\subseteq\bscal{W}^{q,p}_{\bfvarepsilon}(Q_T)$ and
          ${\bsu \in \bscal{W}^{q,p}_{\bfvarepsilon}(Q_T)}$ satisfy
          \eqref{eq:7.2.a} and $\limsup_{n\to\infty}{\langle (\bscal{A}+\bscal{B})\bsu_n,}$
          $\bsu_n-\bsu\rangle_{\bscal{X}^{q,p}_{\bfvarepsilon}(Q_T)}\leq
          0$. Set ${a_n:=\langle \bscal{A}\bsu_n,\bsu_n-\bsu\rangle_{\bscal{X}^{q,p}_{\bfvarepsilon}(Q_T)}}$
          and $b_n:=\langle \bscal{B}\bsu_n,\bsu_n-\bsu\rangle_{\bscal{X}^{q,p}_{\bfvarepsilon}(Q_T)}$
          for $n\in\setN$.  Then, it holds
          $\limsup_{n\to\infty}{a_n}\leq 0$ and
          ${\limsup_{n\to\infty}{b_n}\leq 0}$. In fact, suppose, e.g.,
          that ${\limsup_{n\to\infty}{a_n}=a> 0}$. Then, there exists
          a subsequence, such that $a_{n_k}\to a$ $(k\to\infty)$, and
          thus
          ${\limsup_{k\to\infty}{b_{n_k}}\leq
            \limsup_{k\to\infty}{a_{n_k}+b_{n_k}}-\lim_{k\to\infty}{a_{n_k}}\leq-a<0}$,
          i.e., a contradiction, as then the
          $\frac{\mathbf{d}}{\mathbf{dt}}$-pseudo-monotonicity of
          $\bscal{B}: \bscal{W}^{q,p}_{\bfvarepsilon}(Q_T) \subseteq \bscal{X}^{q,p}_{\bfvarepsilon}(Q_T)\to
          \bscal{X}^{q,p}_{\bfvarepsilon}(Q_T)^*$ implies
          $0\leq \liminf_{\Lambda\ni n\to\infty}{b_n}< 0$. Thus, we
          have $\limsup_{n\to\infty}{a_n}\leq 0$ and
          $\limsup_{n\to\infty}{b_n}\leq 0$, and the
          $\frac{\mathbf{d}}{\mathbf{dt}}$-pseudo-monotonicity of
          $\bscal{A},\bscal{B}:
          \bscal{W}^{q,p}_{\bfvarepsilon}(Q_T) \subseteq \bscal{X}^{q,p}_{\bfvarepsilon}(Q_T)\to
          \bscal{X}^{q,p}_{\bfvarepsilon}(Q_T)^*$ provides
          $\langle
          \bscal{A}\bsu,\bsu-\bsv\rangle_{\bscal{X}^{q,p}_{\bfvarepsilon}(Q_T)}\leq
          \liminf_{n\to \infty}{\langle
            \bscal{A}\bsu_n,}$\linebreak $\bsu_n-\bsv\rangle_{\bscal{X}^{q,p}_{\bfvarepsilon}(Q_T)}$
          and
          $\langle
          \bscal{B}\bsu,\bsu-\bsv\rangle_{\bscal{X}^{q,p}_{\bfvarepsilon}(Q_T)}\leq
          \liminf_{n\to \infty}\langle
            \bscal{B}\bsu_n,$ $\bsu_n-\bsv\rangle_{\bscal{X}^{q,p}_{\bfvarepsilon}(Q_T)}$
          for every $\bsv\in
          \bscal{X}^{q,p}_{\bfvarepsilon}(Q_T)$. Summing these
          inequalities yields the assertion.\hfill$\qed$
	\end{proof}
	
	\begin{thm}\label{7.4}
		Let
                $\bscal{A}:\bscal{W}^{q,p}_{\bfvarepsilon}(Q_T)\subseteq
                \bscal{X}^{q,p}_{\bfvarepsilon}(Q_T)\to
                \bscal{X}^{q,p}_{\bfvarepsilon}(Q_T)^*$ be coercive,
                $\frac{\mathbf{d}}{\mathbf{dt}}$-pseudo-monotone
                operator and assume that there exist a bounded function $\psi:\setR_{\ge 0}\times \setR_{\ge 0}\to \setR_{\ge 0}$ and a constant $\theta\in\left[0,1\right)$, such that for every $\bsu\in \bscal{W}^{q,p}_{\bfvarepsilon}(Q_T)$ it holds
		\begin{align}
		\|\bscal{A}\bsu\|_{\bscal{X}^{q,p}_{\bfvarepsilon}(Q_T)^*}\leq \psi(\|\bsu\|_{\bscal{X}^{q,p}_{\bfvarepsilon}(Q_T)},\|\bsu_c(0)\|_{Y})+\theta \bigg\|\frac{\mathbf{d}\bsu}{\mathbf{dt}}\bigg\|_{\bscal{X}^{q,p}_{\bfvarepsilon}(Q_T)^*}.\label{bbdcon}
		\end{align}
		Then, for arbitrary $\bfu_0\in Y$ and $\bsu^*\in \bscal{X}^{q,p}_{\bfvarepsilon}(Q_T)^*$ there exists a solution $\bsu\in \bscal{W}^{q,p}_{\bfvarepsilon}(Q_T)$ of 
		\begin{alignat*}{2}
		\frac{\mathbf{d}\bsu}{\mathbf{dt}}+\bscal{A}\bsu&=\bsu^*&&\quad\textup{ in }\bscal{X}^{q,p}_{\bfvarepsilon}(Q_T)^*,\\
		\bsu_c(0)&=\bfu_0&&\quad\textup{ in }Y.
		\end{alignat*}
		Here, the initial condition has to be understood in the sense of the unique continuous representation $\bsu_c\in \bscal{Y}^0(Q_T)$ (cf.~Proposition \ref{6.9} (i)).
	\end{thm}
	
	\begin{proof}
          Since $\bscal{W}^{q,p}_{\bfvarepsilon}(Q_T)$ is a reflexive
          Banach space (cf.~Proposition~\ref{6.4}), it is readily seen by means of the standard convergence principle (cf.~\cite[Kap. I, Lemma 5.4]{GGZ74}) that for a sequence
          $(\bsu_n)_{n\in \setN} \subseteq \bscal{W}^{q,p}_{\bfvarepsilon}(Q_T)$ it holds
          $\bsu_n\weakto \bsu$ in $\bscal{W}^{q,p}_{\bfvarepsilon}(Q_T)$ $(n\to \infty)$ if
          and only if $\bsu_n\weakto \bsu$ in $\bscal{X}^{q,p}_{\bfvarepsilon}(Q_T)$ $(n\to \infty)$ and
          $(\frac{\mathbf{d}\bsu_n}{\mathbf{dt}})_{n\in \setN}\subseteq \bscal{X}^{q,p}_{\bfvarepsilon}(Q_T)^*$ is
          bounded. Hence, the notion of pseudo-monotonicity with
          respect to $\frac{\mathbf{d}}{\mathbf{dt}}:D(\frac{\mathbf{d}}{\mathbf{dt}})=\bscal{W}^{q,p}_{\bfvarepsilon}(Q_T)\subseteq
          \bscal{X}^{q,p}_{\bfvarepsilon}(Q_T)\to \bscal{X}^{q,p}_{\bfvarepsilon}(Q_T)^*$ in
          Proposition \ref{2.1a} coincides with our notion of
          $\frac{\mathbf{d}}{\mathbf{dt}}$-pseudo-monotonicity
          in~Definition~\ref{7.2}. Putting everything together, the
          assertion follows from Proposition \ref{2.1a} in conjunction
          with Proposition \ref{7.0} and Proposition \ref{7.1}.~\hfill$\qed$
	\end{proof}
	
	\section{Application}
	\label{sec:8}
	In this final section we apply the developed theory of the preceding section on the model problem \eqref{eq:model}.
	
	Throughout the entire section, let $\Omega\subseteq\setR^d$, $d\ge 2$, be a bounded Lipschitz domain, $I=\left(0,T\right)$, with $0<T<\infty$, and $p\in \mathcal{P}^{\log}(Q_T)$ with $p^->\frac{2d}{d+2}$ and let $q=p^-$, i.e., particularly we have $X_-^{q,p}\embedding Y=L^2(\Omega)^d$.   
	
	\begin{prop}\label{ex1}
		Let $\bfS:Q_T\times \mathbb{M}^{d\times d}_{\sym}\to \mathbb{M}^{d\times d}_{\sym}$  be a mapping satisfying (\hyperlink{S.1}{S.1})--(\hyperlink{S.4}{S.4}). Then, the operator $\bscal{S}:\bscal{X}^{q,p}_{\bfvarepsilon}(Q_T)\to \bscal{X}^{q,p}_{\bfvarepsilon}(Q_T)^*$, for every $\bsu\in \bscal{X}^{q,p}_{\bfvarepsilon}(Q_T)$ given via
		\begin{align*}
			\bscal{S}\bsu:=\bscal{J}_{\bfvarepsilon}(\boldsymbol{0},\bfS(\cdot,\cdot,\bfvarepsilon(\bsu)))\quad\text{ in }\bscal{X}^{q,p}_{\bfvarepsilon}(Q_T)^*,
		\end{align*}
		is well-defined, bounded, continuous, monotone, coercive and  $\frac{\mathbf{d}}{\mathbf{dt}}$-pseudo-monotone.
	\end{prop}

	\begin{proof}
		\textbf{1. Well-definedness:}  Let $\bsu\in \bscal{X}^{q,p}_{\bfvarepsilon}(Q_T)$.
		Since $\bfS:Q_T\times \mathbb{M}^{d\times d}_{\sym}\to 
		\mathbb{M}^{d\times d}_{\sym}$ is a Carath\'eodory mapping (cf.~(\hyperlink{S.1}{S.1})), the function $((t,x)^\top\mapsto\bfS(t,x,\bfvarepsilon(\bsu)(t,x))):Q_T\to \mathbb{M}^{d\times d}_{\sym}$ is Lebesgue measurable. Therefore, we can inspect it for integrability. In doing so, using (\hyperlink{S.2}{S.2}) and  repeatedly $(a+b)^s\!\leq \!2^s(a^s+b^s)$ for all $a,b\ge 0$, $s>0$, we obtain for almost every $(t,x)^\top\in Q_T$
		\begin{align}
		\begin{split}
		\vert\bfS(t,x,\bfvarepsilon(\bsu)(t,x))\vert^{p'(t,x)}&\leq \big[\alpha( \delta+\vert\bfvarepsilon(\bsu)(t,x)\vert)^{p(t,x)-1}+\beta(t,x)\big]^{p'(t,x)}
		\\&\leq 2^{p'(t,x)}\big[\alpha^{p'(t,x)}( \delta+\vert\bfvarepsilon(\bsu)(t,x)\vert)^{p(t,x)}+\beta(t,x)^{p'(t,x)}\big]
		\\&\leq 2^{q'}\big[\alpha^{q'}2^{p(t,x)}\big( \delta^{p(t,x)}+\vert\bfvarepsilon(\bsu)(t,x)\vert^{p(t,x)}\big)+\beta(t,x)^{p'(t,x)}\big]
		\\&\leq 2^{q'}\big[\alpha^{q'}2^{p^+}\big( (1+\delta)^{p^+}+\vert\bfvarepsilon(\bsu)(t,x)\vert^{p(t,x)}\big)+\beta(t,x)^{p'(t,x)}\big].
		\end{split}\label{eq:2.52.1}
		\end{align}
		In consequence, we have $\bfS(\cdot,\cdot,\bfvarepsilon(\bsu))\in L^{p'(\cdot,\cdot)}(Q_T,\mathbb{M}^{d\times d}_{\sym})$. Therefore, we conclude $\bscal{S}\bsu\in \bscal{X}^{q,p}_{\bfvarepsilon}(Q_T)^*$, i.e., $\bscal{S}:\bscal{X}^{q,p}_{\bfvarepsilon}(Q_T)\to\bscal{X}^{q,p}_{\bfvarepsilon}(Q_T)^*$ is well-defined. \\[-3mm]
		
		\textbf{2. Boundedness:} Since $\|\bsv\|_{\bscal{X}^{q,p}_{\bfvarepsilon}(Q_T)}\leq 1$ implies $\rho_{p(\cdot,\cdot)}(\bfvarepsilon(\bsv))\leq 2$ for every $\bsv\in \bscal{X}^{q,p}_{\bfvarepsilon}(Q_T)$, we infer from \eqref{eq:2.52.1} for every $\bsu\in \bscal{X}^{q,p}_{\bfvarepsilon}(Q_T)$ that
		\begin{align*}
		\|\bscal{S}\bsu\|_{\bscal{X}^{q,p}_{\bfvarepsilon}(Q_T)^*}&\leq\sup_{\substack{\bsv\in \bscal{X}^{q,p}_{\bfvarepsilon}(Q_T)\\\|\bsv\|_{\bscal{X}^{q,p}_{\bfvarepsilon}(Q_T)}\leq 1}}{\langle \bscal{S}\bsu,\bsv\rangle_{\bscal{X}^{q,p}_{\bfvarepsilon}(Q_T)}}
		\leq \rho_{p'(\cdot,\cdot)}(\bfS(\cdot,\cdot,\bfvarepsilon(\bsu)))+\rho_{p(\cdot,\cdot)}(\bfvarepsilon(\bsv))\\&
		\leq 2^{q'}\big[\alpha^{q'}2^{p^+}\big((1\!+\!\delta)^{p^+}\vert Q_T\vert+\rho_{p(\cdot,\cdot)}(\bfvarepsilon(\bsu))\big)+\rho_{p'(\cdot,\cdot)}(\beta)\big]+2,
		\end{align*}
		i.e., $\bscal{S}:\bscal{X}^{q,p}_{\bfvarepsilon}(Q_T)\to \bscal{X}^{q,p}_{\bfvarepsilon}(Q_T)^*$ is bounded.\\[-3mm]
		
		\textbf{3. Continuity:} Let $(\bsu_n)_{n\in
                  \setN}\subseteq
                \bscal{X}^{q,p}_{\bfvarepsilon}(Q_T)$ be a sequence,
                such that $\bsu_n\to\bsu$ in
                $\bscal{X}^{q,p}_{\bfvarepsilon}(Q_T)$
                $(n\to\infty)$. In particular, we have
		\begin{align}
		\bfvarepsilon(\bsu_n)\overset{n\to \infty}{\to}	\bfvarepsilon(\bsu)\quad\text{ in }L^{p(\cdot,\cdot)}(Q_T,\mathbb{M}^{d\times d}_{\sym}).\label{eq:2.52.4}
		\end{align}
		From \eqref{eq:2.52.4} we in turn deduce  the existence of a subsequence $(\bsu_n)_{n\in \Lambda}$, with $\Lambda\subseteq \setN$, such that
		\begin{align}
		\bfvarepsilon(\bsu_n)\overset{n\to \infty}{\to}	\bfvarepsilon(\bsu)\quad\text{ in }\mathbb{M}^{d\times d}_{\sym}\quad(n\in \Lambda)\quad\text{ a.e. in }Q_T.\label{eq:2.52.5}
		\end{align}
		Since $\bfS:Q_T\times \mathbb{M}^{d\times d}_{\sym}\to 
		\mathbb{M}^{d\times d}_{\sym}$ is a Carath\'eodory mapping (cf.~(\hyperlink{S.1}{S.1})), \eqref{eq:2.52.5}  leads to 
		\begin{align*}
		\bfS(\cdot,\cdot,\bfvarepsilon(\bsu_n))\overset{n\to \infty}{\to}\bfS(\cdot,\cdot,\bfvarepsilon(\bsu))\quad\text{ in }\mathbb{M}^{d\times d}_{\sym}\quad(n\in \Lambda)\quad\text{ a.e. in }Q_T.
		\end{align*}
		On the other hand, we have by \eqref{eq:2.52.1}, for almost every $(t,x)^\top\in Q_T$ and every $n\in \Lambda$
		\begin{align}
		\begin{split}
		\vert\bfS(t,x,\bfvarepsilon(\bsu_n)(t,x))\vert^{p'(t,x)}\leq 2^{q'}\big[\alpha^{q'}2^{p^+}\big( (1+\delta)^{p^+}+\vert\bfvarepsilon(\bsu_n)(t,x)\vert^{p(t,x)}\big)+\beta(t,x)^{p'(t,x)}\big].
		\end{split}\label{eq:2.52.7}
		\end{align}
		Combining \eqref{eq:2.52.4} and \eqref{eq:2.52.7}, we infer that $(\bfS(\cdot,\cdot,\bfvarepsilon(\bsu_n)))_{n\in\Lambda}\subseteq L^{p'(\cdot,\cdot)}(Q_T,\mathbb{M}^{d\times d}_{\sym})$ is $L^{p'(\cdot,\cdot)}(Q_T)$-uniformly integrable. As a result, Vitali's theorem yields
		\begin{align}
		\bfS(\cdot,\cdot,\bfvarepsilon(\bsu_n))\overset{n\to \infty}{\to}\bfS(\cdot,\cdot,\bfvarepsilon(\bsu))\quad\text{ in }L^{p'(\cdot,\cdot)}(Q_T,\mathbb{M}^{d\times d}_{\sym})\quad(n\in \Lambda).\label{eq:2.52.8}
		\end{align}
		By means of the standard convergence principle
                (cf.~\cite[Kap. I, Lemma 5.4]{GGZ74}), we conclude
                that \eqref{eq:2.52.8} holds true even if
                $\Lambda=\setN$. Since
                $\bscal{J}_{\bfvarepsilon}:L^{q'}(Q_T)^d\times
                L^{p'(\cdot,\cdot)}(Q_T,\mathbb{M}^{d\times
                  d}_{\sym})\to
                \bscal{X}^{q,p}_{\bfvarepsilon}(Q_T)^*$ is continuous
                (cf.~Proposition \ref{4.9}), we conclude that $\bscal{S}\bsu_n=\bscal{J}_{\bfvarepsilon}(\boldsymbol{0},\bfS(\cdot,\cdot,\bfvarepsilon(\bsu_n)))\to\bscal{J}_{\bfvarepsilon}(\boldsymbol{0},\bfS(\cdot,\cdot,\bfvarepsilon(\bsu)))=	\bscal{S}\bsu$ in $\bscal{X}^{q,p}_{\bfvarepsilon}(Q_T)^*$ $(n\to\infty)$, i.e., $\bscal{S}\bsu\in \bscal{X}^{q,p}_{\bfvarepsilon}(Q_T)^*$, i.e., $\bscal{S}:\bscal{X}^{q,p}_{\bfvarepsilon}(Q_T)\to\bscal{X}^{q,p}_{\bfvarepsilon}(Q_T)^*$ is continuous.\\[-3mm]
		
		\textbf{4. Monotonicity:} Follows immediately from condition  (\hyperlink{S.4}{S.4}).\\[-3mm]
		
		\textbf{5. Coercivity:} Using  (\hyperlink{S.3}{S.3}), we obtain for every $\bsu\in \bscal{X}^{q,p}_{\bfvarepsilon}(Q_T)$ and almost every $(t,x)^\top\in Q_T$
		\begin{align}
		\bfS(t,x,\bfvarepsilon(\bsu)(t,x)):\bfvarepsilon(\bsu)(t,x)\ge c_0(\delta+\vert \bfvarepsilon(\bsu)(t,x)\vert)^{p(t,x)-2}\vert \bfvarepsilon(\bsu)(t,x)\vert^2-c_1(t,x).\label{eq:2.52.9}
		\end{align}
		Next, using that $(\delta+a)^{p(t,x)-2}a^2\ge \frac{1}{2}a^{p(t,x)}-\delta^{p(t,x)}$\footnote{Here, we used for $a,\delta\ge 0$ that if $p\ge 2$ and $a\ge 0$, then it holds $(\delta+a)^{p-2}a^2\ge a^p$, and if $1<p<2$, then $(\delta+a)^{p-2}a^2\ge (\delta+a)^{p-2}\big(\frac{1}{2}(\delta +a)^2-\delta^2\big)\ge  \frac{1}{2}(\delta+a)^{p}-\delta^{p}\ge \frac{1}{2}a^{p}-\delta^{p}$, since $(\delta+a)^2\leq 2(\delta^2+a^2)$ and $(\delta+a)^{p-2}\leq \delta^{p-2}$.
		} for all $a\ge 0$ and  $(t,x)^\top\in Q_T$ in \eqref{eq:2.52.9}, we deduce further for every $\bsu\in \bscal{X}^{q,p}_{\bfvarepsilon}(Q_T)$ and almost every $(t,x)^\top\in Q_T$ that
		\begin{align}
		\bfS(t,x,\bfvarepsilon(\bsu)(t,x)):\bfvarepsilon(\bsu)(t,x)\ge \frac{c_0}{2}\vert \bfvarepsilon(\bsu)(t,x)\vert^{p(t,x)}-c_0\delta^{p(t,x)}-c_1(t,x).\label{eq:2.52.10}
		\end{align}
		Therefore, we infer from \eqref{eq:2.52.10} that for every $\bsu\in \bscal{X}^{q,p}_{\bfvarepsilon}(Q_T)$ it holds
		\begin{align*}
		\langle \bscal{S}\bsu,\bsu\rangle_{\bscal{X}^{q,p}_{\bfvarepsilon}(Q_T)}\ge \frac{c_0}{2}\rho_{p(\cdot,\cdot)}( \bfvarepsilon(\bsu))-c_0\rho_{p(\cdot,\cdot)}(\delta)-\|c_1\|_{L^1(Q_T)},
		\end{align*}
		i.e., $\bscal{S}:\bscal{X}^{q,p}_{\bfvarepsilon}(Q_T)\to
                \bscal{X}^{q,p}_{\bfvarepsilon}(Q_T)^*$ coercive, since $q=p^-$.\\[-3mm]
		
		\textbf{6. $\frac{\bfd}{\bfd\bft}$-pseudo-monotonicity:} Since $\bscal{S}:\bscal{X}^{q,p}_{\bfvarepsilon}(Q_T)\to \bscal{X}^{q,p}_{\bfvarepsilon}(Q_T)^*$  is monotone and continuous, it is also pseudo-monotone. In addition, it is readily seen that pseudo-monotonicity implies $\frac{\bfd}{\bfd\bft}$-pseudo-monotonicity.\hfill$\qed$
	\end{proof}

	\begin{prop}\label{ex2}
		Let $\bfb:Q_T\times \setR^d\to \setR^d$  be a mapping that satisfies (\hyperlink{B.1}{B.1})--(\hyperlink{B.3}{B.3}). Then, the operator $\bscal{B}:\bscal{W}^{q,p}_{\bfvarepsilon}(Q_T)\subseteq \bscal{X}^{q,p}_{\bfvarepsilon}(Q_T) \to \bscal{X}^{q,p}_{\bfvarepsilon}(Q_T)^*$, for every $\bsu\in \bscal{W}^{q,p}_{\bfvarepsilon}(Q_T)$ given via
		\begin{align*}
		\bscal{B}\bsu:=\bscal{J}_{\bfvarepsilon}(\mathbf{b}(\cdot,\cdot,\bsu),\boldsymbol{0})\quad\text{ in }\bscal{X}^{q,p}_{\bfvarepsilon}(Q_T)^*,
		\end{align*}
		is well-defined, $\frac{\mathbf{d}}{\mathbf{dt}}$-pseudo-monotone and satisfies the boundedness condition \eqref{bbdcon}.
	\end{prop}
	
		\begin{proof}
		\textbf{1. Well-definedness:} Let $\bsu\in \bscal{W}^{q,p}_{\bfvarepsilon}(Q_T)$.
		Since $\bfb:Q_T\times \setR^d\to
		\setR^d$ is a Carath\'eodory~mapping (cf.~(\hyperlink{B.1}{B.1})), the function $((t,x)^\top\mapsto\bfb(t,x,\bsu(t,x))):Q_T\to \setR^d$ is Lebesgue measurable. In addition, using (\hyperlink{B.2}{B.2}) and that $(a+b)^r\leq 2^r(a^r+b^r)$ for all $a,b\ge 0$,~we~obtain
		\begin{align}
		\begin{split}
		\| \mathbf{b}(\cdot,\cdot,\bsu)\|_{L^{q'}(Q_T)^d}&\leq \gamma\|(1+\vert \bsu\vert)^r\|_{L^{q'}(Q_T)}+\|\eta\|_{L^{q'}(Q_T)^d}\\&=
		\gamma\|1+\vert \bsu\vert\|_{L^{rq'}(Q_T)}^r+\|\eta\|_{L^{q'}(Q_T)^d}\\&\leq 
		\gamma 2^r\big[\vert Q_T\vert^{\frac{1}{q'}}+\|\bsu\|_{L^{rq'}(Q_T)^d}^r\big]+\|\eta\|_{L^{q'}(Q_T)^d}.
		\end{split}\label{eq:2.53.0}
		\end{align}
		Owning to the embedding $\bscal{W}^{q,p}_{\bfvarepsilon}(Q_T)\embedding \bscal{X}^{q,p}_-(Q_T)\cap L^\infty(I,Y)\embedding L^s(Q_T)^d$ for all $s<q_*$ (cf.~\cite[Lemma 3.89]{Ru04}) and since $rq'<q_*$, we conclude from \eqref{eq:2.53.0} that $\mathbf{b}(\cdot,\cdot,\bsu)\in L^{q'}(Q_T)^d$, and therefore $\bscal{B}\bsu\in \bscal{X}^{q,p}_{\bfvarepsilon}(Q_T)^*$, i.e., $\bscal{B}:\bscal{W}^{q,p}_{\bfvarepsilon}(Q_T)\subseteq\bscal{X}^{q,p}_{\bfvarepsilon}(Q_T)\to \bscal{X}^{q,p}_{\bfvarepsilon}(Q_T)^*$ is well-defined.\\[-3mm]
		
		\textbf{2. $\frac{\mathbf{d}}{\mathbf{dt}}$-pseudo-monotonicity:} Let $(\bsu_n)_{n\in \setN}\subseteq \bscal{W}^{q,p}_{\bfvarepsilon}(Q_T)$ be a sequence, such that
		\begin{align}
		\begin{split}
		\bsu_n\overset{n\to\infty}{\weakto}\bsu\quad\text{ in }\bscal{W}^{q,p}_{\bfvarepsilon}(Q_T),\\
		\limsup_{n\to\infty}{\langle \bscal{B}\bsu_n,\bsu_n-\bsu\rangle_{\bscal{X}^{q,p}_{\bfvarepsilon}(Q_T)}}\leq 0.
		\end{split}\label{eq:2.53.1}
		\end{align}
		Owning to the embedding $(\cdot)_c:\bscal{W}^{q,p}_{\bfvarepsilon}(Q_T)\to \bscal{Y}^0(Q_T)$, we infer from \eqref{eq:2.53.1} that
		\begin{align*}
		(\bsu_n)_c(t)\overset{n\to\infty}{\weakto }\bsu_c(t)\quad\text{ in }Y\quad\text{ for all }t\in \overline{I},
		\end{align*}
		i.e., we have $\bsu_n(t)\weakto\bsu(t)$ in $Y$ $(n\to \infty)$ for almost every $t\in I$. Therefore, since $(\bsu_n)_{n\in \setN}$ is also bounded in $\bscal{X}^{q,p}_-(Q_T)\cap L^\infty(I,Y)$, due to $\bscal{W}^{q,p}_{\bfvarepsilon}(Q_T)\embedding \bscal{X}^{q,p}_-(Q_T)\cap L^\infty(I,Y)$, Landes' and Mustonen's compactness principle (cf.~\cite[Proposition 1]{LM87}) yields up to a subsequence
		\begin{align}
			\bsu_n\overset{n\to\infty}{\to}\bsu\quad\text{ in }\setR^d\quad\text{ a.e. in }Q_T.\label{eq:2.53.2.2}
		\end{align}
		In particular, \eqref{eq:2.53.2.2}  due to (\hyperlink{B.1}{B.1}) gives us
		\begin{align*}
		\mathbf{b}(\cdot,\cdot,\bsu_n)\overset{n\to\infty}{\to }\mathbf{b}(\cdot,\cdot,\bsu)\quad\text{ in }\setR^d\quad\text{ a.e. in }Q_T.
		\end{align*}
		Using anew that $\bscal{W}^{q,p}_{\bfvarepsilon}(Q_T)\embedding L^s(Q_T)^d$ for all $s<q_*$ and that $rq'<q_*$, we infer that the sequence  $(\bsu_n)_{n\in \setN}\subseteq L^{rq'}(Q_T)^d$ is uniformly integrable.  Proceeding as for \eqref{eq:2.53.0}, we get for every measurable set $K\subseteq Q_T$ and $n\in \setN$
		\begin{align*}
		\| \mathbf{b}(\cdot,\cdot,\bsu_n)\|_{L^{q'}(K)^d}\leq 
		\gamma 2^r\big[\vert K\vert^{\frac{1}{q'}}+\|\bsu_n\|_{L^{rq'}(K)^d}^r\big]+\|\eta\|_{L^{q'}(K)^d},
		\end{align*}
		i.e., $(\mathbf{b}(\cdot,\cdot,\bsu_n))_{n\in \setN}\subseteq L^{q'}(Q_T)^d$ is uniformly integrable. Hence, Vitali's theorem yields
		\begin{align*}
		\mathbf{b}(\cdot,\cdot,\bsu_n)\overset{n\to \infty}{\to }\mathbf{b}(\cdot,\cdot,\bsu)\quad\text{ in }L^{q'}(Q_T)^d,
		\end{align*}
		which by dint of the continuity of  $\bscal{J}_{\bfvarepsilon}:L^{q'}(Q_T)^d\times L^{p'(\cdot,\cdot)}(Q_T,\mathbb{M}^{d\times d}_{\sym})\to \bscal{X}^{q,p}_{\bfvarepsilon}(Q_T)^*$  provides
		\begin{align}
		\bscal{B}\bsu_n:=\bscal{J}_{\bfvarepsilon}(\mathbf{b}(\cdot,\cdot,\bsu_n),\mathbf{0})\overset{n\to\infty}{\to }\bscal{J}_{\bfvarepsilon}(\mathbf{b}(\cdot,\cdot,\bsu),\mathbf{0})=\bscal{B}\bsu\quad\text{ in }\bscal{X}^{q,p}_{\bfvarepsilon}(Q_T)^*.\label{eq:2.53.2.3}
		\end{align}
		Eventually, from \eqref{eq:2.53.2.3} one easily infers that for every $\bsv\in\bscal{X}^{\circ,p}_{\bfvarepsilon,\mathbf{c}}(Q_T)$
		\begin{align*}
		\langle \bscal{B}\bsu,\bsu-\bsv\rangle_{\bscal{X}^{q,p}_{\bfvarepsilon}(Q_T)}\leq \liminf_{n\to\infty}{\langle \bscal{B}\bsu_n,\bsu_n-\bsv\rangle_{\bscal{X}^{q,p}_{\bfvarepsilon}(Q_T)}}.
		\end{align*}
		
		\textbf{3. Boundedness condition \eqref{bbdcon}:}
		We exploit, see e.g. \cite[Lemma 3.89]{Ru04}, that there exist constants $c_q>0$ and ${\delta_q\in \left(0,2\right]}$, such that for every $\bsu\in \bscal{W}^{q,p}_{\bfvarepsilon}(Q_T)$ it holds
		\begin{align}
			\|\bsu\|_{L^{q_*}(Q_T)^d}^{q_*}\leq c_q\|\bfvarepsilon(\bsu)\|_{L^q(Q_T)^{d\times d}}^q\|\bsu\|_{L^\infty(I,Y)}^{\delta_q}\leq c_q\|\bsu\|_{\bscal{X}^{q,p}_{\bfvarepsilon}(Q_T)}^q\|\bsu\|_{L^\infty(I,Y)}^{\delta_q},\label{eq0}
		\end{align}
		where we used the classical constant exponent Korn's
                inequality in the first estimate.  
	Setting $\sigma:=\frac{\delta_qr}{q_*}$, inserting \eqref{eq0} in  \eqref{eq:2.53.0}, as $rq'<q_*$, and exploiting  the Lipschitz continuity of
		$\bscal{J}_{\boldsymbol{ \varepsilon}}\!:\!L^{q'}(Q_T)^d\times L^{p'(\cdot,\cdot)}(Q_T,\mathbb{M}_{\textup{sym}}^{d\times d})\!\to\! \bscal{X}^{q,p}_{\bfvarepsilon}(Q_T)^*$ (cf.~Prop.~\ref{4.9}), we get for every ${\bsu\in \bscal{W}^{q,p}_{\bfvarepsilon}(Q_T)}$
		\begin{align}
		\begin{split}
		\|\bscal{B}\bsu\|_{\bscal{X}^{q,p}_{\bfvarepsilon}(Q_T)^*}\leq 
		\gamma 2^{r+1} \big[\vert Q_T\vert^{\frac{1}{q'}}+c_q\|\bsu\|_{\bscal{X}^{q,p}_{\bfvarepsilon}(Q_T)}^{\frac{qr}{q_*}}\|\bsu\|_{L^\infty(I,Y)}^{\sigma}\big]+2\|\eta\|_{L^{q'}(Q_T)^d}.
		\end{split}\label{eq1}
		\end{align}
		In addition, using \eqref{eq:6.9.6.1} in \eqref{eq1} and setting $c_\sigma:=\big(\max\big\{T^{-\frac{1}{2}}2(1+\vert Q_T\vert),\sqrt{2}\big\}\big)^{\sigma}$, we further deduce for every $\bsu\in \bscal{W}^{q,p}_{\bfvarepsilon}(Q_T)$, also using  $(a+b)^{\sigma}\leq 2^{\sigma}(a^{\sigma}+b^{\sigma})$ for all $a,b\ge 0$, that 
		\begin{align}
		\|\bsu\|_{\bscal{X}^{q,p}_{\bfvarepsilon}(Q_T)}^{\frac{qr}{q_*}}\|\bsu\|_{L^\infty(I,Y)}^{\sigma}&\leq \|\bsu\|_{\bscal{X}^{q,p}_{\bfvarepsilon}(Q_T)}^{\frac{qr}{q_*}}c_\sigma 2^{\sigma}\bigg[\|\bsu\|_{\bscal{X}^{q,p}_{\bfvarepsilon}(Q_T)}^{\sigma}+\bigg\|\frac{\bfd\bsu}{\bfd\bft}\bigg\|_{\bscal{X}^{q,p}_{\bfvarepsilon}(Q_T)^*}^{\frac{\sigma}{2}}\|\bsu\|_{\bscal{X}^{q,p}_{\bfvarepsilon}(Q_T)}^{\frac{\sigma}{2}}\bigg]
		\notag\\&\leq  c_\sigma 2^{\sigma}(1+\|\bsu\|_{\bscal{X}^{q,p}_{\bfvarepsilon}(Q_T)})^{\frac{qr}{q_*}+\sigma}\bigg[1+\bigg\|\frac{\bfd\bsu}{\bfd\bft}\bigg\|_{\bscal{X}^{q,p}_{\bfvarepsilon}(Q_T)^*}^{\frac{\sigma}{2}}\bigg].\label{eq2}
		\\&\leq  c_\sigma 2^{\sigma}(1+c_{\sigma_0}(\vep))(1+\|\bsu\|_{\bscal{X}^{q,p}_{\bfvarepsilon}(Q_T)})^{(\frac{qr}{q_*}+\sigma)\sigma_0'}+\vep\bigg\|\frac{\bfd\bsu}{\bfd\bft}\bigg\|_{\bscal{X}^{q,p}_{\bfvarepsilon}(Q_T)^*},\notag
		\end{align}
		where we applied in the last inequality the $\vep$-Young inequality with respect to the exponent $\sigma_0=\frac{2}{\sigma}=\frac{2q_*}{\delta_qr}$ and the constant $c_{\sigma_0}(\vep):=(\sigma_0\vep)^{1-\sigma_0'}(\sigma_0')^{-1}$ for all $\vep>0$, which is allowed because $\sigma_0>1$, since  $\sigma<2$ due to $\frac{r}{q*}<\frac{1}{q'}<1$ and $\delta_q\leq 2$.
		Eventually, inserting \eqref{eq2} in \eqref{eq1}, we conclude the boundedness condition \eqref{bbdcon}.\hfill$\qed$
	\end{proof}

	\begin{thm}
		Let $\bfS:Q_T\times \mathbb{M}^{d\times d}_{\sym}\to \mathbb{M}^{d\times d}_{\sym}$  be a mapping  satisfying
	 conditions (\hyperlink{S.1}{S.1})--(\hyperlink{S.4}{S.4}) and
         ${\mathbf{b}:Q_T\times \setR^d\to \setR^d}$  a mapping
         satisfying conditions (\hyperlink{B.1}{B.1})--(\hyperlink{B.3}{B.3}). Then, for arbitrary ${\bfu_0\in Y}$, ${\bsf\in L^{q'}(Q_T)^d}$ and $\bsF\in L^{p'(\cdot,\cdot)}(Q_T,\mathbb{M}_{\sym}^{d\times d})$ there exists a  solution $\bsu\in \bscal{W}_{\bfvarepsilon}^{q,p}(Q_T)$ with a continuous representation $\bsu_c\in \bscal{Y}^0(Q_T)$ of the initial value problem
		\begin{alignat*}{2}
		\frac{\mathbf{d}\bsu}{\mathbf{dt}}+\bscal{S}\bsu+\bscal{B}\bsu&=\bscal{J}_{\bfvarepsilon}(\bsf,\bsF)&&\quad\textup{ in }\bscal{X}_{\bfvarepsilon}^{q,p}(Q_T)^*,\\
		\bsu_c(0)&=\bfu_0&&\quad \textup{ in }Y,
		\end{alignat*}
		or equivalently, for every $\bfphi\in C^\infty(\overline{Q_T})$ with  $\textup{supp}(\bfphi)\subset\subset \left[0,T\right)\times \Omega$ there holds
		\begin{align*}
		-\int_{Q_T}{\bsu(t,x)\cdot\pa_t\bfphi(t,x)\,dtdx}&+\int_{Q_T}{\mathbf{b}(t,x,\bsu(t,x))\cdot\bfphi(t,x)+\bfS(t,x,\bfvarepsilon(\bsu)(t,x)):\bfvarepsilon(\bfphi)(t,x)\,dtdx}\\&=(\bfu_0,\bfphi(0))_Y+\int_{Q_T}{\bsf(t,x)\cdot\bfphi(t,x)+\bsF(t,x):\bfvarepsilon(\bfphi)(t,x)\,dtdx}.
		\end{align*}
	\end{thm}

	\begin{proof}
		According to Proposition \ref{ex1} and Proposition
		\ref{ex2} in conjunction with Proposition \ref{sum}, the sum ${\bscal{S}+\bscal{B}:\bscal{X}^{q,p}_{\bfvarepsilon}(Q_T)\cap L^\infty(I,Y)\to \bscal{X}^{q,p}_{\bfvarepsilon}(Q_T)^*}$ is $\frac{\mathbf{d}}{\mathbf{dt}}$-pseudo-monotone and satisfies the boundedness condition \eqref{bbdcon}.
		In addition, since $\bscal{S}:\bscal{X}^{q,p}_{\bfvarepsilon}(Q_T)\to \bscal{X}^{q,p}_{\bfvarepsilon}(Q_T)^*$ is coercive and the non-coercive part $\bscal{B}:\bscal{X}^{q,p}_{\bfvarepsilon}(Q_T)\cap L^\infty(I,Y)\to \bscal{X}^{q,p}_{\bfvarepsilon}(Q_T)^*$ satisfies  $\langle\bscal{B}\bsu,\bsu\rangle_{\bscal{X}^{q,p}_{\bfvarepsilon}(Q_T)}\ge-\|c_2\|_{L^1(Q_T)}$ for all $\bsu\in \bscal{X}^{q,p}_{\bfvarepsilon}(Q_T)\cap L^\infty(I,Y)$ (cf.~(\hyperlink{B.3}{B.3})), ${\bscal{S}+\bscal{B}:\bscal{X}^{q,p}_{\bfvarepsilon}(Q_T)\cap L^\infty(I,Y)\subseteq \bscal{X}^{q,p}_{\bfvarepsilon}(Q_T)\to \bscal{X}^{q,p}_{\bfvarepsilon}(Q_T)^*}$ is coercive. The assertion follows now by means of Theorem  \ref{7.4}.\hfill$\qed$
	\end{proof}
	
	\bibliographystyle{my-amsplain}
	\bibliography{literatur}

\end{document}